\documentclass[svgnames]{amsart}

\usepackage[utf8]{inputenc}
\usepackage[english]{babel}

\usepackage{amssymb}
\usepackage{amsmath}
\usepackage{amsthm}
\usepackage{mathtools}
\usepackage{mathabx}

\usepackage{enumitem}

\usepackage{comment}

\usepackage{exscale}
\usepackage{latexsym}
\usepackage{graphicx}
\usepackage{tikz}\usetikzlibrary{cd}

\usepackage{svg}

\usepackage{xcolor}
\usepackage{hyperref}
\hypersetup{colorlinks,citecolor=Black,linkcolor=Black,urlcolor=Blue}

\theoremstyle{plain}
\newtheorem{theorem}{Theorem}[section] 
\newtheorem{thmintro}{Theorem}
\newtheorem{corollary}[theorem]{Corollary}
\newtheorem{lemma}[theorem]{Lemma}
\newtheorem{prop}[theorem]{Proposition}
\newtheorem{question}{Question}

\theoremstyle{definition}
\newtheorem{definition}[theorem]{Definition}

\theoremstyle{remark}
\newtheorem{remark}[theorem]{Remark}
\newtheorem{remarkintro}{Remark}

\newcommand{\link}{\text{Lk}}

\newcommand{\N}{\mathcal{N}}
\newcommand{\C}{\mathcal{C}}
\newcommand{\Cp}{\mathcal{C}^1}
\newcommand{\Css}{\mathcal{C}^{ss}}
\newcommand{\ov}{\overline}

\newcommand{\Cdt}{\mathcal{C}/DT_K}

\newcommand{\Cpdt}{\mathcal{C}^1/DT_K}
\newcommand{\Cssdt}{\mathcal{C}^{ss}/DT_K}

\newcommand{\nest}{\sqsubseteq}
\usepackage{mathabx}
\usepackage{dsfont}
\newcommand{\propnest}{\sqsubsetneq}

\newcommand{\orth}{\bot}
\newcommand{\transverse}{\pitchfork}
\newcommand{\frakS}{\mathfrak{S}}

\title{Rigidity of mapping class groups mod powers of twists}

\author{Giorgio Mangioni}
\address{Maxwell Institute and Department of Mathematics, Heriot-Watt University, Edinburgh, UK}
    \email{gm2070@hw.ac.uk}

\author{Alessandro Sisto}
    \address{Maxwell Institute and Department of Mathematics, Heriot-Watt University, Edinburgh, UK}
    \email{a.sisto@hw.ac.uk}

\begin{document}

\begin{abstract}
\setlength\parindent{0pt}
We study quotients of mapping class groups of punctured spheres by suitable large powers of Dehn twists, showing an analogue of Ivanov's theorem for the automorphisms of the corresponding quotients of curve graphs. Then we use this result to prove quasi-isometric rigidity of these quotients, answering a question of Behrstock, Hagen, Martin, and Sisto in the case of punctured spheres.  Finally, we show that the automorphism groups of our quotients of mapping class groups are "small", as are their abstract commensurators. This is again an analogue of a theorem of Ivanov about the automorphism group of the mapping class group.

In the process we develop techniques to extract combinatorial data from a quasi-isometry of a hierarchically hyperbolic space, and use them to give a different proof of a result of Bowditch about quasi-isometric rigidity of pants graphs of punctured spheres.

\textbf{Keywords}: Mapping class group, quasi-isometric rigidity, Ivanov’s Theorem, Dehn twist quotients.

MSC class: 20F65 (primary), 57K20 (secondary)
\end{abstract}

\maketitle
\setcounter{tocdepth}{1}
\tableofcontents

\section{Introduction}

\subsubsection*{Dehn twist quotients} Given a finite-type surface $S$, let $MCG(S)$ be its mapping class group. For every $K\in\mathbb{N}_{>0}$, let $DT_K$ be the subgroup generated by the collection $\{T^K_\alpha\}$, where $\alpha$ varies among all isotopy classes of essential simple closed curves on $S$, and $T_\alpha$ denotes the Dehn Twist around $\alpha$. In this paper we study the properties of quotients of the form $MCG/DT_K$, where $K$ is chosen to be “sufficiently large”.

There are plenty of reasons of interest in these quotients. Firstly, they can be regarded as "Dehn filling quotients" of mapping class groups, as pointed out and explored in \cite{dfdt}, where they are proven to be acylindrically hyperbolic. In fact, these and similar quotients are used in \cite{BHMS} (where they are proven to be hierarchically hyperbolic) to relate questions on finite quotients of mapping class groups to residual finiteness of certain hyperbolic groups; this should be compared with the use of Dehn fillings in the resolution of the virtual Haken conjecture \cite{vhak}.

Moreover, when the surface is closed and hyperbolic, these quotients are conjectured to be the outer automorphism group of the Burnside quotient $\Gamma/\Gamma^n$, where $\Gamma$ is the fundamental group of the surface \cite{CoulonSela}. This fact would be the analogue of the celebrated Dehn-Nielsen-Baer theorem \cite{Dehn,Nielsen}.

These quotients of mapping class groups also appear naturally in a somewhat unrelated area of mathematics, namely the study of topological quantum field theories (see e.g. \cite{TQFT}). As it turns out, if $G$ is a compact Lie group, every homomorphism $MCG(S_g) \to G$ (with $g\ge 3$) factors through the quotient by large enough powers of Dehn twists \cite[Corollary 2.6]{Aramayona_Souto}, and this applies in particular to the so-called quantum representations of mapping class groups.

Yet another reason of interest in these quotients is simply that given a group $G$ and a collection of elements $g_1,\dots,g_n$, it is natural to study the normal closure of the set, and the corresponding quotient. In the case of our quotients of mapping class groups, we are considering a collection of conjugacy representatives of Dehn twists, and we are "stabilising" by taking powers.
\subsubsection*{Rigidity results} We have three main results, analogous to results for mapping class groups, that illustrate three different forms of rigidity of the groups we are considering. In all cases, we consider punctured spheres; we discuss below what would need to be done to cover surfaces with genus. We note that, up to commensurability, we could equivalently regard the quotient groups we consider as quotients of braid groups.

The first result we state is that our quotients are quasi-isometrically rigid, thus answering \cite[Question 3]{BHMS} for the case of punctured spheres. We say that two groups $G$ and $H$ are \emph{weakly commensurable} if there exist two finite normal subgroups $L\unlhd H$ and $M\unlhd G$ such that the quotients $H/L$ and $G/M$ have two finite index subgroups that are isomorphic.

\begin{thmintro}[Quasi-isometric rigidity]\label{qimcgdtn}
    Let $S=S_{0,b}$ be a punctured sphere, with $b\ge 7$ punctures. There exists $K_0\in\mathbb{N}_{>0}$ such that, if $K$ is a non-trivial multiple of $K_0$, then $H=MCG(S)/DT_K$ is quasi-isometrically rigid, meaning that if a finitely generated group $G$ is quasi-isometric to $H$ then $G$ and $H$ are weakly commensurable. 
\end{thmintro}

Quasi-isometric rigidity of mapping class groups was first proven in \cite{BKMM}, see also \cite{Hamenstadt:geometry}, but our argument is closer to the proof given in \cite[Section 5]{quasiflats}. The key result from that paper is \cite[Theorem 5.7]{quasiflats}, which roughly speaking allows one to extract an automorphism of a certain graph from a quasi-isometry of a hierarchically hyperbolic space satisfying suitable assumptions. We cannot use this result as stated, since our quotients do not satisfy its assumptions; instead, we will show that a similar result holds under different hypotheses that our quotient groups do satisfy. This is done in Theorem \ref{hingesauto}, which we believe to be of independent interest since we enlarge the family of hierarchically hyperbolic spaces to which one can apply the arguments from \cite[Section 5]{quasiflats} to study quasi-isometric rigidity.

In our case of interest, the graph from Theorem \ref{hingesauto} can be related to a certain subgraph of the quotient of the curve graph. From there we are able to prove quasi-isometric rigidity of our quotient groups, adapting arguments due to Bowditch, especially from \cite{BowPants}, where the author proved quasi-isometric rigidity of pants graphs. 


Combining quasi-isometric rigidity and results on automorphisms of acylindrically hyperbolic groups from \cite{Commensurating}, we are able to deduce the following, which can be regarded as a form of algebraic rigidity. Here, $MCG^\pm$ denotes the extended mapping class group, where orientation-reversing mapping classes are allowed. 

\begin{thmintro}[Algebraic rigidity]\label{thm:intro2}
Let $S=S_{0,b}$ be a punctured sphere, with $b\ge 7$ punctures. There exists $K_0\in\mathbb{N}_{>0}$ such that, if $K$ is a non-trivial multiple of $K_0$, then
\begin{enumerate}
\item $\text{Aut}(MCG(S)/DT_K)=MCG^\pm(S)/DT_K$;
\item $\text{Out}(MCG(S)/DT_K)\cong \mathbb Z/2\mathbb{Z}$;
\item The abstract commensurator of $MCG^\pm(S)/DT_K$ is isomorphic to its inner automorphism group. In other words, any isomorphism between finite index subgroups of $MCG^\pm(S)/DT_K$ is the restriction of an inner automorphism.
\end{enumerate}
\end{thmintro}

In the case of mapping class groups, the analogue of the result above was proven by Ivanov in his famous paper on automorphisms of curve graphs \cite{Ivanov:autC}. Indeed, our quasi-isometric rigidity and in turn algebraic rigidity results rely on an analogue of Ivanov's theorem for quotients of the \emph{curve graph} $\C(S)$ of $S$, that is, the graph whose vertices are isotopy classes of essential simple closed curves and edges correspond to disjointness (up to isotopy). This is our final main result, and it can be thought of as a form of combinatorial rigidity:

\begin{thmintro}[Combinatorial rigidity]\label{mainthm}
Let $S=S_{0,b}$ be a punctured sphere, with $b\ge 7$ punctures, and let $\C(S)$ be its curve graph. There exists $K_0\in\mathbb{N}_{>0}$ such that, if $K$ is a non-trivial multiple of $K_0$, then the natural map $MCG^\pm(S)/DT_K\to\text{Aut}(\C(S)/DT_K)$ is an isomorphism.
\end{thmintro}

In fact, the result applies to other quotients of mapping class groups of punctured spheres, see Theorem \ref{CombRig} and the discussion in the outline section. As explained below, the main tools used in the proof of Theorem \ref{mainthm} are the lifting techniques first used in \cite{dfdt} and the finite rigid set of Aramayona-Leiniger \cite{AL}. Indeed, we actually obtain a finite rigidity result as well, see Corollary \ref{FinRigquot}. An analogue of the theorem also holds for $b=4$, but in that case the map has finite kernel, see Theorem \ref{mainthmb4}.

\begin{remarkintro}
    We believe that an integer $K_0$ for which Theorems \ref{qimcgdtn}, \ref{thm:intro2}, and \ref{mainthm} hold could in principle be explicitly computed. The starting point for this are explicit constants for various properties of subsurface projections, for example the bounded geodesic image theorem \cite{Webb:BGI}. Then one would need to track all the constants at play in previous papers on the subject, in particular \cite[Subsection 1.2.1, Section 3]{dahmani:rotating} (which in turn depends on the construction from \cite{BBF}), \cite[Proof of Theorem 2.1]{dfdt}, and \cite[Notation 8.11]{BHMS}. It would be interesting to have more direct arguments yielding an explicit (and hopefully low) constant, at least for low numbers of punctures in our setting.
\end{remarkintro}

\subsection{Open questions}

Throughout the paper we only cover the case of punctured spheres/braid groups, so the first questions, which we believe to have positive answer, are the following:

\begin{question}
Does Theorem \ref{mainthm} generalise to arbitrary genus?
\end{question}

\begin{question}
Do Theorems \ref{qimcgdtn} and \ref{thm:intro2} generalise to arbitrary genus?
\end{question}

It will be clear from the outline section that the key result to generalise is combinatorial rigidity, Theorem \ref{mainthm}, which is why we stated the first question separately.

It will be very important for us that the quotients we deal with are hierarchically hyperbolic. Proving hierarchical hyperbolicity of other quotients would be key to proving further quasi-isometric and algebraic rigidity results, and there is a natural class of quotients that contains the quotients we consider in this paper:

\begin{question}
Let $g_1,\dots, g_n$ be elements of a mapping class group $MCG(S)$. Does there exist $K_0\in\mathbb N_{>0}$ such that for all non-trivial multiples $K_i$ of $K_0$ we have that $MCG(S)/\langle\langle\{g_i^{K_i}\}\rangle\rangle$ is hierarchically hyperbolic?
\end{question}
In forthcoming work we will give a positive answer for the case where $S$ is a five-holed sphere. Furthermore, if one allows $S$ to be any finite-type surface, there are certain classes of quotients for which the answer is known to be affirmative, such as those covered in \cite{BHMS}, as discussed above, and \cite{HHS:asdim}, where quotients by powers of pseudo-Anosovs are considered. Additional quotients covered by the question include quotients of mapping class groups by suitable powers of Dehn twists around non-separating curves only, for instance, as well as the quotients considered in \cite{CM}.

\subsection{Outline of proofs}

We will first prove combinatorial rigidity, Theorem \ref{mainthm}. The main idea for doing so is the following. Fix a punctured sphere $S$, say with at least 7 punctures, let $\C$ be its curve graph, and let $K$ be a suitable large integer. While the map $\C\to \C/DT_K$ is not a covering map (as it is quite far from being locally injective), there are still various subgraphs of $\C/DT_K$ that can be lifted. This idea was first used in \cite{dfdt} to show that $\C/DT_K$ is hyperbolic by lifting geodesic triangles, and was developed further in \cite{BHMS} to show that $MCG(S)/DT_K$ is in fact hierarchically hyperbolic. We push these techniques even further to show that the graphs constructed in \cite{AL} can be lifted. These are finite rigid sets for $\C$, that is, all isometric embeddings of them into $\C$ differ by an element of the extended mapping class group. Consider now an automorphism $\ov \phi$ of $\C/DT_K$ and a copy $X$ of one such graph in $\C/DT_K$. One can consider the following diagram, where the hats denote lifts and $\pi$ is the projection map:

$$\begin{tikzcd}
\widehat{X}\ar[r,"g"]\ar{d}{\pi}&\widehat{\phi(X)}\ar{d}{\pi}\\
X\ar[r,"\ov\phi"]&\phi(X)\\
\end{tikzcd}$$

From the diagram we obtain a candidate element of the extended mapping class group that induces $\phi$, namely the element $g$ mapping $\widehat{X}$ to $\widehat{\phi(X)}$. Showing that this candidate is the desired element requires more care, and further lifts, but this is the basic idea. We note that we do not know of a graph that cannot be lifted from $\C/DT_K$ to $\C$, but even a path of length 2 does not have a \emph{unique} $DT_K$-orbit of lifts, see Remark \ref{rmk:lifts_exist?}, while uniqueness of orbits of lifts is crucial in many arguments.

In order to prove quasi-isometric rigidity we will also need combinatorial rigidity for some subgraphs of the quotient graphs. This is shown in Sections \ref{section:1tostrong} and \ref{section:strongtoall}, where we combine arguments due to Bowditch with further lifting techniques.

\begin{remarkintro}
The first step towards generalising our results to the case of surfaces with genus, using the same outline, would be to show that some finite rigid sets admit lifts in those cases as well. The finite rigid sets from \cite{AL} are more complicated for surfaces with genus, but they can hopefully be used for this purpose.
\end{remarkintro}

\subsubsection{Quasi-isometric rigidity} To prove quasi-isometric rigidity, we will use a version of \cite[Theorem 5.7]{quasiflats} which states, roughly, that any quasi-isometry $f$ of a hierarchically hyperbolic group satisfying three assumptions induces an automorphism of a certain graph called the hinge graph. This is a graph that encodes the "standard flats" of the group, as well as their intersection patterns. The three assumptions do not apply to our case, but we show that a similar statement still holds under different assumptions that do apply; this is Theorem \ref{hingesauto}. 
Afterwards, we use combinatorial arguments to show that an automorphism of the hinge graph induces automorphisms of the graphs from Sections \ref{section:1tostrong} and \ref{section:strongtoall}. Then, by combinatorial rigidity of such graphs, these automorphisms are induced by some element $g$ of the mapping class group, and with some more effort one can show that $f$ and $g$ are at bounded distance.


\subsubsection{Algebraic rigidity}
Since automorphisms of a group, and more generally isomorphisms between finite index subgroups, induce quasi-isometries, one can expect to obtain strong algebraic rigidity results from quasi-isometric rigidity. We do just that in Theorem \ref{Out}, using results from \cite{Commensurating}. Finite normal subgroups could cause the outer automorphism group to be finite rather than trivial, so the key technical results we need is that $MCG^\pm(S)/DT_K$ does not contain non-trivial finite normal subgroups, Lemma \ref{nofinite}.

\subsection{Outline of sections}
In Section \ref{sec:setting} we introduce the basic lifting tools that we will use throughout the paper, and we also discuss finite rigid sets. Section \ref{sec:proj} studies the projection maps from curve graphs to their relevant quotients, and a key result here is that our chosen finite rigid sets map injectively, see Theorem \ref{projisometry}. In Section \ref{sec:lifting} we show that (graphs isomorphic to our) finite rigid sets can be lifted from quotients of curve graphs to curve graphs. We then use this in Section \ref{sec:comb_rig} to show combinatorial rigidity, Theorem \ref{mainthm}, see Theorem \ref{CombRig}. Further combinatorial rigidity results, that will be needed in Section \ref{sec:qi}, are proved in Sections \ref{section:1tostrong} and \ref{section:strongtoall}.

The goal of Section \ref{sec:extract} is to show that quasi-isometries of suitable hierarchically hyperbolic spaces induce automorphisms of an associated graph, Theorem \ref{hingesauto}. Then in Section \ref{sec:qi} we combine Theorem \ref{hingesauto} and the combinatorial rigidity results from Sections \ref{section:1tostrong} and \ref{section:strongtoall} to prove Theorem \ref{selfqimcg}, which says that quasi-isometries are all at controlled distance from left-multiplications. Then Theorem \ref{qimcgdtn} follows by combining Theorem \ref{selfqimcg} and the general Lemma \ref{lem:schwarz}. Finally, in Section \ref{sec:alg} we show algebraic rigidity, Theorem \ref{thm:intro2}, see Theorems \ref{Out} and \ref{autmcg}.

 
\subsection*{Acknowledgements}

GM would like to thank Roberto Frigerio for his good early career advice, indirectly leading to this paper. The authors would also like to thank Beatrice Pelloni for her crucial help in making it possible for GM to visit AS. Finally, we thank the referee for the numerous suggestions on how to improve the paper.

\section{Setting and main tools}
\label{sec:setting}

\subsection{Mapping classes and curves}
Let $S_b=S_{0,b}$ be the sphere with $b$ punctures, on which we fix some orientation. By a \emph{curve} on $S_b$ we will always mean the isotopy class of a simple closed curve which is essential, meaning that it does not bound a disk with less than two punctures. Let $\C_b=\mathcal{C}(S_b)$ be the curve graph of $S_b$, that is, the simplicial graph whose vertices are curves and two curves are adjacent if and only they admit disjoint representatives. We will often denote $S_b$ by $S$ and $\C_b$ by $\C$ whenever the dependence on $b$ is not relevant. Given two curves $\alpha,\beta$, let $i(\alpha,\beta)$ denote their \emph{geometric intersection number}, that is, the minimum number of transverse intersection points between a representative for $\alpha$ and a representative for $\beta$.

Let $MCG^\pm(S_b)$ be the extended mapping class group of $S_b$, where we allow orientation-reversing homeomorphisms. We are interested in (powers of) the following type of mapping classes:

\begin{definition}[Dehn Twist]
Consider the Euclidean plane $\mathbb{R}^2$ with polar coordinates $(\theta, r)$, and consider the annulus $A = \{(\theta,r) |1\le r\le 2\}$. Choose the orientation on $A$ such that a counterclockwise rotation is positive. Let $T :\, A \to A$ be the twist map of $A$ given by the formula $T(\theta, r)=(\theta+2\pi r, r)$. We call $T$ the \emph{left twist} of our annulus. See Figure~\ref{fig:dt_image}.

\begin{figure}[htp]
    \centering
    \includegraphics[width=0.85\textwidth]{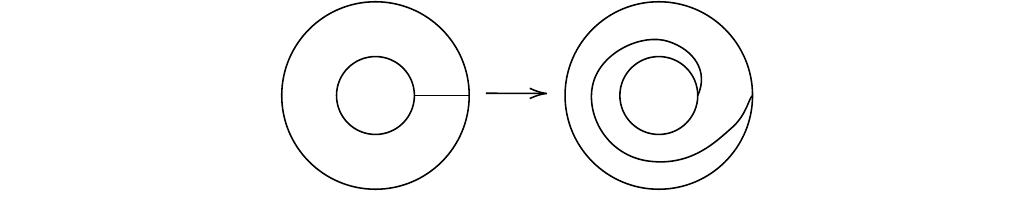}
    \caption{The left twist $T$ of an annulus.}
    \label{fig:dt_image}
\end{figure}

Now let $S$ be an oriented surface and let $\alpha$ be a curve in S. Let $N$ be a regular neighborhood of $\alpha$, which we call an \emph{annulus} with \emph{core curve} $\alpha$, and choose an orientation-preserving homeomorphism $\phi:\,A\to N$. The \emph{(left) Dehn twist} about $\alpha$, denoted $T_\alpha$, is the homeomorphism
$$x\to \begin{cases} \phi \circ T \circ \phi^{-1}(x) \mbox{ if }x \in N,\\
x \mbox{ if }x \in S \setminus N.
\end{cases}$$
The isotopy class of the Dehn twist $T_\alpha$ does not depend on the choice of $N$ and $\phi$, nor on the choice of $\alpha$ within its free isotopy class. Thus we will often abuse notation slightly and view $T_\alpha$ as a mapping class. 
\end{definition}

For every $K\in\mathbb{N}_{>0}$, let $DT_K$ be the subgroup generated by the collection $\{T_\alpha^K\}$, where $\alpha$ varies among all curves on $S$. Furthermore, let $\pi:\,\C\to\Cdt$ be the quotient projection.

\begin{definition}[{\cite{MM}}]
    Let $\gamma$ be a curve, and let $p_\gamma:\,S_\gamma\to S$ be the \emph{cyclic covering} associated to $\gamma$. There is a natural compactification of $S_\gamma$ to a closed annulus $\widehat{S_\gamma}$.
    
    The \emph{annular curve graph} associated to the curve $\gamma$ is the simplicial graph $\C(\gamma)$ whose vertices are the paths connecting the two boundary components of $\widehat{S_\gamma}$, modulo homotopies that fix the endpoints, and where two such paths are adjacent if they have representatives with disjoint interiors (but may share one or both endpoints). 
\end{definition}
If a curve $\alpha$ intersects $\gamma$, we denote by $\pi_\gamma(\alpha)$ the \emph{annular projection} of $\alpha$ onto $\gamma$, that is, a lift of any representative of $\alpha$ to $\widehat{S_\gamma}$, which is an arc connecting the two boundary components. This is well-defined, as any two lifts of (any two representatives of) $\alpha$ are homotopic relative to the endpoints. If $\alpha$ and $\beta$ both intersect $\gamma$, we shall write $d_\gamma(\alpha,\beta)$ as shorthand for $d_{\C(\gamma)}(\pi_\gamma(\alpha),\pi_\gamma(\beta))$.

\subsection{Complexity reduction and lifting properties}
In all our arguments, we will only need one property of $DT_K$ (for suitable $K$), shown in \cite{dfdt} using \cite{dahmani:rotating}, which we now state.

\begin{prop}\label{cor3.6}
There exists $K_0\in\mathbb{N}$ such that for every $\Theta>0$ there exists $K'$ such that for all multiples $K$ of $K_0$ larger than $K'$ the following holds. There exists a well-ordered set $\mathcal O$ and a map $\alpha:DT_K\to\mathcal O$, which we call \emph{complexity}, such that, for all $x\in \C$ and $g\in DT_K-\{1\}$ there exist $s\in \C$ and some $\gamma_s\in DT_K$, which is a power of the Dehn twist around $s$, such that $\alpha(\gamma_sg)<\alpha(g)$ and one of the following holds:
\begin{itemize}
\item $i(x,s)=0$, or
\item $d_s(x,g(x))>\Theta$.
\end{itemize}
\end{prop}

\begin{proof}
This is \cite[Corollary 3.6]{dfdt} (which applies to $DT_K$ in view of \cite[Proposition 5.1]{dfdt}).
\end{proof}

The second case will always be used in conjunction with the Bounded geodesic image Theorem:

\begin{theorem}[\cite{MM}]\label{bgit}
There exists a constant $B$ such that the following holds for all finite-type surfaces. For all vertices $s,x,y\in \C(S)$, if $d_s(x,y)$ is defined and larger than $B$, then any geodesic from $x$ to $y$ intersects the star of $s$.
\end{theorem}

\begin{definition}\label{Ndeep}
    For short, in the statements below we will say that a normal subgroup $\mathcal N\unlhd MCG(S_b)$ is \emph{deep enough} if it satisfies the conclusion of Proposition \ref{cor3.6} for some $\Theta$ depending on the data that has been fixed up to that point. 
\end{definition}
\begin{remark}\label{gendatwist}
Notice that a deep enough subgroup $\N$ is generated by powers of Dehn twists. This is proved by induction on the complexity $\alpha(g)$ of an element $g\in\N$: if $g=1$ we have nothing to prove, otherwise there exists $\gamma_s\in \langle T_s\rangle\cap \N$ for some $s\in\C$ such that $g'=\gamma_s g$ has lower complexity, and hence it is a product of powers of Dehn twists by induction hypothesis. Thus $g=\gamma_s^{-1} g'$ is a product of Dehn twists.
\end{remark}

The key use of the conclusion of Proposition \ref{cor3.6} will be to lift a variety of subgraphs from a quotient of the curve graph to the curve graph. This is also done in \cite{dfdt} as well as generalised in \cite{BHMS}, and in fact we now state a result from the latter paper. First, a definition from \cite{BHMS}. A \emph{generalised} $m$-gon in a graph is a sequence $\tau_0,\dots,\tau_{m-1}$ such that:
     \begin{itemize}
 \item Each $\tau_j$ is either a simplex, together with non-empty sub-simplices $\tau_j^\pm$, or a 
geodesic in $\link(\Delta_j)$ for some (possibly empty) simplex $\Delta_j$ with endpoints $\tau^\pm_j$.
 \item $\tau^+_j=\tau^-_{j+1}$ (indices are taken modulo $m$). 
\end{itemize}
Note that the second bullet implies that $\tau_j\cap\tau_{j+1}$ is non-empty. See Figure~\ref{fig:genmgon} for an example.

\begin{figure}[htp]
    \centering
    \includegraphics[width=0.75\textwidth]{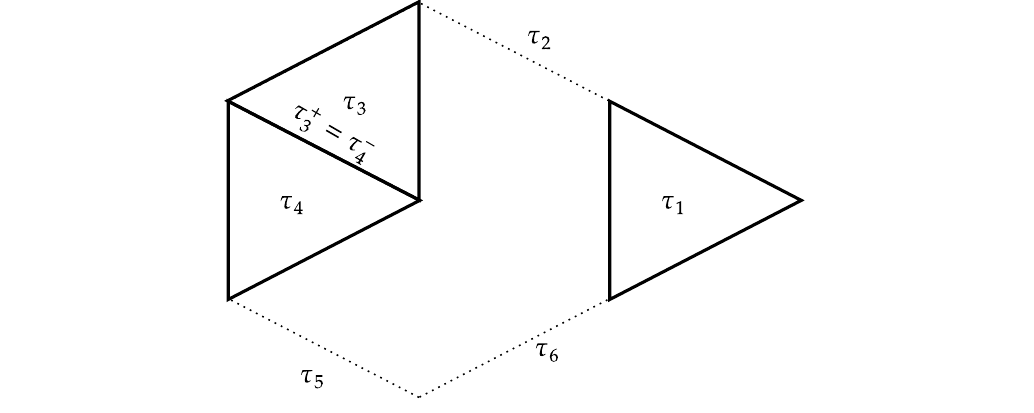}
    \caption{An example of a generalised hexagon. The dotted lines represent geodesics, while the triangles represent simplices.}
    \label{fig:genmgon}
\end{figure}

The following is a small variation on a result from \cite{BHMS}.

\begin{theorem}\label{mgonlift}
For all $m_0\geq 1$ and all deep enough normal subgroups $\N$ the following hold:
\begin{itemize}
    \item For $b\ge4$, any ordered simplex $\overline\Delta\subseteq \C/\N$ admits a unique $\mathcal{N}$-orbit of lifts in $\C$.
    \item For $b\ge4$, given a simplex $\ov \Delta$ in $\C/\N$, a lift $\Delta$ of $\ov \Delta$ in $\C$, and a geodesic $\gamma$ in the link of $\ov\Delta$, we have that $\gamma$ can be lifted to a geodesic in the link of $\Delta$.
    \item For $b\ge5$ and $m\leq m_0$, any generalised $m$-gon in $\C/\N$ can be lifted in $\C$.
\end{itemize}
\end{theorem}

\begin{proof}
    This is essentially \cite[Proposition 8.31]{BHMS}, whose proof only uses \cite[Corollary 3.6]{dfdt} and therefore works for any deep enough subgroup. Said proposition is stated for $m_0=3 \text{Gen}(S)+2p(S)-3$, but the only reason for that particular threshold in \cite{BHMS} is that this is the maximal number of sides of generalised $m$-gons considered in that paper, but the proof works for any fixed $m_0$. Finally, \cite[Proposition 8.31]{BHMS} as stated deals only with the case $b\ge5$, but with similar tools it is easy to prove the uniqueness of the orbit of lifts also for edges and triangles in the Farey complex. This shall be done in Lemmas \ref{edgelift} and \ref{trianglelift}, whose proofs are prototypical of many arguments throughout the paper.
\end{proof}

\begin{lemma}\label{edgelift}
    Whenever $\N$ is deep enough, every edge $\{\ov x, \ov y\}\subset \C_4/\N$ admits a unique $\N$-orbit of lifts.
\end{lemma}

\begin{proof}
    By definition of the quotient graph, there exist two adjacent vertices $x,y\in \C_4$ whose projections are $\ov x$ and $\ov y$, respectively. Therefore there exists at least one lift. Now let $\{x',y'\}$ be another lift of $\{\ov x, \ov y\}$. Up to the action of $\N$ we can assume that $y=y'$. Therefore we have a path $\{x,y,x'\}$ inside $\C_4$, and there is an element $g\in \N$ such that $g(x)=x'$.
    
    If $g$ is the identity we are done, otherwise let $(s,\gamma_s)$ as in Proposition \ref{cor3.6}. If $d_\C(x,s)\le 1$, so that $\gamma_s$ fixes $x$, we can apply $\gamma_s$ to both edges. Now we have a path $\{\gamma_s(x),\gamma_s(y),\gamma_s(x')\}$, and $x=\gamma_s(x)$ is sent to $\gamma_s(x')$ by the element $\gamma_s g$. Since $\alpha(\gamma_s g)<\alpha(g)$ we can proceed by induction on the complexity $\alpha(g)$.

    Otherwise $d_s(x,x')>\Theta$. We claim that $d_\C(y,s)\le 1$. If this is not the case then $\pi_s(y)$ is well-defined, and by triangle inequality for annular projections either $d_s(x,y)>\Theta/2$ or $d_s(x',y)>\Theta/2$. Without loss of generality, we can assume to be in the first case. If we choose $\Theta \ge 2B$, where $B$ is the constant from the bounded geodesic image Theorem \ref{bgit}, we get that every geodesic from $x$ to $y$ (that is, the edge between these vertices) should pass through the star of $s$, which is a contradiction. Therefore $\gamma_s$ must fix $y$, and if we replace $x'$ with $\gamma_s(x')$ we can again reduce the complexity of $g$, while preserving the fact that the two edges share an endpoint.
    
    At the end of the inductive argument we have that $x=x'$, thus proving uniqueness of the orbit of lifts of an edge. 
\end{proof}

\begin{lemma}\label{trianglelift}
Whenever $\N$ is deep enough, every triangle $\ov\Delta\subset \C_4/\N$ admits a unique $\N$-orbit of lifts.
\end{lemma}
\begin{proof}
First we show that $\ov\Delta$ has a lift. Let $\{x,y,z,x'\}$ a lift of the triangle, which we see as a closed path $\{\ov x, \ov y,\ov z,\ov x\}$. Such a lift exists by how edges in the quotient graph are defined, but it is possibly open. Let $g\in \N$ be an element mapping $x$ to $x'$, and suppose that $g$ is not the identity. Then let $(s,\gamma_s)$ as in Proposition \ref{cor3.6}. If $d(x,s)\le 1$ we can apply $\gamma_s$ to the whole path, and proceed by induction on the complexity $\alpha(g)$. 

Otherwise we claim that either $d_{\C}(s,y)\le1$ or $d_{\C}(s,z)\le1$. If none of this happens then $\pi_s(y)$ and $\pi_s(z)$ are both defined, and by triangle inequality we can assume, without loss of generality, that $d_s(x,y)>\Theta/3$. But then again if we choose $\Theta\ge 3B$, where $B$ is the constant from the Bounded Geodesic Image Theorem \ref{bgit}, we have a contradiction. Hence $\gamma_s$ fixes a \emph{cut point} between $y$ and $z$, and we can apply $\gamma_s$ "beyond" this point. Then we get a new lift of the path, whose endpoints are $x$ and $\gamma_s(x')=\gamma_s g(x)$, and again we can proceed by induction.

At the end of the inductive procedure we will have glued $x$ to $x'$, thus we will have a closed lift of the triangle.

Now let $\Delta=\{x,y,z\}$ and $\Delta'=\{x',y',z'\}$ be two lifts inside $\C_4$ of the same triangle $\ov\Delta=\{\ov x,\ov y,\ov z\}$. By Lemma \ref{edgelift} we know that the edge $\ov y,\ov z$ admits a unique $\N$-orbit of lifts, therefore we can assume that $y=y'$ and $z=z'$. Moreover let $g\in\N$ be an element mapping $x$ to $x'$. 

If $g$ is the identity we are done, otherwise let $(s,\gamma_s)$ as in Proposition \ref{cor3.6}. If $d(x,s)\le 1$ then $\gamma_s$ fixes $x$, and we can apply $\gamma_s$ to both triangles and proceed by induction on the complexity $\alpha(g)$. Otherwise $d_s(x,x')>\Theta$, and with the same argument as in Lemma \ref{edgelift} we have that $\gamma_s$ fixes both $y$ and $z$. But then if we replace $\Delta'$ with $\gamma_s(\Delta')$ we can again reduce the complexity of $g$, while preserving the fact that the two triangles share an edge.

At the end of the inductive argument we have that $\Delta=\Delta'$, thus proving uniqueness of the orbit of lifts of a triangle. 
\end{proof}

Later we will need the following refinement to Theorem \ref{mgonlift}, which shows that, when lifting a quadrilateral (that is, a geodesic quadrangle), two opposite sides can be chosen somewhat "rigidly":
\begin{lemma}\label{quadranglelift}
    Let $\N$ be deep enough with respect to some $\Theta>0$, let $\ov Q\subset \C/\N$ be a quadrilateral with vertices $\ov v_1, \ov w_1, \ov v_2, \ov w_2$ and let $Q\subset \C$ be one of its lifts. If the geodesics $[\ov v_1, \ov w_1]$, $ [\ov v_2, \ov w_2]$ have lifts $[v_i, w_i]$ so that $d_s(v_i, w_i)\le \Theta$ whenever the quantity is defined, then the lifts $[v_i', w_i']$ of $[\ov v_i, \ov w_i]$ contained in $Q$ is an $\N$-translate of $[v_i, w_i]$.
\end{lemma}
\begin{proof}
    This is the "moreover" part of \cite[Proposition 4.3]{dfdt}, whose proof only uses \cite[Corollary 3.6]{dfdt}.
\end{proof}

\subsection{Special intersections}
For the following definitions, let $\Gamma$ be either $\C$ or $\C/\N$.
\begin{definition}\label{chain}
A \emph{chain} is a collection of vertices $v_1,\ldots,v_k\in\Gamma$ such that $d_{\Gamma}(v_i,v_j)> 1$ iff $|i-j|=1$. A chain is closed if the same holds with indices mod $k$. 
\end{definition}
In the curve graph, a chain is just a sequence of curves such that two of them intersect iff they are consecutive.\\
The following should be compared with the notion of $\mathcal{X}$-detectable intersection from \cite{AL}:
\begin{definition}[Special intersection]\label{specint}
Let $b\ge5$. Given a facet $P\subseteq\Gamma$ (that is, a simplex of codimension one in a maximal simplex) we say that the vertex $\alpha\in\Gamma$ has \emph{special intersection} with the vertex $\beta\in \Gamma$ with respect to $P$ if:
\begin{itemize}
\item $\alpha$ and $\beta$ both complete $P$ to maximal simplices;
\item there exist $\gamma,\,\delta$ such that $\gamma,\, \alpha,\, \beta,\,\delta$ is a chain;
\item there exists $\varepsilon \in P$ such that $\link(\gamma)\cap P=\link(\delta)\cap P= P-\{\varepsilon\}$.
\end{itemize}
We say that $\gamma$ and $\delta$ are the auxiliary vertices that detect the intersection of $\alpha$ and $\beta$.  The situation is illustrated in Figure~\ref{fig:5catena}.
\end{definition}

\begin{remark}\label{rem:specint2}
If $\alpha$ and $\beta$ are in $\C$ and have special intersection then they have intersection number $2$ and lie in the four-punctured sphere that $P$ cuts out (see the remark following \cite[Lemma 2.5]{AL}).
\end{remark}

\begin{remark}\label{crossspecint}
Definition \ref{specint} actually describes an isometrically embedded generalised pentagon $\mathfrak{P}$, as in Figure \ref{fig:specint}. More precisely, if we define $R$ as the codimension $2$ simplex such that $P=\varepsilon\star R$, then $\gamma,\, \alpha,\, \beta,\,\delta,\,\varepsilon$ form a closed chain in the link of $R$; this becomes an isometrically embedded pentagon when the vertices are ordered as $\alpha,\,\varepsilon,\, \beta,\,\gamma,\,\delta$. Note that these five vertices must play symmetric roles, meaning that when two of them are not adjacent in $\Gamma$ they have special intersection, and the special intersection is detected by the others. For example, $\gamma$ and $\varepsilon$ have special intersection with respect to the facet $R\star\beta$, and $\alpha$ and $\delta$ detect it. Notice that having special intersection is a completely combinatorial property, and hence it is preserved by graph automorphisms.
\end{remark}

\begin{figure}[htp]
    \centering
    \includegraphics[width=0.75\textwidth]{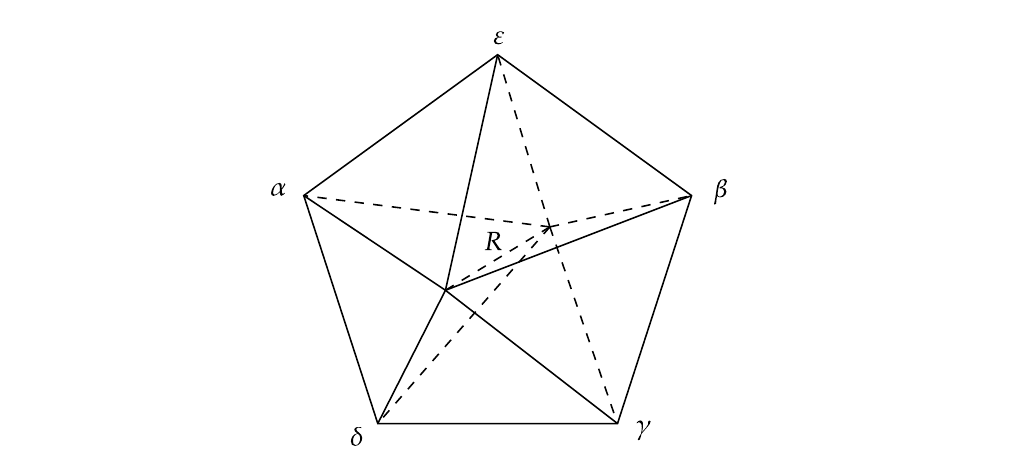}
    \caption{The generalised pentagon $\mathfrak{P}$ described in \ref{specint}, which detects that any two non-adjacent vertices in the link of $R$ have special intersection.}
    \label{fig:specint}
\end{figure}
\begin{definition}\label{specialPent}
We will call an isometrically embedded generalised pentagon $\mathfrak{P}\subset \Gamma$ which is isomorphic to the one in Figure \ref{fig:specint} a \emph{special} pentagon. 
\end{definition}
\begin{figure}[htp]
    \centering
    \includegraphics[width=0.75\textwidth]{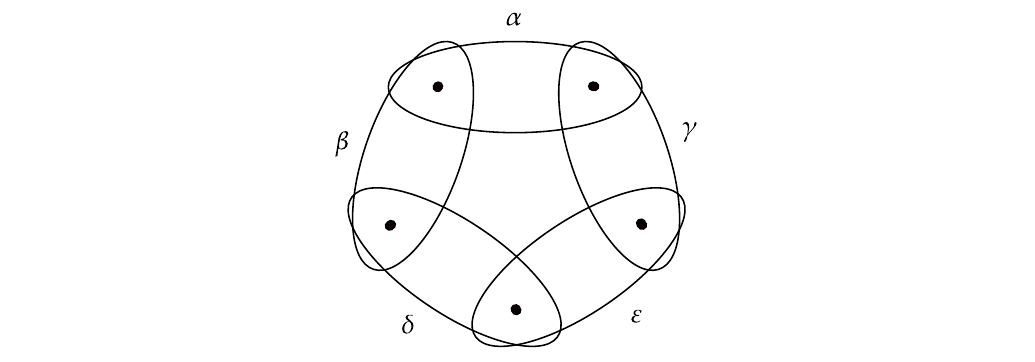}
    \caption{The five curves involved in Definition \ref{specint} for $\Gamma=\C$, forming a chain on the five-holed sphere that $R$ cuts out. Any intersection is special, therefore the intersection number is always $2$.}
    \label{fig:5catena}
\end{figure}

Another consequence of stating special intersection in combinatorial terms is the following lifting property:
\begin{lemma}\label{liftspec}
For every $b\ge5$ and for every $\N$ deep enough, if $\overline\alpha,\, \overline\beta\in \C/\N$ have special intersection then they admit lifts $\alpha,\beta$ with special intersection.
\end{lemma}
\begin{proof}
    By Theorem \ref{mgonlift}, whenever $\N$ is deep enough every special pentagon $\ov {\mathfrak{P}}$ as in Definition \ref{specialPent} admits a lift $\mathfrak{P}$, which remains isometrically embedded since the projection map is $1$-Lipschitz. More precisely, if $\ov x,\ov y\in \ov{\Gamma}$ and $x,y\in \Gamma$ are their lift then
$$d_{\Cp}(x,y)\ge d_{\Cpdt}(\ov x, \ov y)=d_{\ov{\mathfrak{P}}}(\ov x, \ov y)=d_{\mathfrak{P}}(x, y).$$
\end{proof}

For the rest of the paper, given a curve $\beta\in \C$ which bounds a twice-punctured disk, let $H_\beta$ be the \emph{half Dehn twist} around $\beta$, with respect to the given orientation of $S_b$. If $b\ge5$ there is only one twice-punctured disk bounded by $\beta$, while if $b=4$ for each disk $D$ bounded by $\beta$ we will refer to the half Dehn twist of $D$ as $H_D$. The following lemma describes the set of curves with special intersection with $\beta$ with respect to some facet $P$:

\begin{lemma}\label{specintunica}
    Let $b\ge4$ and let $\alpha, \alpha'\in \C$ be two curves which both have special intersection with the same $\beta$, with respect to the same $P$. Let $S_4$ be the four-holed sphere that $P$ cuts out, and let $D\subset S_4$ be any of the two disks bounded by $\beta$ inside $S_4$. Then there exists an integer $k\in\mathbb{Z}$ such that $\alpha'=H_D^k(\alpha)$.
\end{lemma}

\begin{proof}
    $P$ cuts out a sphere $S_4$ with four punctures, and by Remark \ref{rem:specint2} having special intersection implies that the intersection number is $2$. Thus we want to show that, if $\alpha,\alpha'$ are both adjacent to $\beta$ in the Farey complex and $D$ is one of the two disks bounded by $\beta$, then $\alpha'=H_D^k(\alpha)$ for some $k\in \mathbb{Z}$. By properties of the Farey complex there exists a sequence of triangles $T_1,\ldots, T_k$ such that each triangle contains $\beta$, $\alpha\in T_1$, $\alpha'\in T_k$ and every two consecutive triangles $T_i$ and $T_{i-1}$ share and edge containing $\beta$. Thus it suffices to prove that if $\alpha,\alpha',\beta$ are the vertices of a triangle then $\alpha'=H^{\pm1}_D(\alpha)$. This is true, since the two curves $H^{\pm1}_D(\alpha)$ are adjacent to both $\alpha$ and $\beta$, and $\alpha$ and $\beta$ belong to exactly two triangles.
\end{proof}

\begin{remark}\label{rem:hdt_in_s4}
    When $b\ge5$ we have to be a little cautious. If $\beta$ already bounds a twice-punctured disk $D$ inside $S_b$ then this disk embeds inside $S_4$, since it cannot contain any curve of $P$. Hence the map $H_D$, defined on $S_4$, extends to the usual half Dehn Twist along $\beta$, which is defined on the whole $S$. Therefore Lemma \ref{specintunica} shows that, whenever $\alpha,\alpha'$ are two curves with special intersection with $\beta$ with respect to the same $P$, there exists $k\in\mathbb{Z}$ such that $\alpha'=H_\beta^k(\alpha)$.
    
    However, if $\beta$ does not bound a twice-punctured disk inside $S_b$ there might be no way to extend $H_D$ to the whole surface $S$, so it is important to underline that the conclusion holds \emph{inside the four-holed sphere that $P$ cuts out}.
\end{remark}

\subsection{The finite rigid set}
For the rest of the section let $b\ge5$. We denote by $X_b$ the full subgraph of $\C$ defined in \cite[Section 3]{AL}. More precisely, we represent $S_b$ as the double of a regular $b$–gon with vertices removed. An arc connecting non-adjacent sides of the $b$-gon doubles to a curve on $S_b$, and let $X_b$ be the subgraph spanned by such curves, as in Figure \ref{fig:esempioX} for the case $b=8$.

\begin{figure}[htp]
    \centering
    \includegraphics[width=0.95\textwidth]{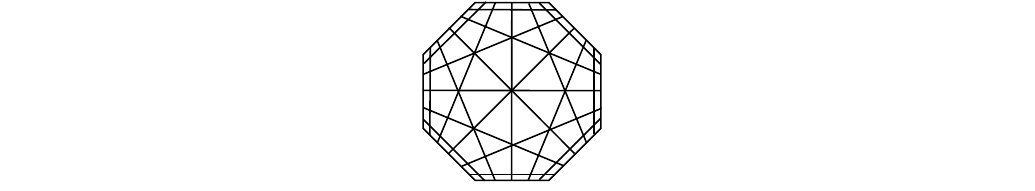}
    \caption{Doubling these arcs gives a copy of $X_8$.}
    \label{fig:esempioX}
\end{figure}
A \emph{copy} of $X_b$ will be the image of an isometric embedding $X_b\to \C$ (respectively $\C/\N$). The reason we are interested in $X_b$ is the following theorem, which is a somewhat refined version of Ivanov's and was proven in \cite[Theorem 3.1]{AL}:
\begin{theorem}
\label{finiterig}
For every $b\ge5$, any locally injective simplicial map $\phi:\, X_b\to \mathcal{C}$ is induced by a mapping class $h\in MCG^\pm$, meaning $\phi=h|_{X_b}$. Moreover, any two such $h$ differ by an element of the pointwise stabiliser $\text{Pstab}(X_b)$, generated by the reflection $r$ that swaps the two copies of the $b$-gon.
\end{theorem}

Another important fact about $X_b$ is that every intersection is special, regardless of the ambient graph $\Gamma$ because it can be detected using only vertices that belong to $X_b$ (this fact should be compared with \cite[Lemma 3.2]{AL}). For example, the special intersection between $\alpha$ and $\beta$ in Figure \ref{fig:stella} is detected by the special pentagon spanned by $\alpha,\beta$, the simplex $R\subset X_b$ of codimension two and the three curves $x,y,z\in X_b$.

\begin{figure}[htp]
    \centering
    \includegraphics[width=0.75\textwidth]{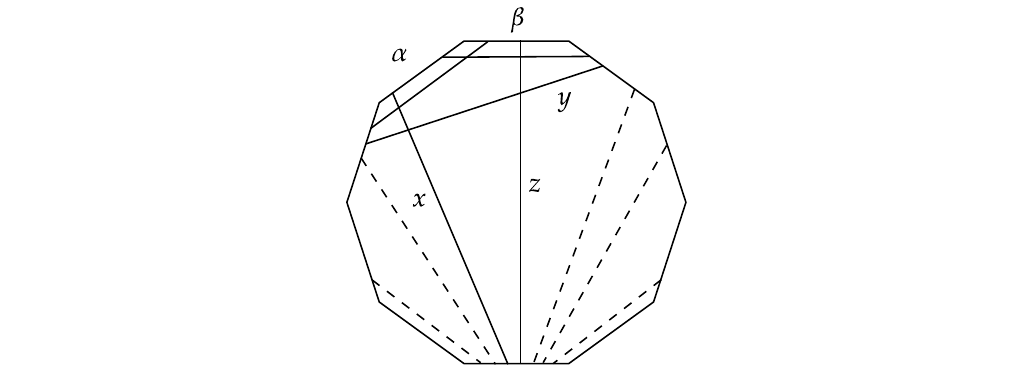}
    \caption{The dashed lines represent the codimension two simplex $R$. Any two intersecting curves of this "star" have special intersection, with respect to some facet that extends $R$.}
    \label{fig:stella}
\end{figure}

The next lemmas show how to construct some copy of $X_b$ starting from a given copy of $X_{b-1}$, and describe how much freedom one has to do so.

\begin{definition}\label{minimalcurve}
    If $b\ge6$, we say that a vertex $\beta$ in a copy of $X_b$ is \emph{minimal} if $\link_{X_b}(\beta)\cong X_{b-1}$.
\end{definition}

\begin{lemma}[Existence of an extension of $X_{b-1}$]\label{completion}
Let $b\ge6$. Every copy of $X_{b-1}$ inside $\C$ which is in the link of a vertex $\beta$ may be completed to a copy of $X_b$ that contains $\beta$. Moreover, if $\alpha$ has special intersection with $\beta$, with respect to some facet $P\subseteq X_{b-1}$, then we can choose $X_b$ to contain $\alpha$.
\end{lemma}

\begin{proof}
First we show that $\beta$ bounds a twice-punctured disk. To see this, notice that, by construction of the graph $X_{b-1}$, if two curves $x,y\in X_{b-1}$ are at distance $1$ then there exists a curve $z\in X_{b-1}$ which is at distance 2 from both. This fact still holds if we replace distances in $X_{b-1}$ with distances in the curve graph, since $X_{b-1}$ is isometrically embedded, and it translates to the fact that whenever two curves in $X_{b-1}$ are disjoint there is a third curve which intersects both. Hence the union of all curves of $X_{b-1}$ must lie on the same connected component of $S_b\setminus \beta$, which is a disk that we call $S'$. But then $S'$ contains $b-4$ pairwise disjoint curves (other than the boundary curve), and therefore must contain at least $b-2$ punctures. Thus the other connected component of $S_b\setminus \beta$ must be a disk with at most $2$ punctures, and since $\beta$ is an essential curve it cannot bound disks with less than two punctures.

Now, since $\beta$ bounds a twice-punctured disk it is easy to find some $X_b'$ which contains $\beta$ as a minimal curve. Let $X_{b-1}'=\link_{X_b'}(\beta)$, and consider the $S_{b-1}$ obtained by shrinking $\beta$ to a puncture $B$. By finite rigidity of $X_{b-1}$ there exists a mapping class $f\in MCG(S_{b-1})$ mapping $X_{b-1}'$ to $X_{b-1}$, which we may choose to be the identity in a neighbourhood of $B$, up to rotations of $S_{b-1}$ and isotopy. Therefore $f$ extends to a mapping class $F\in MCG(S_b)$ (for example, by setting $F$ to be the identity in a neighbourhood of the twice-punctured disk surrounded by $\beta$), and now $F(X_b')$ is a copy of $X_b$ that completes $\beta\star X_{b-1}$.

For the "moreover" part, choose a copy of $X_b$ that completes $\beta\star X_{b-1}$ and let $\alpha'$ be the curve in this copy that corresponds to $\alpha$ (i.e., the curve that has special intersection with $\beta$ with respect to the same $P$). By Lemma \ref{specintunica} and the discussion in Remark \ref{rem:hdt_in_s4} there is a suitable power $H_\beta^k$ of the half twist around $\beta$ that maps $\alpha$ to $\alpha'$. Since $H_\beta$ fixes $\beta\star X_{b-1}$ we may apply this mapping class to obtain the desired copy of $X_b$.
\end{proof}

\begin{lemma}[Uniqueness of the extension]\label{alfaiunica}
Let $b\ge 6$. In any copy of $X_b$ inside $\C$, let $X_{b-1}$ be the link of some minimal vertex $\beta$ and let $\alpha \in X_{b}$ be a curve intersecting $\beta$. Then for every other $z\in X_b$ that intersects $\beta$ there is a facet $P\subset X_{b-1}$ such that $z$ is the unique curve in $\C$ that belongs to $\link (P)\cap \link(\alpha)$. In other words, a copy of $X_{b}$ is uniquely determined by the star of a minimal vertex $\beta$ and one of the curves that intersect $\beta$.
\end{lemma}
\begin{proof}
Let $R\subset X_{b-1}$ be a simplex of codimension two such that $\alpha,z\in\link(R)$, and complete it to a codimension one simplex $P$ by adding some curve $x\in \link(\alpha_i)\cap X_{b-1}$ that intersects $\alpha$. For example, $x$ and $R$ can be chosen as in Figure \ref{fig:stella}. We need to show that if some curve $z'$ lies in $\link (P)\cap \link(\alpha)$ then it is unique. First notice that $z'$ must be an essential curve of the five-punctured sphere $S_5$ that $R$ cuts out (that is, it cannot surround only one puncture of $S_5$, or it would coincide with some curve in $R$). Moreover, we know that $\alpha$ and $x$ are both essential curves of $S_5$ and they have intersection number $2$, so we are in the situation depicted in Figure \ref{fig:alfaiunica}. Now, $\alpha$ and $x$ together cut $S_5$ in three once-punctured disk and a twice-punctured disk. Since $z'$ is essential in $S_5$ it must be the boundary of the twice-punctured disk, and thus it is unique.
\end{proof}

\begin{figure}[htp]
    \centering
    \includegraphics[width=0.75\textwidth]{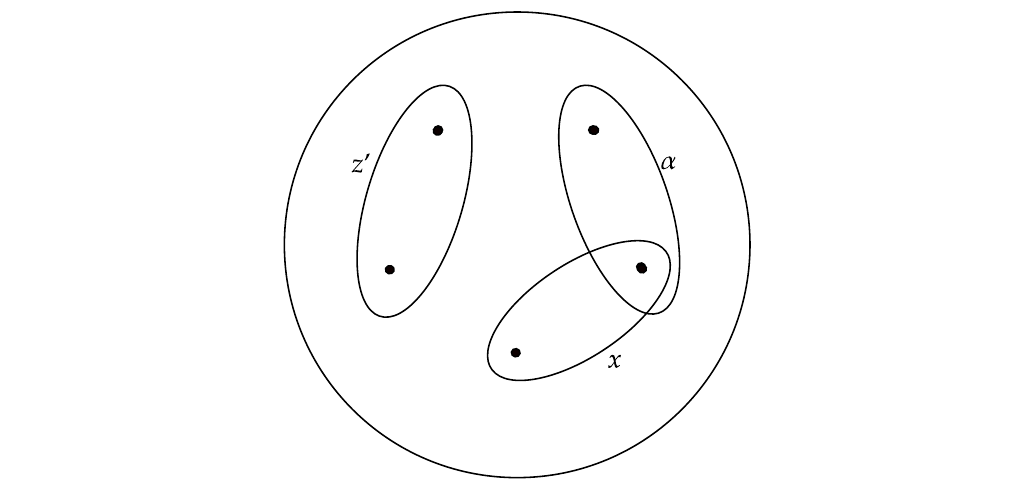}
    \caption{The five-punctured sphere that $R$ cuts out.}
    \label{fig:alfaiunica}
\end{figure}

\begin{remark}
When $b=5$ we cannot define minimal curves as in Definition \ref{minimalcurve}. Thus, with a slight abuse of notation, we will say that every curve in $X_5$ is minimal, because it bounds a twice-punctured disk. Moreover Lemma \ref{completion} is false: $\link_{X_5}(\beta)$ is made of two points, but if we take two vertices $\gamma,\gamma'\in\link_{\C_5}(\beta)$ with intersection number greater than $2$ we cannot complete them to a copy of $X_5$. However, we can run the proof of Lemma \ref{alfaiunica}, with $R=\emptyset$, to get the following:
\end{remark}

\begin{lemma}\label{5alfaiunica}
For $b=5$, if two copies $X_5, X_5'\subset \C_5$ share four curves then they must coincide.
\end{lemma}

\section{Projection lemmas}\label{sec:proj}
In this section we gather various results about the nature of the projection map $\pi:\C\to\C/\mathcal{N}$. We start by showing that the restriction of $\pi$ to the curve graph of a subsurface is what one would expect.
\begin{lemma}\label{Cu0}
Let $U$ be a subsurface of $S$. Let $\N$ be a deep enough subgroup, and let $\N(U)$ be the subgroup of $\N$ generated by the elements with support in $U$. If $x,x'$ are curves on $U$ and there exists an element $g\in \N$ mapping $x$ to $x'$ then we may find another element $h\in \N(U)$ mapping $x$ to $x'$. 
\end{lemma}
\begin{proof}
    Let $\Delta$ be a pants decomposition of the complement of $U$, including its boundary curves. We proceed by induction on the complexity of $g$. If $g$ is the identity then we are done; otherwise let $(s,\gamma_s)$ be as in Proposition \ref{cor3.6}.\\
    If $d_{\C}(x,s)\le1$ then we can apply $\gamma_s$ to both curves, in order to reduce the complexity of $g$. Then by inductive hypothesis we can find some $h\in \N(\gamma_s(U))$ mapping $\gamma_s(x)$ to $\gamma_s(x')$, and therefore $\gamma_s^{-1} h\gamma_s$ maps $x$ to $x'$ and belongs to $\N(U)$ by properties of Dehn twists under conjugation.\\
    Otherwise $d_s(x,x')>\Theta$. Arguing exactly as in the proof of Lemma \ref{edgelift} we get that $\gamma_s$ fixes $\Delta$ pointwise whenever $\N$ is deep enough. Now recall that, by construction, $\gamma_s\in \langle T^{N}_s\rangle$, where $T_s$ is the Dehn twist around $s$. Then $s$ must be disjoint from all curves in $\Delta$, which means that either $s\in \Delta$ (which is impossible, since otherwise $\gamma_s$ would fix $x$) or $s\subseteq U$. Therefore $\gamma_s\in \N(U)$, and if we apply it to $x'$ we can conclude by induction.
\end{proof}
\begin{corollary}\label{CUdt}
    Let $U$ be a connected subsurface of $S$ of complexity at least $2$, and let $\N$ be a deep enough subgroup. If $\pi:\,\C(S)\to\C(S)/\N$ is the quotient projection then $\pi(\C(U))\cong \C(U)/\N(U)$.
\end{corollary}
\begin{proof}
    By Lemma \ref{Cu0} these graphs have the same vertices, and since $U$ has complexity at least $2$ the adjacency relation is having disjoint lifts in both cases. 
\end{proof}

\subsection{Isometric projections}
We move on to show that, whenever $\N$ is deep enough, the projection is an isometry on a variety of subgraphs of the curve graph. All the result of this subsection rely on the following refinement of \cite[Lemma 4.4]{dfdt}, which roughly speaking says that the projection $\pi:\C\to\C/\mathcal{N}$ preserves directions with small projections, in the link of some fixed simplex:
\begin{lemma}\label{projinlink}
Fix $\Theta>0$. Let $\N$ be deep enough and let $\overline\Delta\subseteq \C/\N$ be a simplex and $\Delta\subseteq X$ be one of its lifts. Suppose that $x, y \in \link(\Delta)$ project to $\overline{x}, \overline{y}\in \link(\overline\Delta)$ and have the property that $d_s(x, y) <\Theta$ for every $s\in\C$ for which the quantity is defined. Then $\pi|_{[x,y]}$ is an isometric embedding into $\link(\overline\Delta)$, for any geodesic $[x, y]\subseteq \link(\Delta)$.
\end{lemma}

\begin{proof}
The proof is very similar to that of \cite[Lemma 4.4]{dfdt}. Suppose by contradiction that there is a shorter path from $\overline{x}$ to $\overline{y}$ inside $\link(\overline\Delta)$. Lift this shorter path as a geodesic segment $[y, x']\subseteq \link(\Delta)$, which can be done by Lemma \ref{mgonlift}. There exists $\gamma \in \N$ such that $\gamma x = x'$. We proceed by induction on the complexity of $\gamma$, to prove that $d_{\link(\Delta)}(y, x') = d_{\link(\Delta)}(y, x)$ (thus contradicting that $d_{\link(\overline\Delta)}(\overline x,\overline y) < d_{\link(\Delta)}(x, y)$). If $\gamma= 1$ we are done, otherwise let $(s, \gamma_s)$ be as in \ref{cor3.6}. If $d_{\C}(x,s)\le1$ then $\gamma_s x = x$, and we can apply $\gamma_s$ to both geodesics and to $\Delta$ and conclude by the induction hypothesis. Otherwise $d_s(x, \gamma x) > \Theta$, and arguing as in Lemma \ref{edgelift} we see that $\gamma_s$ must fix $\Delta$ pointwise. Moreover, either $d_\C(y,s)\le1$ or $d_s(y,x')\ge d_s(x,x')-d_s(x,y)$ must be large, since $d_s(x, y)$ is assumed to be small. In both cases there must be some $s' \in [y, x']$ in the star of $s$. Thus, one can change the lift of $[y, x]$ as $[y, \gamma_s x']$, while keeping it an isometric lift inside the link of $\Delta$. One concludes by induction hypothesis, which applies to $\gamma_s\gamma$.
\end{proof}

The previous lemma is particularly useful in the following form:
\begin{corollary}\label{projV}
    For every finite set of vertices $V\subseteq \C$ there exists $\Theta$ such that, whenever $\N$ is deep enough with respect to $\Theta$, the projection is an isometry on $V$.
\end{corollary}

\begin{proof}
    Since $V$ is finite, we only need to show that for every $x,z\in \C$ there is some constant $M(x,z)$ such that $\sup_{s\in \C} d_s(x,z)<M(x,z)$, because then we can choose $\Theta> M=\max_{x,z\in V} M(x,z)$ to ensure that the hypothesis of Lemma \ref{projinlink} applies. One way to see this is to complete $x,z$ to complete clean markings $\mu,\nu$, in the sense of \cite[Section 2.5]{MM}. By the Distance formula \cite[Theorem 6.12]{MM} there exists a constant $M'(S)$ such that the sum $\sum_{s\in \C,\,d_s(\mu,\nu) > M'} d_s(\mu,\nu)$ is bounded above in terms of the distance between $\mu$ and $\nu$ in the marking graph. In particular the sum is finite, and since every term is greater than a constant there must be a finite number of terms.
    Moreover, for every $s\in \C$ for which the quantity $d_s(x,y)$ is defined (in particular, not a base curve of $\mu$ or $\nu$), let $x'\in \mu$ and $y'\in \nu$ be curves which realise $d_s(\mu,\nu)$. Then by triangle inequality
    $$d_s(x,y)\le d_s(x,x')+d_s(x',y')+d_s(y',y)=d_s(x,x')+d_s(\mu,\nu)+d_s(y',y)$$
    and since annular projections are $2$-Lipschitz (see \cite[Lemma 2.3]{MM}) we have that
    $$d_s(x,y)\le d_s(\mu,\nu)+4$$
    Thus it suffices to choose $M(x,z)> \max_{s\in \C,\,d_s(\mu,\nu) > M'} d_s(\mu,\nu)+4$.
\end{proof}

Now we want to apply Corollary \ref{projV} to the rigid set.
\begin{definition}\label{Y}
Given a copy of $X_b\subset \C$ for $b\ge5$, set 
$$Y_b=X_b\bigcup_{\beta\in X_b\,minimal} H_\beta^{\pm1}(X_b).$$
\end{definition}

\begin{lemma}\label{Mbound}
For every $b\ge5$ there exists $M\in\mathbb{R}^+$ with the following property. Given a copy $X_b\subset \C$ set $Y_b$ as in Definition \ref{Y}. Then for every $x,z\in Y_b$ and every $s\in \C$ we have that $ d_s(x,z)\le M$.
\end{lemma}

\begin{proof}
Fix a copy $X_b$ and let $M=\max_{x,z\in Y_b} M(x,z)$, where $M(x,z)$ is defined as in Corollary \ref{projV}. We are left to prove that, if $X_b'$ is another copy of the rigid set and $Y_b'$ is the corresponding union of half twists, then $M$ works also for $Y_b'$. Take an extended mapping class $f$ that maps $X_b$ to $X_b'$, which exists by Theorem \ref{finiterig}. In particular $f$ maps minimal curves to minimal curves. Therefore 
$$Y_b'=X_b'\bigcup_{\beta'\in X_b'\,minimal} H_{\beta'}^{\pm1}(X_b')=f(X_b)\bigcup_{\beta\in X_b\,minimal} H_{f(\beta)}^{\pm1}\circ f(X_b)=$$
$$=f(X_b)\bigcup_{\beta\in X_b\,minimal} f\circ H_{\beta}^{\pm1}(X_b)=f(Y_b)$$
Now the conclusion follows, since for every $x,z\in Y_b'$ and $s\in \C$ we have that $$d_s(x,z)=d_{f^{-1}(s)}(f^{-1}(x),f^{-1}(z))\le M.$$
\end{proof}

This corollary is crucial for us, so we state it as a theorem:
\begin{theorem}\label{projisometry}
For every $b\ge5$ the following holds whenever $\N$ is deep enough. Let $X_b\subset \C$ be a copy of the rigid set and let $Y_b$ be as in Definition \ref{Y}. Then the restriction $\pi|_{Y_b}$ is an isometry. In particular there exists an isometrically embedded copy of $X_b$ inside $\C/\N$.
\end{theorem}
\begin{proof}
Just combine Lemma \ref{Mbound} and Corollary \ref{projV}.
\end{proof}

There are other two consequences of Lemma \ref{projinlink} that will be quite useful later. The first one deals with the cardinality of the quotient:
\begin{corollary}\label{infinite}
Whenever $\N$ is deep enough, the quotient $\C/\N$ is infinite.
\end{corollary}
\begin{proof}
For every curve $x\in \C$ and every pseudo-Anosov mapping class $g\in MCG(S)$ there exists $\Theta>0$ such that $\sup_{s\in\C,n\in\mathbb{Z}}d_s(x, g^n(x))<\Theta$. This follows from an argument in the proof of \cite[Theorem 2.1]{dfdt} which just uses the Bounded geodesic image Theorem \ref{bgit} and the fact that for every $x,z\in\C$ then $\sup_{s\in\C}d_s(x,z)<\infty$, which we proved in Corollary \ref{projV}. Hence, if $\N$ is deep enough with respect to $\Theta$, by Lemma \ref{projinlink} the projection is an isometry on the orbit $\{g^n(x)\}_{n\in\mathbb{Z}}$, which is infinite since $g$ is pseudo-Anosov.
\end{proof}
The second consequence deals with filling curves:
\begin{corollary}\label{projfilling}
Fix $\Theta>0$. Let $\N$ be deep enough and let $\overline\Delta\subseteq \C/\N$ be a simplex and $\Delta\subseteq X$ be one of its lift. Suppose that $\Delta$ cuts out a single subsurface $\Sigma$ of complexity at least $1$, and let $x,y\subset\Sigma$ be a pair of filling curves with the property that $d_s(x, y) <\Theta$ whenever the quantity is defined. If $x,y$ project to $\ov x, \ov y\in \link (\ov\Delta)$ then any two lifts $x',y'\in \link(\Delta)$ again fill $\Sigma$.
\end{corollary}

\begin{proof}
Recall that $x$ and $y$ fill $\Sigma$ if and only if $d_{\link(\Delta)}(x,y)\ge3$. If $x$ and $y$ satisfy the hypothesis of Lemma \ref{projinlink} then the distance between $\ov x$ and $\ov y$ in $\link(\ov\Delta)$ remains at least three. But then every pair of lifts must be at distance at least three, since the projection map is $1$-Lipschitz.
\end{proof}

\subsection{Puncture separations}
Similarly to the previous subsection, we aim to establish some results of isometric projection, but this time they will arise from some observations on how curves separate punctures on a sphere. Choose an enumeration of the punctures of $S_b$. A curve $\alpha$ on $S_b$ is separating, therefore it splits the punctures into two sets $\mathcal B^+(\alpha)$ and $\mathcal B^-(\alpha)$.

\begin{definition}
The \emph{puncture separation} induced by $\alpha$ is the unordered couple $\{\mathcal B^+(\alpha),\,\mathcal B^-(\alpha)\}$.
\end{definition}

\begin{remark}\label{puntsepdiverse}
Notice that a Dehn twist fixes each puncture, therefore preserves puncture separations. In fact if two punctures are on the same side of $\alpha$ they may be joined by an arc which is disjoint from $\alpha$, and the image of this arc via a Dehn twist will again join the same punctures while being disjoint from the image of $\alpha$. In other words, any two curves that induce different puncture separations cannot be identified in a quotient of the curve graph by a subgroup generated by powers of Dehn twists, such as all deep enough subgroups $\N$ in the light of Remark \ref{gendatwist}.
\end{remark}

\begin{definition}
Two curves $\alpha$ and $\alpha'$ induce \emph{nested} puncture separations if $\mathcal B^+(\alpha)\subseteq\mathcal B^\pm(\alpha')$. In other words, the splitting induced by $\alpha\cup\alpha'$ refines the splitting induced by $\alpha$.
\end{definition}

Notice that two disjoint curves induce nested puncture separations, but the converse is not true (for example, choose some $\beta$ intersecting $\alpha$ and set $\alpha'=T_\beta(\alpha)$, which induces the same puncture separation by Remark \ref{puntsepdiverse}). The following lemma shows a peculiar behaviour of the projection $\C\to \Cdt$ which holds for all $K\in\mathbb{N}$:
\begin{lemma}\label{projpuntsep}
For every subgroup $H\le MCG(S)$ generated by powers of Dehn Twists, the following hold for the projection $\C\to\C/H$:
\begin{enumerate}
    \item For $b=4$, if $\alpha$ and $\alpha'$ are adjacent in the Farey complex then their projections remain at distance $1$.
    \item For $b\ge5$, if $\alpha$ and $\alpha'$ are disjoint curves then their projections remain at distance $1$.
    \item For $b\ge5$, if $\alpha$ and $\alpha'$ intersect and their puncture separations are not nested then their projections remain at distance at least $2$.
\end{enumerate}
\end{lemma}

\begin{proof} In all cases we will use that the projection is $1$-Lipschitz, since the action of $H$ over $\C$ is simplicial.
\begin{enumerate}
    \item \label{i1} Two adjacent curves in the Farey complex induce different puncture separations, therefore their projections must be at distance $1$ since they cannot coincide.
    \item \label{i2} If $\alpha$ and $\alpha'$ are disjoint then they must induce different puncture separations, otherwise they would bound an unpunctured annulus and therefore they would be isotopic. Then their projections can only be at distance $1$, for the same reasons as before.
    \item \label{i3} Suppose that $\alpha$ and $\alpha'$ intersect and induce non-nested puncture separations. Firstly, $\alpha$ and $\alpha'$ cannot have the same projection, since they induce different puncture separations. Furthermore, if by contradiction $\pi(\alpha)$ and $\pi(\alpha')$ are the endpoints of an edge then we may lift it, since the action is simplicial. In other words there exists some $g\in H$ such that $\alpha$ is disjoint from $g(\alpha')$. But then $\alpha$ and $g(\alpha')$ must induce nested puncture separations, and so must do $\alpha$ and $\alpha'$ since Dehn twists preserve puncture separations by Remark \ref{puntsepdiverse}.
\end{enumerate}
\end{proof}
As an immediate consequence of Items (\ref{i1}) and (\ref{i2}) we get:
\begin{corollary}\label{projsimplex}
    For every $b\ge4$, the projection map is an isometry on every simplex $\Delta\subset\C$. 
\end{corollary}

\begin{remark} Item (\ref{i1}) directly implies Theorem \ref{projisometry} for $b=4$. Moreover, Items (\ref{i2}) and (\ref{i3}) prove Theorem \ref{projisometry} for $b\ge5$, since $Y_b$ has diameter $2$ and by construction each pair of intersecting curves induce non-nested puncture separations. This alternative proof works for every choice of $K$, though only if the surface is a punctured sphere.\end{remark}

\section{Lifting the rigid set}\label{sec:lifting}
The aim of this section is to prove the following lifting result:
\begin{theorem}\label{lifting}
For every $b\ge 5$ and every $\N$ deep enough, every copy of $X_b$ in $\C/\N$ has a lift inside $\C$, and any two lifts are conjugated by an element of $\mathcal{N}$.
\end{theorem}

\begin{remark}
\label{rmk:lifts_exist?}
At the moment, it is not known to the authors whether there exists a finite graph $\ov G$ such that, whenever $\ov G$ is immersed in some $\C/DT_K$, then $\ov G$ does not admit a lift in $\C$. However there are easy examples where a graph $\ov G$ admits two lifts $G,G'$ which belong to different $DT_K$-orbits; for applications it is often crucial for us that there is only one orbit of lifts of specific graphs. For example, for every $K>1$ let $\ov l=\{\ov\alpha,\ov\beta,\ov\gamma\}$ be a geodesic segment of length $2$ in $\C_4/DT_K$, and let $l=\{\alpha,\beta,\gamma\}$ be one of its lifts (whose existence is granted by Lemma \ref{mgonlift}).  Now let $\gamma'=T^{K}_\beta(\gamma)$, so that $l'=\{\alpha,\beta,\gamma'\}$ is another isometric lift of $\ov l$. If by contradiction $g\in DT_K\setminus\{1\}$ maps $l$ to $l'$ then it must fix $\alpha$ and $\beta$. Notice moreover that, if an orientation preserving mapping class fixes an edge of the Farey complex then it fixes the whole complex. Thus $g$ is a hyperelliptic involution, and since it must permute the punctures it cannot belong to any group generated by powers of Dehn twists, such as $DT_K$.
\end{remark}

\subsection{Existence of a lift}
\begin{proof}[Proof of existence] We proceed by induction, using the case $b=5$ as the base case since $X_5$ is a pentagon, which lifts by Theorem \ref{mgonlift}. Assume that $b\ge6$ and that every copy of $X_{b-1}$ has a lift inside $\mathcal{C}_{b-1}$, and let $\ov{X}_b$ be a copy of $X_b$ inside $\C_b/\N$. Let $\overline\beta\in \ov{X}_b$ be a minimal vertex, and let $\ov{X}_{b-1}:=\link_{\ov{X}_b}(\overline\beta)$.
In order to apply the inductive hypothesis we need to show that $\link_{\C_b}(\beta)$ is isomorphic to $\C_{b-1}$ and that its projection is isomorphic to $\C_{b-1}/\N(S_{b-1})$. We start with the first assertion.
\begin{lemma}
Any lift $\beta$ of a minimal vertex $\overline\beta\in\ov{X}_{b}$ bounds a twice-punctured disk. In other words $\link_{\C_b}(\beta)\cong\mathcal{C}_{b-1}$.
\end{lemma}
\begin{proof}
Let $\overline P$ be a facet inside $\ov{X}_{b-1}$. Since $\ov \beta\star\ov P$ is a simplex we can find a lift $\beta'\star P$ by Lemma \ref{mgonlift}, and up to applying some elements of $\N$ we can assume that $\beta'=\beta$ and therefore $P\in \link_{\C_b}(\beta)$. We want to show that all curves in $\link_{\C_b}(\beta)$ lie in the same subsurface of $S\setminus \beta$, or in other words that, given any two curves $\gamma,\delta\in \link_{\C_b}(\beta)$, there is a chain $\gamma=\eta_0, \eta_1, \ldots,\eta_k=\delta\in \link_{\C_b}(\beta)$. Since every curve in $\link_{C_b}(\beta)$ either coincides with or intersects a curve in $P$, which has maximal dimension in $\link_{\C_b}(\beta)$, it is enough to show that for every $\gamma,\,\delta\in P$ there is a curve $\eta$ that intersects both. Now, since $\ov{X}_{b-1}$ is isometrically embedded inside $\C_b/\N$ we may pick some $\overline\eta \in\ov{X}_{b-1}$ at distance $2$ from both $\ov \gamma$ and $\ov \delta$, and let $\eta$ be one of its lifts in $\link_{\C_b}(\beta)$. Since the projection map is $1$-Lipschitz, $\eta$ must be at distance at least $2$ from (hence intersect) both $\gamma$ and $\delta$, and we are done.
\end{proof}
Now, Corollary \ref{CUdt} shows that the projection of $\link_{\C_b}(\beta)$ is isomorphic to $\C_{b-1}/\N(S_{b-1})$. Therefore we can use the inductive hypothesis to lift $\ov{X}_{b-1}$ to some $X_{b-1}$ inside $\link_{\C_b}(\beta)$. Now the key point of the following construction will be this uniqueness lemma:
\begin{lemma}[Unique lift for "very" special intersections]\label{uniqueliftxspec}
For $b\ge5$ the following holds whenever $\N$ is deep enough. Let $x,x'\in \C$ be two curves with special intersection with the same $y$, with respect to the same facet $P$, and let $S$ be any graph in the link of $y$. Suppose that $x,\,y,\,S$ can be completed to a copy of $X_b$. Then if $x$ and $x'$ project to the same element in the quotient there is some $g\in \mathcal{N}$ which maps $x$ to $x'$ and fixes $y$ and $S$.
\end{lemma}

\begin{proof}
Since both $x$ and $x'$ have special intersection with $y$, with respect to the same facet $P$, by Lemma \ref{specintunica} there exists $k\in \mathbb{Z}$ such that, in the four-holed sphere $S_4$ that $P$ cuts out, $x'=H_D(x)$, where $D$ is one of the disks bounded by $y$. Now we claim that $k=2m$ must be even, and therefore $H^{2m}_D=T_y^m$ extends to the whole surface $S$. To see this first notice that since $x$ and $x'$ are in the same $\N$-orbit they must induce the same puncture separation by Remark \ref{puntsepdiverse}. Moreover $x$ and $y$ intersect twice, thus they separate the punctures of $S$ into four sets $\mathcal{A}^\pm,\mathcal{B}^\pm$, each of which corresponds to one of the punctures of $S_4$. Suppose that $x$ induces the separation $\{\mathcal{A}^+\cup\mathcal{B}^+,\mathcal{A}^-\cup\mathcal{B}^-\}$ while $y$ induces the separation $\{\mathcal{A}^+\cup\mathcal{A}^-,\mathcal{B}^+\cup\mathcal{B}^-\}$, and that $D$ contains the punctures corresponding to $\mathcal{A}^\pm$. Then $H_D$ swaps the two punctures corresponding to $\mathcal{A}^\pm$, which shows that $H_D^k(x)$ induces the same puncture separation as $x$ if and only if $k$ is even. 

Now let $\gamma\in \mathcal{N}$ such that $x'=\gamma(x)$. If $\gamma=1$ then we are done. Otherwise let $(s,\gamma_s)$ be as in Proposition \ref{cor3.6}. If $d_{\C}(x,s)\le1$ then we may apply $\gamma_s$ to the whole data and proceed by induction on the complexity of $\gamma$. Otherwise $d_s(x,x')\ge \Theta$. Now we claim that any other $z\in S\cup\{y\}$ is in the star of $s$, so that we can replace $x'$ with $\gamma_s(x')$ and proceed by induction. If this is not the case then $\pi_s(z)$ is well-defined, and we have that
$$d_s(x,x')\le d_s(x,z)+d_s(x',z)=d_s(x,z)+d_s(T^m_y(x),T^m_y(z))=d_s(x,z)+d_{T_y^{-m}(s)}(x,z)$$
where we used that $z$ belongs to the star of $y$ and is therefore fixed by $T_y$. But now 
$$\Theta\le d_s(x,x')\le d_s(x,z)+d_{T_y^{-m}(s)}(x,z)\le 2\max_{p,q\in X_{b},\, s\in\C} d_s(p,q)\le 2M $$
where $M$ is the constant from Lemma \ref{Mbound}, which does not depend on $\Theta$. This is a contradiction if $\N$ is deep enough.
\end{proof}

Now, in order to complete our lift of $\ov X_b$ we still need to lift the vertices $\ov{\alpha}_1,\ldots,\overline\alpha_{b-3}\in \ov X_b$ outside the link of $\overline{\beta}$. For every $i=1,\ldots, b-3$ we can find vertices $\ov x, \ov y\in \ov X_{b-1}$ and a codimension two simplex $\ov R\subset \ov X_{b-1}$ such that $\ov{\beta},\, \ov{\alpha}_1,\, \ov{\alpha}_i,\, \ov x,\, \ov y$ and $\ov R$ span a special pentagon $\ov{\mathfrak{P}}$. For example, one could look at Figure \ref{fig:stella}, replace $z$ with $\alpha_i$ and then consider the corresponding vertices inside $\ov X_{b}$. 

Lift $\ov{\mathfrak{P}}$ to a special pentagon $\mathfrak{P}'$, let $Q'=\beta'\star y'\star R'$ be the lift of the simplex $\overline Q=\overline \beta\star \overline y\star \overline R$ inside $\mathfrak{P}'$, and let $Q=\beta\star y\star R$ be the corresponding simplex inside $\beta\star X_{b-1}$. Since there is a unique $\mathcal{N}$-orbit of lifts of $\ov Q$ there exists $g\in \mathcal{N}$ mapping $Q'$ to $Q$. Thus we can apply $g$ to $\mathfrak{P}'$ and assume that $Q=Q'$. Moreover, let $x'\in \mathfrak{P}'$ be the lift of $\ov x$, and $x$ the corresponding lift inside $X_{b-1}$. If we apply Lemma \ref{uniqueliftxspec} to $x,x',y$ and $S=\beta\star R$ we find some element $g\in \mathcal{N}$ that maps $x'$ to $x$ and fixes $Q$. Thus we can apply $g$ to $\mathfrak{P}'$ and assume that also $x=x'$.\\
Now, let $G_i=\mathfrak{P}'\cup X_{b-1}$. Let $\alpha_1^i\in \mathfrak{P}'$ be the lift of $\ov{\alpha}_1$ inside the special pentagon. Lemma \ref{completion} provides a copy $X_b^i$ which extends $\beta\star X_{b-1}$ and contains $\alpha_1^i$. By Lemma \ref{alfaiunica}, $X_b^i$ also contains $\alpha_i$, therefore $G_i\subset X_b^i$. Now, consider the copies $X_b^2$ and $X_b^j$, for $3\le j\le b-3$. These copies coincide on $\beta\star X_{b-1}$, and by construction $\alpha_1^2$ and $\alpha_1^j$ project to the same element $\ov{\alpha}_1$. Therefore, if we apply Lemma \ref{uniqueliftxspec} with $x=\alpha_1^2$, $x'=\alpha_1^j$, $y=\beta$ and $S=X_{b-1}$, we find an element $g\in \mathcal{N}$ that maps $\alpha_1^2$ to $\alpha_1^j$ and fixes $\beta\star X_{b-1}$. Then we conclude that $g(X_{b}^2)=X_b^j$ by Lemma \ref{alfaiunica}, which implies that the projection of $X_b^2$ coincides with $\ov{X}_b$ also on $\overline{\alpha_j}$. Since we may do this for every $3\le j\le b-3$ we see that $X_b^2$ is the lift we were looking for.
\end{proof}

\subsection{Uniqueness of the lift}\label{section:uniquelift}
\begin{proof}[Proof of uniqueness]
Again we proceed by induction. First we discuss the base case $b=5$.
\begin{lemma}
If $\N$ is deep enough, any two lifts of $\ov X_5$ inside $\C_5$ differ by an element of $\mathcal{N}$.
\end{lemma}

\begin{proof}
The proof is just a sequence of iterations of Lemma \ref{uniqueliftxspec}. More precisely, let $\{\overline{\gamma}_i\}_{i=1,\ldots,5}$, be the vertices of $\ov X_5$, and let $\gamma_i$ and $\gamma_i'$ be lifts that form two pentagons $X_5$ and $X_5'$. Up to applying some element of $\mathcal{N}$ to $X_5'$ we may assume that $\gamma_1=\gamma_1'$ and $\gamma_2=\gamma_2'$, because the edge $\ov\gamma_1,\ov\gamma_2$, which is a simplex of dimension $1$, admits a unique orbit of lifts by Lemma \ref{mgonlift}. Now, if we apply Lemma \ref{uniqueliftxspec} with $x=\gamma_3$, $x'=\gamma_3'$, $y=\gamma_1$ and $S=P=\gamma_2$ we find some element $g\in\N$ mapping $\gamma_3$ to $\gamma_3'$ and fixing $\gamma_1,\gamma_2$, so we can apply $g$ to $X_5$ and assume that $\gamma_3=\gamma_3'$. Moreover, if we repeat the argument with $x=\gamma_4$, $x'=\gamma_4'$, $y=\gamma_2$, $P=\gamma_3$ and $S=\{\gamma_1,\,\gamma_3\}$ we may also assume that $\gamma_4=\gamma_4'$. Now $X_5$ and $X_5'$ coincide on four curves, and therefore they must coincide because of Lemma \ref{5alfaiunica}.
\end{proof}

Now assume that $b\ge6$ and every copy of $X_{b-1}$ inside $\C_{b-1}/\N(S_{b-1})$ has a unique lift inside $\mathcal{C}_{b-1}$. The next step is the following:
\begin{lemma} If $\overline\beta\in\ov{X}_b$ is a minimal vertex, then $\overline G=\overline\beta\star\ov{X}_{b-1}$ has a unique lift inside $\C_b$, up to elements in $\mathcal{N}$.
\end{lemma}
\begin{proof}
Let $G=\beta\star X_{b-1}$ and $G'=\beta'\star X_{b-1}'$ be two lifts of $\overline G$. Up to an element of $\mathcal{N}$ we may assume that $\beta=\beta'$. Now $X_{b-1}$ and $X_{b-1}'$ are in $\link_{C_b}(\beta)\cong\mathcal{C}_{b-1}$, where again we see $S_{b-1}$ as the surface obtained by shrinking $\beta$ to a puncture. Moreover, as already noticed the projection of $\link_{\C_b}(\beta)$ is isomorphic to $\C_{b-1}/\N(S_{b-1})$, thus by induction there is an element $h\in \mathcal{N}(S_{b-1})$ that maps $X_{b-1}$ to $X_{b-1}'$. If we extend $h$ to be the identity on the twice-punctured disk bounded by $\beta$ we get a mapping class $H\in \mathcal{N}(S_b)$ that maps $G$ to $G'$, as required.
\end{proof}

Now we are finally able to prove the uniqueness part of Theorem \ref{lifting}. Choose two lifts $X_b$ and $X_b'$ of $\ov{X}_b$. By the previous lemma we may assume that they coincide on $\beta\star X_{b-1}$. Now apply Lemma \ref{uniqueliftxspec} with $x=\alpha_1$, $x'=\alpha_1'$, $y=\beta$ and $S=X_{b-1}$, so that we may assume that $X_b$ and $X_b'$ coincide also on $\alpha_1$. Finally, by Lemma \ref{alfaiunica} we see that the two lifts now coincide.
\end{proof}

\section{Combinatorial rigidity}\label{sec:comb_rig}
The aim of this section is to prove Theorem \ref{mainthm} (combinatorial rigidity).

\subsection{The four-punctured case}
Firstly, we show an analogue of Theorem \ref{mainthm} for the special case $b=4$. The proof will highlight the core of the general case, though the latter will require some more machinery. The subgraph playing the role of the rigid set will be any triangle $T$ inside the Farey complex, since the following well-known result holds (for example, it follows from the discussions in \cite[Section 3]{AL} and in \cite[Section 3.4]{FarbMargalit}):
\begin{theorem}\label{Trig}
    Given two triangles $T,T'\subset\C_4$ there exists an extended mapping class $h\in MCG^\pm(S_{0,4})$ mapping $T$ to $T$. Any two such $h$ differ by an element of the Klein four-group $\mathcal{K}$, generated by the classes of the two hyperelliptic involutions in Figure \ref{fig:Klein}.
\end{theorem}

\begin{figure}[htp]
    \centering
    \includegraphics[width=0.75\textwidth]{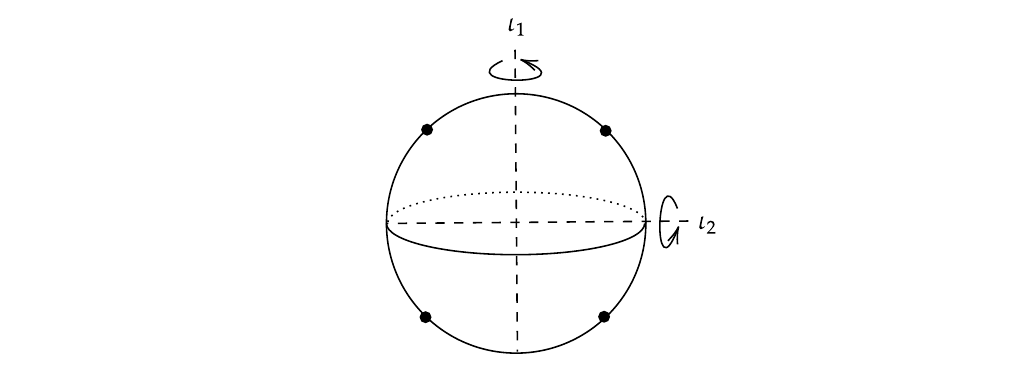}
    \caption{The hyperelliptic involutions $\iota_1,\iota_2$ generate a Klein four-group $\mathcal{K}$ which is the kernel of the action of $MCG^\pm$ over the curve graph.}
    \label{fig:Klein}
\end{figure}

Now we want to prove the following:
\begin{theorem}\label{mainthmb4}
There exists $K_0\in\mathbb{N}_{>0}$ such that if $K$ is a non-trivial multiple of $K_0$ the action $MCG^\pm(S_{0,4})/DT_K\to\text{Aut}(\C(S_{0,4})/DT_K)$ is an epimorphism, with kernel  the projection of the Klein four-group $\mathcal{K}$, generated by the classes of the two hyperelliptic involutions in Figure \ref{fig:Klein}.
\end{theorem}

We split the proof in two propositions below.

\begin{prop}\label{surjn4}
For any normal subgroup $\N$ deep enough, any automorphism $\overline\phi\in\text{Aut}(\mathcal{C}_4/\mathcal{N})$ is induced by an element of $MCG^\pm(S_4)$.
\end{prop}

The proof relies on the following:
\begin{lemma}\label{unicotriang} The following facts hold:
\begin{enumerate}
    \item \label{item1} Any two adjacent curves $\alpha,\,\beta\in \mathcal{C}_4$ belong to exactly two triangles.
    \item \label{item2} Any two adjacent vertices $\overline\alpha,\,\overline\beta\in \mathcal{C}_4/\mathcal{N}$ belong to at most two triangles (namely, the projections of the two triangles that complete some of their lifts).
\end{enumerate}
\end{lemma}

\begin{proof}
Item (\ref{item1}) is true by construction of the Farey complex. Regarding Item (\ref{item2}) fix a pair $\alpha,\,\beta$ of adjacent lifts of $\overline\alpha,\,\overline\beta$. Any triangle $\overline{T}=\{\overline{\alpha},\,\overline\beta,\,\overline\gamma\}$ has a lift $\{\alpha',\,\beta',\,\gamma'\}$, and by uniqueness of the orbit of lifts of an edge (which is a particular case of Theorem \ref{mgonlift}) we may assume that $\alpha=\alpha'$ and $\beta=\beta'$. Therefore $\overline T$ is the projection of one of the two triangles containing $\alpha,\,\beta$.
\end{proof}

\begin{proof}[Proof of Proposition \ref{surjn4}]
Let $T\subseteq\mathcal{C}_4$ be a triangle, and let $\overline{T}$ be its projection, which is still an isometrically embedded triangle by Item (\ref{i1}) of Lemma \ref{projpuntsep}. Let $\widetilde{T}$ be a lift of $\overline{\phi}(\overline T)$, which exists by Theorem \ref{mgonlift}. By Theorem \ref{Trig} there is a mapping class $h\in MCG^\pm(S_4)$ mapping $T$ to $\widetilde{T}$, which means that $\overline{\phi}$ coincides with the induced map $\overline{h}$ on $\overline{T}$. In other words, we showed that the following diagram commutes on $T\subset\mathcal{C}_4$:
\begin{equation}\begin{tikzcd}\label{comm4}
    \mathcal{C}_4\ar{d}{\pi}\ar{r}{h}&\mathcal{C}_4\ar{d}{\pi}\\
    \mathcal{C}_4/\mathcal{N}\ar{r}{\overline{\phi}}&\mathcal{C}_4/\mathcal{N}
\end{tikzcd}\end{equation}
Our next goal is to show that, if the diagram commutes on some triangle $T$, then it commutes on any triangle $T'$ that shares an edge with $T$. If we prove the claim we are done, since any triangle in the Farey complex maps to a fixed triangle via a finite sequence of reflections along sides. Let $\alpha,\,\beta\in T$ be the vertices of the common edge. Notice that $h$, which is a graph automorphism, must map $T'$ to the triangle $\widetilde{T}$ that contains $h(\alpha)$ and $h(\beta)$ and is not $\widetilde{T}$, and similarly $\overline\phi$ must map $\pi(T')$ to $\pi(\widetilde{T}')$ since there are no other triangles containing $\pi(\alpha)$ and $\pi(\beta)$. Therefore Diagram (\ref{comm4}) commutes also on $T'$, as required.
\end{proof}

\begin{prop}\label{injn4}
For $b=4$ and any $\N$ deep enough, if $g\in MCG^\pm$ induces the identity on $\C_4/\N$, then $g\in \mathcal{N}\mathcal{K}$.
\end{prop}

\begin{proof}
Choose a copy $\ov T\subset \C/\N$ of the rigid set, and let $T$ be one of its lifts. Then $g(T)$ is also a lift of $\ov T$, and by Lemma \ref{trianglelift} there is an element $h\in \mathcal{N}$ mapping $T$ to $g(T)$. Hence $h^{-1}\circ g$ is the identity on $T$, and therefore belongs to $\mathcal{K}$ by the "moreover" part of Theorem \ref{Trig}.
\end{proof}

While the injectivity part of the case $b\ge5$ will be very similar to the proof of Proposition \ref{injn4}, for the surjectivity part we will need a more general version of Lemma \ref{unicotriang}. The following subsections are devoted to establish such a result, the main tool being half twists around minimal vertices.

\subsection{Characterisation of half twists}
From now on let $b\ge5$. Let $\beta\in X_b\subseteq \C$ be a minimal curve, and recall that we denote by $H_\beta$ the half twist around $\beta$. Our first goal is to show that for any $\alpha\in X_b$ its image $H_\beta(\alpha)$ admits a graph-theoretic characterisation, which therefore is preserved by any automorphism of the curve complex. More precisely we want to prove the following:

\begin{lemma}\label{chartwist}
For any $\alpha\in X_b$:
\begin{itemize}
    \item $H_\beta(\alpha)=\alpha$ iff $d_{\C}(\alpha,\beta)\le 1$;
    \item otherwise there exists a facet $P\subseteq X_b$ such that $H_\beta(\alpha)$ is one of the two curves with special intersection with both $\alpha$ and $\beta$, with respect to $P$.
\end{itemize}
\end{lemma}

Notice that we did not uniquely characterise $H_\beta(\alpha)$, as it is impossible to distinguish it from $H_\beta^{-1}(\alpha)$ just from the curve graph (indeed, the orientation-reversing reflection $r$ fixes $X_b$ pointwise but swaps the half-twists). 

\begin{proof}
The first bullet is clearly true, so we focus on the second. Let $\alpha$ be a curve in $X_b$ that intersects $\beta$, and let $P$ be a facet that both $\alpha$ and $\beta$ complete, which cuts out a four-holed sphere $S_4$. Now, clearly $H_\beta(\alpha)$ has special intersection with both $\alpha$ and $\beta$ with respect to $P$ (the curves $\gamma$ and $\delta$ are easy to find in both cases); therefore it suffices to notice that there are only two curves on $S_4$ with intersection $2$ with both $\alpha$ and $\beta$, since in the curve graph of $S_{4}$ the edge with endpoints $\alpha$ and $\beta$ belongs to only two triangles.
\end{proof}

\subsection{Definition of half twists in the quotient}\label{section:htquot}
For the rest of the section, let $\N$ be a deep enough subgroup.
\begin{definition}
Let $\ov{X}_b$ be a copy of $X_b$ inside $\C/\N$, and let $\overline{\beta}\in\ov{X}_b$ be a minimal vertex. Let $s:\,\ov{X}_b\to \C$ be a lift. Then we define the half twist around $\overline\beta$ as the map $H_{\overline\beta}\colon \ov{X}_b\to \C/\N$, defined as the composition $\pi\circ H_{s(\overline\beta)}\circ s$.
\end{definition}
In other words, we define the half twist by lifting $\ov{X}_b$, applying a half twist and projecting, so that the half twist commutes with the quotient projection. The definition is well-posed, since any two lifts differ by some $g\in \mathcal{N}$. Notice moreover that $H_{\overline\beta}(\ov{X}_b)$ is still a copy of $X_b$ by Theorem \ref{projisometry}, since it is the projection of $H_{s(\overline\beta)}\left(s(\ov{X}_b)\right)$.\\
Our next goal is to show that the combinatorial characterisation of the half twist described in Theorem \ref{chartwist} still works in the quotient. In order to do so we need to show that:
\begin{enumerate}
    \item $H_{\overline\beta}(\overline\alpha)$ still satisfies the properties described in Theorem \ref{chartwist};
    \item Any $\overline\eta$ that satisfies the properties with $\overline\alpha$, $\overline\beta$ and $\overline P$ lifts to a curve that satisfies the properties with some lifts $\alpha$, $\beta$, $P$. Therefore $\overline\eta$ must be the projection of $H_\beta^{\pm 1} (\alpha)$.
\end{enumerate}
The following subsections will be dedicated to the proofs of these items.

\subsubsection{The projection of a twist looks like a twist}
Let $\ov{X}_b,\,\overline\alpha,\,\overline\beta$ as above. For short we set $\alpha=s(\overline\alpha)$ and similarly for any other vertex and subset of $s(\ov{X}_b)$. We subdivide the proof into two lemmas:
\begin{lemma}
    $H_{\overline\beta}(\overline\alpha)=\overline\alpha$ iff $d_{\Cdt}(\overline\alpha,\overline\beta)\le1$.
\end{lemma}
\begin{proof}
    Both $\ov\alpha$ and $\ov\beta$ belong to $\ov{X}_b$, which has diameter $2$. If $d_{\Cdt}(\overline\alpha,\overline\beta)\le1$ then $H_{\overline\beta}(\overline\alpha)=\ov\alpha$, as by construction $H_{\overline\beta}$ is the identity on the link of $\overline{\beta}$. If otherwise $\overline\alpha$ is at distance $2$ from $\overline\beta$, then $\alpha$ and $H_{\beta}(\alpha)$ are different curves in $Y_b$, and by Theorem \ref{projisometry} they project to different points.
\end{proof}
\begin{lemma}
    If $d_{\Cdt}(\overline\alpha,\overline\beta)=2$ then there exists a facet $\overline P$ such that $H_{\overline\beta}(\overline\alpha)$ has special intersection with both $\overline\alpha$ and $\overline\beta$, with respect to $\overline P$.
\end{lemma}
\begin{proof}
    Choose a facet $P\subseteq X_b$ that both $\alpha$ and $\beta$ complete, and let $\gamma,\delta\in X_b$ be some auxiliary curves detecting their special intersection, as in Figure \ref{fig:stella}. The images of these curves under $H_\beta$ are auxiliary curves for $\beta$ and $H_\beta(\alpha)$. Moreover we can find $\gamma',\delta'\in Y_b$ that detect the special intersection between $\alpha$ and $H_\beta(\alpha)$ (see Figure \ref{fig:auxcurve} for an example). The graph $G$ spanned by these curves is the union of two special pentagons patched along the simplex $H_\beta(\alpha)\star P$, as in Figure \ref{fig:grafotwist}. Since $G$ it is a subgraph of the graph $Y_b$, defined in Definition \ref{Y}, by Theorem \ref{projisometry} its projection is an isometric embedding of two special pentagons, detecting that $H_{\overline\beta}(\overline\alpha)=\pi(H_{\beta}(\alpha))$ has special intersection with both $\overline\alpha$ and $\overline\beta$.
\end{proof}

\begin{figure}[htp]
    \centering
    \includegraphics[width=0.75\textwidth]{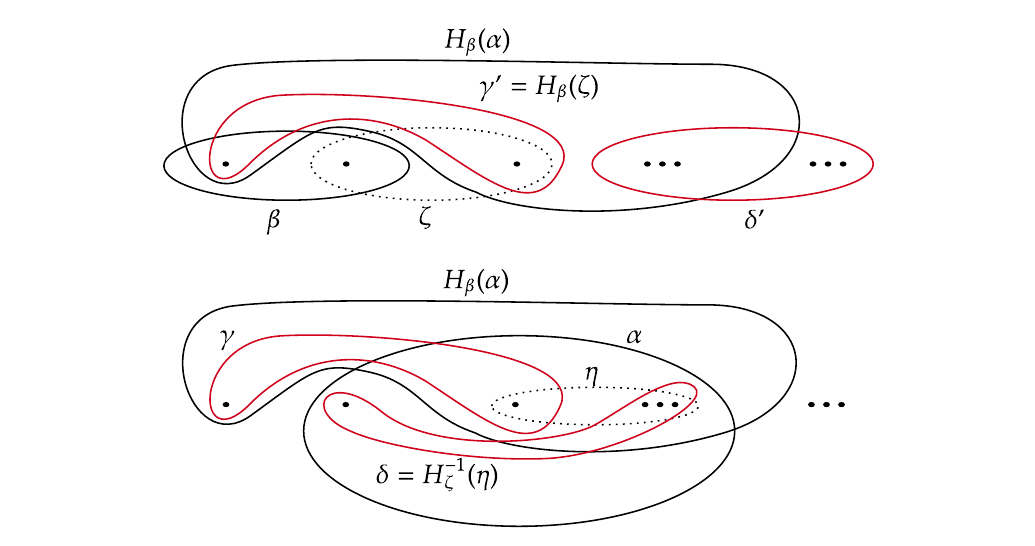}
    \caption{An example of how the auxiliary curves (here in red) may be chosen only using half twists around $\beta$ and $\zeta$, which are both in $X_b$. In our example $\gamma=\gamma'$, but this is not necessary. The case where $\alpha$ bounds a twice-punctured disk can be dealt with similarly.}
    \label{fig:auxcurve}
\end{figure}
    
\begin{figure}[htp]
    \centering
    \includegraphics[width=\textwidth]{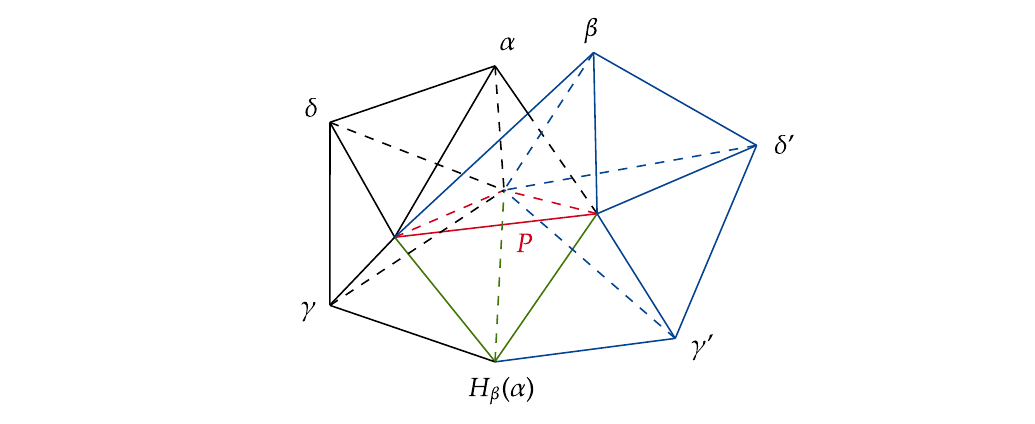}
    \caption{The configuration $G$ that detects the special intersection of $H_\beta(\alpha)$ with both $\alpha$ and $\beta$, that consists in two special pentagons (here, in black and blue) which overlap over a simplex (here, in green and red). Possibly $\gamma=\gamma'$ and $\delta=\delta'$, but each of the pentagons is isometrically embedded.}
    \label{fig:grafotwist}
\end{figure}

\subsubsection{If it looks like a twist it is the projection of a twist}
\begin{lemma}
Let $\beta\in X_b\subset\C$ be a minimal curve, and let $\alpha\in X_b$ have special intersection with $\beta$ with respect to the facet $P$. Let $\overline\alpha,\,\overline\beta,\,\overline P$ be their projections, and let $\overline\eta\in\C/\N$ be a vertex with special intersection with both $\overline\alpha$ and $\overline\beta$ with respect to $\overline P$. Then $\overline\eta$ lifts to a curve $\eta$ with special intersection with both $\alpha$ and $\beta$ with respect to $P$.
\end{lemma}

\begin{proof}
Consider the subgraph $\overline G\subseteq \C/\N$ made of two special pentagons that detect the special intersections, as in Figure \ref{fig:grafotwist}. This subgraph has a lift $G\subseteq \C$ (more precisely, each one of the special pentagons lifts, and we may arrange that the lifts coincide on $\overline\eta\star\overline P$ since there is a unique $\N$-orbit of lifts of simplices). Now, as a corollary of Lemma \ref{uniqueliftxspec} the union of the two simplices $\overline\alpha\star\overline P\cup\overline\beta\star\overline P$ admits a unique lift, up to elements in $\mathcal{N}$, and therefore we may assume that $G$ is glued to $X_b$ in such a way that the corresponding copies of $\alpha,\,\beta,\,P$ coincide. But now $G$ detects that $\overline\eta$ lifts to a curve with special intersection with the copies of $\alpha$ and $\beta$ inside $X_b$, as required.
\end{proof}

\subsection{The general case}
We are finally ready to prove the following combinatorial rigidity statement, which implies Theorem \ref{mainthm} when $\N=DT_K$ for suitable $K$.

\begin{theorem}\label{CombRig}
For every $b\ge5$ and for every deep enough subgroup $\N$, the natural map $MCG^\pm(S_{b})/\N\to\text{Aut}(\C(S_{b})/\N)$ is an isomorphism.
\end{theorem}

As before, we split the proof into two theorems.
\begin{theorem}\label{surjge5}
For $b\ge5$ and any $\N$ deep enough, any automorphism $\overline\phi\in\text{Aut}(\C/\N)$ is induced by an element of $MCG^\pm(S_b)$.
\end{theorem}

\begin{proof}
Consider $X_b$ inside $\C$ and let $\ov{X}_b$ be its projection, which is still a copy of $X_b$ by Theorem \ref{projisometry}. Moreover, let $\widetilde{X_b}$ be a lift of $\overline{\phi}(\ov{X}_b)$. By Theorem \ref{finiterig} there is an extended mapping class $h\in MCG^\pm(S_b)$ mapping $X_b$ to $\widetilde{X_b}$, which means that the following diagram commutes on $X_b\in\C$:
\begin{equation}\begin{tikzcd}\label{commge5}
    \C\ar{d}{\pi}\ar{r}{h}&\C\ar{d}{\pi}\\
    \C/\N\ar{r}{\overline{\phi}}&\C/\N
\end{tikzcd}\end{equation}
Now, choose a minimal vertex $\beta\in X_b$, and let $r$ be the reflection fixing $X_b$. We know that, if we fix a curve $\alpha_1\in X_b$ that intersects $\beta$, $h$ must map $H_{\beta}(\alpha_1)$ to either $H_{h(\beta)}(h(\alpha_1))$ or $H_{h(\beta)}^{-1}(h(\alpha_1))$, since these are the only two curves that satisfy the graph-theoretic characterisation of the half twist, Theorem \ref{chartwist}. Subsection \ref{section:htquot} shows that the same argument works in the quotient, hence there are only two possibilities for $\overline\phi\circ H_{\overline\beta}(\ov{\alpha}_1)$. Up to replacing $h$ with $h\circ r$ we may therefore assume that Diagram \ref{commge5} commutes also on $H_{\beta}(\alpha_1)$, and hence on the whole $H_{\beta}(X_b)$, which by Lemma \ref{alfaiunica} is uniquely determined by $H_{\beta}(\alpha_1)$ (and the copy of $X_{b-1}$ which is the link of $\beta$ in $X_b$; this is fixed by $H_\beta$).

We now claim that, if the diagram commutes on $X_b\cup H_\beta(X_b)$ and $\zeta$ is a minimal curve in $X_b$ that intersects $\beta$ then the diagram commutes also on $H_{\zeta}(X_b)$. In fact $H_\zeta^{-1}(\beta)=H_\beta(\zeta)$, as shown in Figure \ref{fig:Hincrociati}, therefore we already know that the diagram commutes on $H_\zeta^{-1}(\beta)$ (and therefore on $H_\zeta(\beta)$ which is the only other curve that satisfies the same characterisation as $H_\zeta^{-1}(\beta)$). Therefore the diagram commutes also on the whole $H_{\zeta}(X_b)$, which is uniquely determined by $H_{\zeta}(\beta)$.

\begin{figure}[htp]
    \centering
    \includegraphics[width=0.75\textwidth]{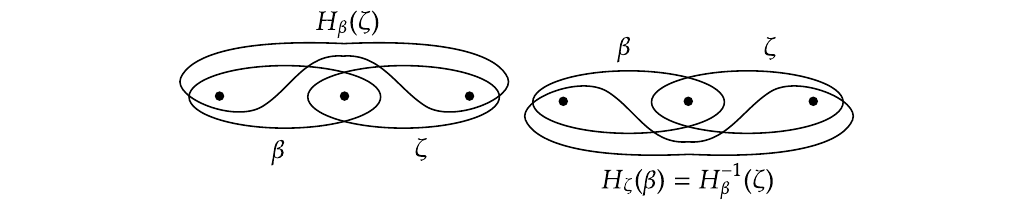}
    \caption{The two minimal curves from the proof of theorem \ref{surjge5} and their respective half twists.}
    \label{fig:Hincrociati}
\end{figure}

We may repeat the previous argument to show that Diagram \ref{commge5} commutes on $X_b$ and on any half twist around one of its minimal curves. Now it suffices to prove that the orbit of $X_b$ by successive half twists covers $\C$. That is, we are left to prove:

\par\medskip

{\bf Claim:} Given two copies $X_b$ and $X_b'$, there is a sequence $X_b=X^0, X^1, \ldots,  X^k=X_b'$ of copies of $X_b$ such that, for every $i=0,\ldots,k-1$, $X^{i+1}=H_\theta^{\pm}(X^i)$ for some minimal curve $\theta\in X^i$.

\begin{proof}[Proof of Claim]
To show this let $h$ be a (orientation-preserving) mapping class that maps $X_b$ to $X_b'$. Now, if $\theta_1,\ldots,\theta_b$ are the minimal curves in $X_b$ we know that $H_{\theta_1},\ldots,H_{\theta_b}$ generate the mapping class group. Therefore let $h=H_{\theta_{i_1}}^{\pm}\ldots H_{\theta_{i_k}}^{\pm}$, for some possibly repeated indices $i_1,\ldots,i_k\in\{1,\ldots,b\}$. Now we proceed by induction on $k$. If $k=1$ we are done. Otherwise notice that, by the properties of half Dehn twists,
$$h=H^{\pm}_{\theta_{i_1}}\ldots H^{\pm}_{\theta_{i_k}}=H^{\pm}_{H^{\pm}_{\theta_{i_1}}(\theta_{i_2})}\ldots H^{\pm}_{H^{\pm}_{\theta_{i_1}}(\theta_{i_k})}H^{\pm}_{\theta_{i_1}}$$
Now we have expressed $h$ in terms of $H_{\theta_{i_1}}^{\pm}$ and $k-1$ half twists around some minimal curves of $X^1:=H_{\theta_{i_1}}^{\pm}(X_b)$. Therefore we conclude by the inductive hypothesis.
\end{proof}

This concludes the proof of Theorem \ref{surjge5}.
\end{proof}

Now we turn to the injectivity part:
\begin{theorem}\label{injge5}
For any $b\ge5$ and any $\N$ deep enough, if $g\in MCG^\pm$ induces the identity on $\C/\N$, then $g\in \mathcal{N}$.
\end{theorem}

\begin{proof}
Choose a copy of $\ov{X}_b\subseteq \C/\N$, and let $X_b$ be one of its lifts. Then $g(X_b)$ is also a lift of $\ov{X}_b$, and by the uniqueness part of Theorem \ref{lifting} there is an element $h\in \mathcal{N}$ mapping $X_b$ to $g(X_b)$. Therefore $h^{-1}\circ g$ is the identity on $X_b$, and by the "moreover" part of Theorem \ref{finiterig} we must have that $h^{-1}\circ g\in\left\langle r\right\rangle$. Now, $h^{-1} g$ induces the identity map on $\C/\N$, as $g$ does so by hypothesis and $h$ is an element of $\N$. Therefore, in order to prove that $h^{-1}g$ must be the identity, we must prove that $r$ does not induce the identity on $\C/\N$. If this is not the case then, for any minimal curve $\beta\in X_b$, we have that $H_\beta(X_b)$ and $r\circ H_\beta(X_b)=r\circ H_\beta\circ r(X_b)=H_\beta^{-1}(X_b)$ coincide in the quotient (here we used that $r(X_b)=X_b$ and the properties of a half Dehn twist under conjugation). Therefore, by uniqueness of the orbit of lifts, there is an element $k\in \mathcal{N}$ which, for every curve $\gamma\in X_b$, maps $H_\beta^{-1}(\gamma)=r\circ H_\beta(\gamma)$ to $H_\beta(\gamma)$. By the "moreover" part of Theorem \ref{finiterig} there is only one orientation-preserving mapping class with this property, and therefore we must have that $T_\beta=k\in \mathcal{N}$ since the Dehn twist $T_\beta$ clearly maps $H_\beta^{-1}(\gamma)$ to $H_\beta(\gamma)$. But then since $\mathcal{N}$ is a normal subgroup it must contain all conjugates of $T_\beta$, that is, all Dehn twists around curves that bound twice-punctured disks. These mapping classes generate the \emph{pure} mapping class group $PMCG(S_b)$, which is the finite index subgroup of $MCG(S_b)$ consisting of elements that fix each puncture individually. But this subgroup is not deep enough if we choose a large enough constant $\Theta$ (for example, since the quotient $\C/PMCG$ is finite, which contradicts Corollary \ref{infinite}).
\end{proof}

\subsection{Finite rigidity of the quotient}
As a by-product of the proofs of Theorems \ref{CombRig} and \ref{mainthmb4} we get an analogue of Theorem \ref{finiterig} for the quotient. 
\begin{corollary}[Finite rigidity of $\C/\N$]\label{FinRigquot}
For $b\ge4$ and $\N$ deep enough, for every fixed copy $\ov{X}_b\subset\C/\N$ (where by $\ov X_4$ we mean a triangle), any isometric embedding $\phi:\ov{X}_b\to \C/\N$ is induced by a mapping class $\overline{h}\in MCG/\mathcal{N}$. Moreover, two such $\ov h$ differ by an element of the pointwise stabiliser $\text{PStab}(\ov X_b)$, which is the projection of $\text{Pstab}(X_b)$.
\end{corollary}

\begin{proof} For the existence part let $X_b$ and $X_b'$ be lifts of $\ov{X}_b$ and $\phi(\ov{X}_b)$, respectively. Then by finite rigidity there is a mapping class $h\in MCG$ mapping $X_b$ to $X_b'$, and the induced map $\overline{h}\in MCG/\mathcal{N}$ maps $\ov{X}_b$ to $\phi(\ov{X}_b)$.\\
For the uniqueness part, if $h$ and $h'$ both induce maps that map $\ov{X}_b$ to $\phi(\ov{X}_b)$ then $h(X_b)$ and $h'(X_b)$ are both lifts of $\phi(\ov{X}_b)$, hence by uniqueness of the orbit of lifts there exists a $g\in \mathcal{N}$ such that $g\circ h|_{X_b}=h'|_{X_b}$. But then $g\circ h$ and $h'$ must differ by an element of $\text{Pstab}(X_b)$.
\end{proof}

\section{From 1-separating to strongly separating}\label{section:1tostrong}
\begin{definition}
We say that a curve $\gamma\in \C$ is \emph{$1$-separating} if $S-\gamma$ has a connected component of complexity $1$. If $S$ is a punctured sphere, $\gamma$ is $1$-separating if and only if it bounds a disk with three punctures. Let $\Cp$ be the subgraph of $\C$ spanned by $1$-separating curves. 
\end{definition}

Moreover, we recall the following definition from \cite{ssgraph}:
\begin{definition}\label{defCss}
Let $\Css$ the full subgraph of the curve graph spanned by all \emph{strongly separating} curves, i.e. those separating curves that do not bound a twice-punctured disk.
\end{definition}
Notice that $\Cp\subseteq\Css$ whenever $b\ge6$. Moreover, the action of $DT_K$ restricts to both $\Cp$ and $\Css$, therefore the quotient $\Cpdt$ is a subgraph of $\Cssdt$, which is in turn a subgraph of $\Cdt$. 

As we will see in Section \ref{sec:qi}, every self-quasi-isometry of $MCG/DT_K$ induces an automorphism of $\Cpdt$. Therefore we must find a way to relate an automorphism of $\Cpdt$ with an automorphism of $\Cdt$, which we understand in view of the combinatorial rigidity Theorem \ref{CombRig}. For short, in the rest of the paper we will say that a proposition holds \emph{for all large multiples K} if there exists $K_0\in\mathbb{N}_{>0}$ such that the proposition holds whenever $K$ is a non-trivial multiple of $K_0$.
\begin{theorem}\label{doppiaext}
    For all $b\ge7$ and for all large multiples $K$, any automorphism of $\Cpdt$ is the restriction of an automorphism of $\Cdt$.
\end{theorem}
Combining this theorem with the combinatorial rigidity Theorem \ref{CombRig} we get the following:
\begin{theorem}\label{AutoExt}
    For all $b\ge7$ and for all large multiples $K$, any automorphism of $\Cpdt$ is the restriction of an element of $MCG^\pm(S_b)/DT_K$.
\end{theorem}
Following the roadmap of \cite{ssgraph, BowPants} we split the proof of Theorem \ref{doppiaext} into two intermediate extensions, one from $\Cpdt$ to $\Cssdt$ and one from $\Cssdt$ to $\Cdt$. 

For the rest of the section, we assume $b\ge 7$ and we take a multiple $K$ which is large enough that $\Cdt$ satisfies the Combinatorial Rigidity Theorem \ref{mainthm}. We first recall some definitions:
\begin{definition} Let $G$ be a graph. The \emph{dual graph} $G^*$ of $G$ is the graph with the same vertices of $G$ and with an edge between two vertices if and only if they are not adjacent in $G$.
\end{definition}
Notice that if $G\subseteq\C$ then $G^*$ is connected iff any two curves can be joined by a chain as in Definition \ref{chain}.

The following definition is from \cite[Section 7]{BowPants}: 
\begin{definition} Let $\Gamma$ be either $\Cp$ or $\Cpdt$. A \emph{division} of $\Gamma$ is an unordered pair $\{P^+,P^-\}$ of disjoint infinite subsets such that $P^+\star P^-\subset \Gamma$ is a maximal join and both $(P^+)^*$ and $(P^-)^*$ are connected. Two divisions $P^\pm,Q^\pm$ are nested if either $P^+\subseteq Q^+$ or $P^+\subseteq Q^-$.
\end{definition}
By \cite[Lemma 7.2]{BowPants} and the discussion following \cite[Lemma 7.5]{BowPants}, any division of $\Cp$ is induced by a unique strongly separating curve $\delta\in\Css\setminus\Cp$: namely, $P^\pm(\delta)$ correspond to the $1$-separating curves that fill the two sides of $\delta$. Moreover, nesting is equivalent to the corresponding $\delta$s being disjoint. Our goal is to give a similar description for the classes of $(\Cssdt)\setminus(\Cpdt)$.
\begin{definition}\label{filldiv}
Let $\Gamma$ be either $\Cp$ or $\Cpdt$. A \emph{filled division} of $\Gamma$ is a triple $(P^\pm,\alpha^\pm,\beta^\pm)$ such that:
\begin{enumerate}
    \item $P^\pm$ is a division of $\Gamma$, called the \emph{underlying} division;
    \item $\alpha^\pm,\beta^\pm\in P^\pm$, respectively;
    \item $d_\Gamma(\alpha^+,\beta^+)=d_\Gamma(\alpha^-,\beta^-)=2$
    \item\label{filldiv3} $P^+=\link_\Gamma(\alpha^-)\cap \link_\Gamma(\beta^-)$, and similarly for $P^-$.
\end{enumerate}
Two filled divisions are said to be equivalent if they have the same underlying $P^{\pm}$. From now on, by division we will always mean an equivalence class of filled divisions (i.e., we won't consider divisions which do not allow a filling).
\end{definition}
\begin{remark}
When $\Gamma=\Cp$, any division $P^\pm$ in the sense of Bowditch admits a filling. For example, since the division is induced by some strongly-separating curve $\delta$, we may choose $\alpha^\pm,\beta^\pm$ to be two pairs of curves that respectively fill the two surfaces cut out by $\delta$ (hence the name "filled division"). Notice, however, that this is not necessarily the case: the subsurfaces filled by $\alpha^\pm,\beta^\pm$ could cut out some twice punctured disks, which are irrelevant to our argument since they cannot contain any $1$-separating curve. We will elaborate on this Remark in the proof of Lemma \ref{Ssurj}, and the situation will be depicted in Figure \ref{fig:filledsubsurf}.
\end{remark}

\begin{remark}\label{fillingpairs}
We could state Definition \ref{filldiv} just in terms of four vertices, by saying that a filled division is given by two pairs of vertices $(\alpha^\pm,\beta^\pm)\in \Gamma$ such that:
\begin{enumerate}
    \item \label{i} $\alpha^+,\alpha^-,\beta^+,\beta^-$ is an isometrically embedded square;
    \item \label{ii} $P^\pm:=\link_\Gamma(\alpha^\mp)\cap \link_\Gamma(\beta^\mp)$ are infinite;
    \item \label{iii} if $\gamma\in P^+$ and $\gamma'\in P^-$ then $d_\Gamma(\gamma,\gamma')=1$.
    
\end{enumerate}
This is because by Conditions \ref{ii} and \ref{iii} of Definition \ref{fillingpairs} we have that $P^+$ and $P^-$ are disjoint infinite subsets that form a join $P^+\star P^-$, which is maximal since any subgraph $Q^+\subset \Gamma$ which forms a join with $\alpha^-$ and $\beta^-$ must be contained in $P^+$. Moreover, Condition \ref {i} says that $\alpha^+$ and $\beta^+$ are connected in $(P^+)^*$. Finally, since $P^+$ and $P^-$ are disjoint we get that there is no $\gamma\in\Gamma$ such that $\alpha^\pm,\beta^\pm\in \link_\Gamma(\gamma)$; therefore a vertex in $P^+=\link_\Gamma(\alpha^-)\cap\link_\Gamma(\beta^-)$ which is not $\alpha^+$ nor $\beta^+$ must be at distance at least $2$ from either $\alpha^+$ or $\beta^+$, so $(P^+)^*$ is connected.
\end{remark}

The following definition is needed to take into account also classes in $\Cpdt$:
\begin{definition}
    A \emph{slice} (respectively, \emph{filled slice}) is either a division (resp. filled division) or a vertex $\delta\in \Gamma$. A vertex is nested into a division $P^\pm$ if it belongs to either $P^+$ or $P^-$. Two vertices are nested if they are disjoint.
\end{definition}
Now we want to assign to every $\overline{\delta}\in(\Cssdt)\setminus(\Cpdt)$ some filled division $(\overline{P}^\pm, \overline\alpha^\pm, \overline\beta^\pm)$. First, choose representatives $\delta_1,\ldots,\delta_k\in \Css\setminus\Cp$ for every homeomorphism type of strongly separating, non-$1$-separating curves. For every $i=1\ldots,k$ choose $\alpha_i^\pm,\beta_i^\pm$ as two pairs of curves in $\Cp$ that fill the sides of $\delta_i$. Then for every $\delta\in \Css\setminus\Cp$ choose some mapping class $f$ such that $\delta=f(\delta_i)$ for some $i\in\{1,\ldots, k\}$, and set $(\alpha^\pm(\delta), \beta^\pm(\delta))=f(\alpha_i^\pm, \beta_i^\pm)$, which are again two pairs of filling curves for the sides of $\delta$. Now, choose $\Theta$ big enough that, for every $i=1,\ldots,k$ and for every $s\in\C$, we have that $\max\{d_s(\alpha_i^\pm,\beta_i^\pm),d_s(\alpha_i^\pm,\beta_i^\mp)\}\le \Theta$. Since $(\alpha^\pm(\delta), \beta^\pm(\delta))$ are defined as images via mapping classes of some $(\alpha_i^\pm, \beta_i^\pm)$ we also get that, for every $\delta\in\Css\setminus\Cp$ and for every $s\in\C$,
\begin{equation}\label{shortprojab}
    \max\{d_s(\alpha^\pm(\delta), \beta^\pm(\delta)),d_s(\alpha^\pm(\delta), \beta^\mp(\delta))\}\le \Theta.
\end{equation}

This way, if $K$ is a large enough multiple with respect to this constant $\Theta$ we will be able to invoke Lemma \ref{projinlink} and its corollaries, which will tell us that the projection map is isometric when restricted to these curves. The key point is that, since there are only finitely many $MCG$-orbits of curves $\delta$, we could fix once and for all a $MCG$-equivariant choice of four curves $\alpha^\pm(\delta), \beta^\pm(\delta)$ in every orbit, and therefore we only need to bound the annular projections of finitely many curves. This kind of argument will be recurrent throughout the paper. 

Now let $(\overline\alpha^\pm, \overline\beta^\pm):=\pi(\alpha^\pm(\delta), \beta^\pm(\delta))$.

\begin{lemma}\label{fillprojtofill}
    The two pairs $(\overline\alpha^\pm, \overline\beta^\pm)$ give a filled division.
\end{lemma}

\begin{proof}
We show that the conditions of Remark \ref{fillingpairs} are satisfied. First notice that, since by construction all projections between $\alpha^\pm$ and $\beta^\pm$ are short (see Equation (\ref{shortprojab})), Lemma \ref{projinlink} tells us that the quotient map $\pi$ is an isometry on the square spanned by these curves. This proves Condition \ref{i}. 

For Condition \ref{ii} let $P^\pm\subset \Cp$ be the division induced by $\delta$, and let $\ov P^\pm$ be defined as in Remark \ref{fillingpairs}. Clearly $\pi(P^\pm)\subset \ov P^\pm$, since any $\gamma\in P^-$ is disjoint from both $\alpha^+$ and $\beta^+$ and therefore $\pi(\gamma)\in \link(\overline\alpha^+)\cap \link(\overline\beta^+)$, again by Lemma \ref{projpuntsep}. Thus it is enough to show that $\pi(P^\pm)$ is infinite. Let $\Sigma^\pm$ be the two subsurfaces cut out by $\delta$, so that $P^\pm\subset \C(\Sigma^\pm)$. Since $\delta\not \in \Cp$ both subsurfaces have complexity at least $2$, therefore $\pi(\C(\Sigma^\pm))\cong \C(\Sigma^\pm)/DT_K(\Sigma^\pm)$ by Corollary \ref{CUdt}. Now we claim that we can find a large multiple $K$ such that 
$\C(\Sigma^\pm)/DT_K(\Sigma^\pm)$ is infinite. In order to do so, for every topological type of $\Sigma^+$ we can choose a curve $x\in P^+$ and a pseudo-Anosov element $g\in MCG(\Sigma^+)$ that fixes the boundary $\delta$ pointwise. Notice that for every $n\in \mathbb{Z}$ we have that $g^n(x)$ is again $1$-separating for $S$, thus $g^n(x)\in P^+$. Now we can proceed as in Corollary \ref{infinite} to show that, whenever $K$ is a large  multiple, the projection $\pi$ is an isometry on the axis $\{g^n(x)\}_{n\in\mathbb{Z}}$, and in particular $\pi(P^+)$ is infinite. Since there are only finitely many topological types of $\Sigma^+$ we can choose a large multiple $K$ that works for any $\Sigma^+$. Moreover the whole argument can be repeated to show that $\pi(P^-)$ is infinite.

For Condition \ref{iii}, first consider a simplex $\Delta^+\in\C$ that contains $\delta$, is disjoint from $\alpha^+$ and $\beta^+$ and is maximal with these properties, and let $\overline\Delta^+\in\Cdt$ be its projection, which is a simplex of the same dimension by Corollary \ref{projsimplex}. Let $\overline\gamma\in\Cpdt$ be a vertex in $\link_{\Cpdt}(\overline\alpha^+)\cap \link_{\Cpdt}(\overline\beta^+)$. If we look at these three vertices inside $\Cdt$ we see that $\overline\alpha^+,\overline\beta^+\in \link_{\Cdt}(\overline\Delta^+)$, i.e. we have a generalised square whose "vertices" are $\overline\gamma,\,\overline\alpha^+,\,\overline\Delta^+,\,\overline\beta^+$ as in Figure \ref{fig:gensquare}.

\begin{figure}[htp]
    \centering
    \includegraphics[width=0.75\textwidth]{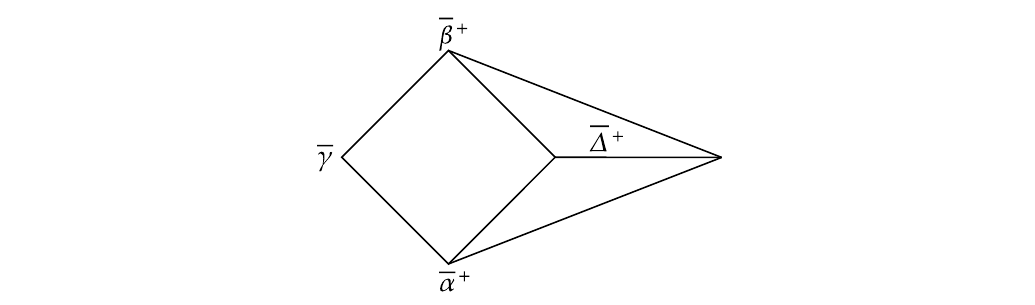}
    \caption{The generalised square from Lemma \ref{fillprojtofill}. The simplex $\overline\Delta^+$ is represented as a segment.}
    \label{fig:gensquare}
\end{figure}

We may lift this square and assume that the lift of $\overline\Delta^+$ is $\Delta^+$, by uniqueness of the orbit of lifts of a simplex, which is Theorem \ref{mgonlift}. In particular, we may assume that the lift of $\overline\delta$ is still $\delta$. Let $\alpha',\,\beta'$ be the lifts of $\overline\alpha^+,\,\overline\beta^+$, which again fill one of the sides of $\delta$ by Corollary \ref{projfilling}. But now the lift $\gamma$ of $\overline\gamma$ is disjoint from $\alpha'$ and $\beta'$ and cannot coincide with $\delta$, since $\delta\not\in \Cp$. Therefore $\gamma$ must lie on the other side of $\delta$ with respect to $\alpha'$ and $\beta'$, i.e. $\gamma\in P^-$. The same argument applies to any vertex $\overline\eta\in \link(\overline\alpha^-)\cap \link(\overline\beta^-)$ and produces a lift $\eta\in P^+$. But then $\gamma$ and $\eta$ lie on different sides of $\delta$, and in particular they must be disjoint. Therefore, by Item 2 of Lemma \ref{projpuntsep}, we get that $d_{\Cpdt}(\ov\gamma,\ov\eta)=1$, as required.
\end{proof}

Now we want to associate to every $\overline\delta$ in $\Cssdt$ a slice of $\Cpdt$. If $\overline\delta\in \Cpdt$ set $S(\overline\delta):=\overline\delta$. Otherwise choose a lift $\delta$, take the corresponding filled division $(\alpha^\pm(\delta), \beta^\pm(\delta))$, constructed as above, and let $S(\overline\delta):=P^\pm(\overline\delta)$ be the underlying division of $(\overline\alpha^\pm, \overline\beta^\pm):=\pi(\alpha^\pm(\delta), \beta^\pm(\delta))$. 

\begin{theorem}\label{thm:1toss}
For all large multiples $K$, the map $S$ described above is a well-defined bijection between vertices in $\Cssdt$ and slices of $\Cpdt$, that translates adjacency into nesting. Hence any automorphism of $\Cpdt$ extends to an automorphism of $\Cssdt$.
\end{theorem}
We need to show that $S$ is a well-defined bijection when restricted to $(\Cssdt)\setminus(\Cpdt)$. If this is the case then clearly it maps adjacent vertices to nested slices and vice versa, and the proof of Theorem \ref{thm:1toss} is complete. We break the proof into a series of lemmas.

\begin{lemma}
    The map $S$ is well-defined. In other words, the corresponding division $\overline P^\pm$ of $(\overline\alpha^\pm, \overline\beta^\pm)$ is independent on the choice of the lift $\delta$.
\end{lemma}

\begin{proof}
When we verified Conditions \ref{ii} and \ref{iii} in the proof of Lemma \ref{fillprojtofill} we actually showed that $\ov P^\pm=\pi(P^\pm)$, and the latter depends only on $\delta$. Now it suffices to notice that, if $g\in DT_K$, then the division induced by $g(\delta)$ is clearly $g(P^\pm)$.
\end{proof}

Now we look for an inverse of $S$.

\begin{lemma}\label{Ssurj}
    Let $\overline P^\pm$ be a division that admits a filling. There exists a vertex $\overline\delta$ such that $S(\overline\delta)=\overline P^\pm$.
\end{lemma}

\begin{proof}
    Let $(\overline\alpha^\pm,\overline\beta^\pm)$ be two pairs of vertices whose corresponding division is $\overline P^\pm$. Lift them to some curves $(\alpha^\pm,\beta^\pm)$ that form a square in $\Cp$. Let $\Sigma^+,\,\Sigma^-$ be the subsurfaces filled by $\alpha^+\cup\beta^+$ and $\alpha^-\cup\beta^-$, respectively. The boundaries of these subsurfaces are separating, as they are curves on a sphere, hence they cut out some peripheral punctured disks and a (possibly punctured) annulus between them. Moreover there is no curve $\gamma\in \Cp$ that does not cross any of the boundaries, because otherwise $\link(\pi(\gamma))$ would contain $(\overline\alpha^\pm,\overline\beta^\pm)$. Therefore we are in the situation of Figure \ref{fig:filledsubsurf}: all peripheral disks contain two punctures each.\\
    If we show that the annulus cannot contain any puncture then the core curve $\delta$ of this annulus is the only strongly separating curve which is disjoint from $\alpha^\pm$ and $\beta^\pm$, which can be also characterised as the only boundary component of the surface filled by $\alpha^+\cup\beta^+$ which lies in $\Css$. If otherwise the annulus contains some puncture we can find two intersecting $1$-separating curves $\gamma^\pm$ such that $\gamma^+\in \link(\alpha^-)\cap \link(\beta^-)$ and $\gamma^-\in \link(\alpha^+)\cap \link(\beta^+)$. This is because the annulus must cut the surface into two subsurfaces, each of which contains the disk with three punctures bounded by $\alpha^+$ and $\alpha^-$, respectively. Hence every side of the annulus contains at least three punctures, and we can choose some curve $\gamma^+$ that bounds the puncture inside the annulus and two other punctures on the side of $\alpha^+$. The same holds for $\gamma^-$, as depicted in Figure \ref{fig:filledsubsurf}. Now, since we can choose $\gamma^\pm$ to induce different puncture separations, their projections remain at distance at least $2$ by Lemma \ref{projpuntsep}, hence they contradict Condition \ref{iii} for $(\overline\alpha^\pm,\overline\beta^\pm)$.

    \begin{figure}[htp]
        \centering
        \includegraphics[width=0.75\textwidth]{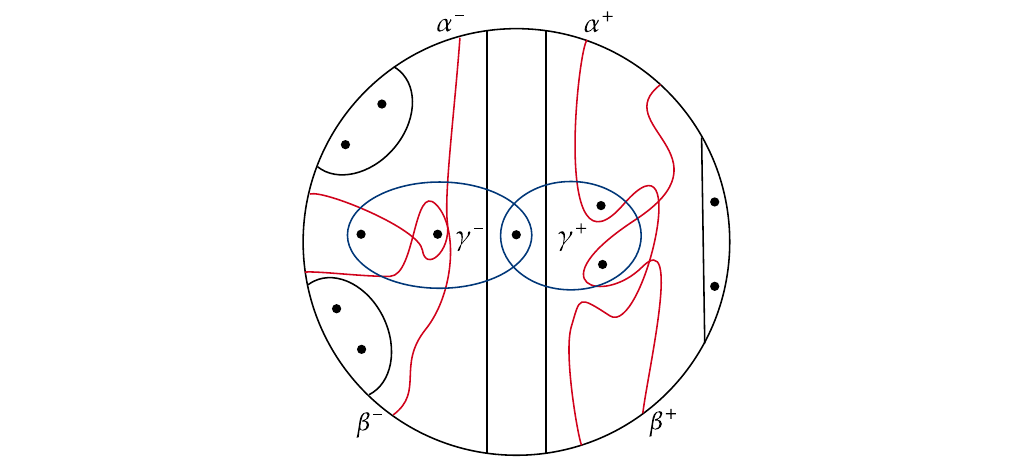}
        \caption{The red curves fill the subsurfaces whose boundaries are the black curves, and may cut out some punctured disks. If the annulus in between contains some punctures then the projections of the blue curves remain at distance at least 2.}
        \label{fig:filledsubsurf}
    \end{figure}

    Now, let $P^\pm$ be the division induced by $\delta$ and filled by $(\alpha^\pm,\beta^\pm)$. Clearly $\pi(P^\pm)\subseteq \overline P^\pm$, since a curve in $\link(\alpha^+)\cap \link(\beta^+)$ projects inside $\link(\overline\alpha^+)\cap \link(\overline\beta^+)$. Moreover by construction $\pi(P^\pm)=S(\pi(\delta))$, and in particular it is a maximal join. Hence $S(\pi(\delta))= \overline P^\pm$ by maximality.
\end{proof}

\begin{lemma}\label{deltaindep}
    For all large multiples $K$, the vertex $\overline\delta$ from Lemma \ref{Ssurj} does not depend on the choice of the filling vertices $(\overline\alpha^\pm,\overline\beta^\pm)$, nor on the choice of their lifts.
\end{lemma}

The proof of this lemma will be prototypical of many arguments throughout the paper.
\begin{proof}
    Choose two pairs of filling vertices $(\overline\alpha_1^\pm,\overline\beta_1^\pm)$ and $(\overline\alpha_2^\pm,\overline\beta_2^\pm)$. Without loss of generality we may assume that $(\overline\alpha_1^+,\overline\beta_1^+)=(\overline\alpha_2^+,\overline\beta_2^+)$, since we can replace one pair at a time. Let $(\alpha_1^\pm,\beta_1^\pm)$ and $(\alpha_2^\pm,\beta_2^\pm)$ two lifts forming two squares. The argument of Lemma \ref{Ssurj}, whose construction works for any choice of lifts of $(\overline\alpha_1^\pm,\overline\beta_1^\pm)$, shows that there exists a unique curve $\delta_1\in\Css$ which is disjoint from $\alpha_1^\pm$ and $\beta_1^\pm$, and we can similarly find $\delta_2$. Up to elements in $DT_K$ we may assume that $\alpha_1^+=\alpha_2^+=\alpha^+$, thus we are in the situation depicted in Figure \ref{fig:twosquares}.\\
    Now, there exists $g\in DT_K$ such that $g(\beta_1^+)=\beta_2^+$. We want to prove that, up to changing the lift, we may "glue" $\beta_1^+$ to $\beta_2^+$, that is, we can find a lift of the two squares in Figure \ref{fig:twosquares} such that $\beta_1^+=\beta_2^+$. If $g$ is not the identity let $(s,\gamma_s)$ be as in Proposition \ref{cor3.6}. If $d_{\C}(s,\beta_1^+)\le1$ we may apply $\gamma_s$ to all data and proceed by induction on the complexity of $g$. Otherwise $d_s(\beta_1^+,\beta_2^+)\ge\Theta$. Now, at least one of the cut sets $\{\alpha_1^-,\beta_1^-,\delta_1\},\,\{\alpha^+\},\,\{\alpha_2^-,\beta_2^-,\delta_2\}$ must be fixed pointwise by $\gamma_s$, because if $DT_K$ is deep enough there cannot be a path from $\beta_1^+$ to $\beta_2^+$ of length $4$ with no points in the star of $s$. In fact, for every fixed $n\in \mathbb{N}$ we can choose $\Theta\ge n B$, where $B$ is the constant from the Theorem \ref{bgit} (BGI). Therefore, whenever $p,q\in \C$ have large projections on some $s$, any piecewise geodesic path between $p$ and $q$ which is made of at most $n$ geodesic subpaths must intersect the star of $s$, because otherwise we could find a geodesic subpath whose endpoints have projections at least $B$ on $s$, thus violating BGI. In our case we can choose $n$ to be $4$. Therefore we may apply $\gamma_s$ to the lift "beyond" the points in the star of $s$ (that is, to the connected component that contains $\beta_2^+$ of the complement of the cut set), and again proceed by induction on the complexity of $g$.\\
    At the end of this procedure we have replaced $\delta_1$ and $\delta_2$ by some of their images via elements of $DT_K$. But now $(\alpha_1^+,\beta_1^+)=(\alpha_2^+,\beta_2^+)$, and therefore $\delta_1=\delta_2$ since they are both characterised as the only boundary component of the surface filled by $\alpha^+\cup\beta^+$ which lies in $\Css$.
\end{proof}
 \begin{figure}[htp]
        \centering
        \includegraphics[width=0.75\textwidth]{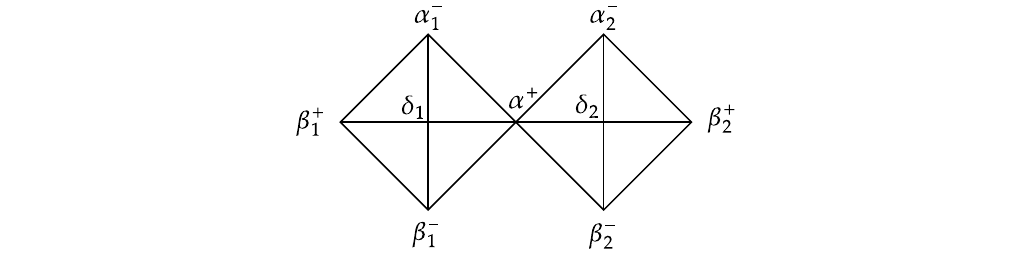} 
        \caption{The two squares inside $\Css$ described in the proof of Lemma \ref{deltaindep}. We may independently choose a point for every column and find a path from $\beta_1^+$ to $\beta_2^+$ passing through those points. Every such path must pass through the star of $s$, i.e. must contain a vertex fixed by $\gamma_s$.}
        \label{fig:twosquares}
    \end{figure}

Now the proof of Theorem \ref{thm:1toss} is complete if we show that:
\begin{corollary}
    The map $S$ is injective.
\end{corollary}

\begin{proof}
    We want to show that if $S(\overline\delta_1)=S(\overline\delta_2)$ then $\overline\delta_1=\overline\delta_2$. Let $\delta_1$ and $\delta_2$ be two lifts of these vertices, let $(\alpha_1^\pm,\beta_1^\pm)$ and $(\alpha_2^\pm,\beta_2^\pm)$ be some filling curves for the lifts and let $(\overline\alpha_1^\pm,\overline\beta_1^\pm)$ and $(\overline\alpha_2^\pm,\overline\beta_2^\pm)$ be their projections, which induce the same division $\overline P^\pm$. Notice that, by construction, $\delta_1$ is the curve obtained by applying the machinery of Lemma \ref{Ssurj} to $(\overline\alpha_1^\pm,\overline\beta_1^\pm)$ with lifts $(\alpha_1^\pm,\beta_1^\pm)$, and similarly for $\delta_2$. But the previous lemma shows that $\overline\delta_1=\overline\delta_2$, since the class of the curve $\delta$ obtained by Lemma \ref{Ssurj} does not depend on the choice of filled divisions for $\overline P^\pm$.
\end{proof}

\section{From strongly separating to all curves}\label{section:strongtoall}
For this whole section we follow the footsteps of \cite{ssgraph}, except we focus on the case of punctured spheres. We recall that, as in Definition \ref{defCss}, we denote by $\Css$ the full subgraph of $\C$ spanned by strongly separating curves, that is, all separating curves that do not bound twice-punctured disks. We want to prove that:
\begin{theorem}\label{csstoc}
    For $b\ge7$ and for all large multiples $K$, any automorphism of $\Cssdt$ extends to an automorphism of $\Cdt$.
\end{theorem}
In \cite{ssgraph}, Bowditch identifies every \emph{minimal curve} $\omega$ (meaning, a curve that bounds a twice-punctured disk) with a certain equivalence class of pairs of curves in $\Css$. We recall the following definitions from that paper:
\begin{definition}
    Two curves $\alpha,\beta\in \Css$ form a \emph{surrounding pair} if they bound three-punctured disks whose intersection is a twice-punctured disk. The boundary of the latter is the minimal curve \emph{surrounded} by $\alpha$ and $\beta$. See Figure \ref{fig:surrpair}.
\end{definition}
\begin{figure}[htp]
    \centering
    \includegraphics[width=0.75\textwidth]{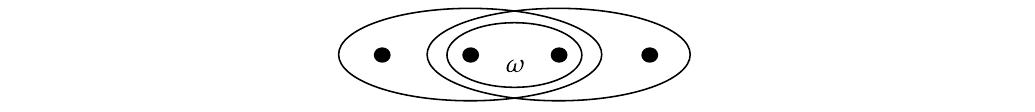}
    \caption{A surrounding pair and the minimal curve $\omega$ it surrounds.}
    \label{fig:surrpair}
\end{figure}

\begin{definition}
    Three curves $\alpha,\beta,\gamma\in \Css$ form a \emph{surrounding triple} if any two of them form a surrounding pair that surrounds the same minimal curve, as in Figure \ref{fig:surrtriple}.
\end{definition}
\begin{figure}[htp]
    \centering
    \includegraphics[width=0.75\textwidth]{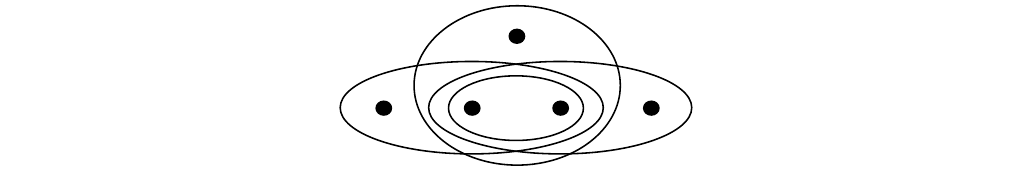}
    \caption{A surrounding triple.}
    \label{fig:surrtriple}
\end{figure}
The following is a restatement of \cite[Lemmas 4.2 and 4.3]{ssgraph}, which are core results of that paper:
\begin{lemma}\label{chainoftriples}
Any two surrounding pairs that surround the same $\omega$ are connected by a finite sequence of surrounding triples (i.e., by successively replacing one of the two curves with a third curve that forms a surrounding triple with the other two). Thus minimal curves correspond to equivalence classes of surrounding pairs, up to surrounding triples.\\
Moreover, two minimal curves $\omega, \omega'$ are disjoint if and only if there are two disjoint curves $\alpha$ surrounding $\omega$ and $\alpha'$ surrounding $\omega'$. A minimal curve $\omega$ is disjoint from a curve $\beta$ if and only if either $\beta$ surrounds $\omega$ or there is some $\alpha$ surrounding $\omega$ and disjoint from $\beta$.
\end{lemma}
This implies that, whenever surrounding pairs and surrounding triples can be recognised inside $\Css$, any automorphism of $\Css$ extends to an automorphism of $\C$. Our goal is to repeat the same kind of argument: we want to define surrounding pairs and triples inside $\Cssdt$ just using combinatorial properties and then show that they actually correspond to projections of surrounding pairs and triples. Then we can rely on Lemma \ref{chainoftriples} to show that two surrounding pairs in $\Cssdt$ are connected by a finite sequence of surrounding triples, and we are almost done. As in \cite{ssgraph} we distinguish three cases, whether the number of punctures is $7$, $8$ or at least $9$.

For the rest of the section $K$ will always denote a large enough multiple, depending on the constants that will successively appear.

\subsection{7 punctures}\label{section:7pt}
We will say that an $n$-cycle inside a graph $\Gamma$ is a cyclically ordered
sequence of n vertices, where consecutive vertices are adjacent. The following Lemma, which summarises \cite[Sections 3 and 4]{ssgraph}, shows that surrounding pairs and triples can be recognised inside $\Css(S_7)$ (which actually coincides with $\Cp(S_7)$) by using $7$-cycles:
\begin{lemma}\label{Bow7}
    In a seven-punctured sphere:
    \begin{itemize}
        \item Inside $\Cp(S_7)$ there is only one $7$-cycle, or heptagon, up to the action of the mapping class group, and this heptagon is isometrically embedded;
        \item Two curves $\alpha,\beta\in \Cp$ form a surrounding pair if and only if they are at distance $2$ in some heptagon;
        \item Three curves $\alpha,\beta,\gamma$ form a surrounding triple if any two of them form a surrounding pair and there is no curve in $\Cp$ that is disjoint from each of $\alpha,\beta,\gamma$. 
    \end{itemize}
\end{lemma}

\begin{figure}[htp]
    \centering
    \includegraphics[width=0.75\textwidth]{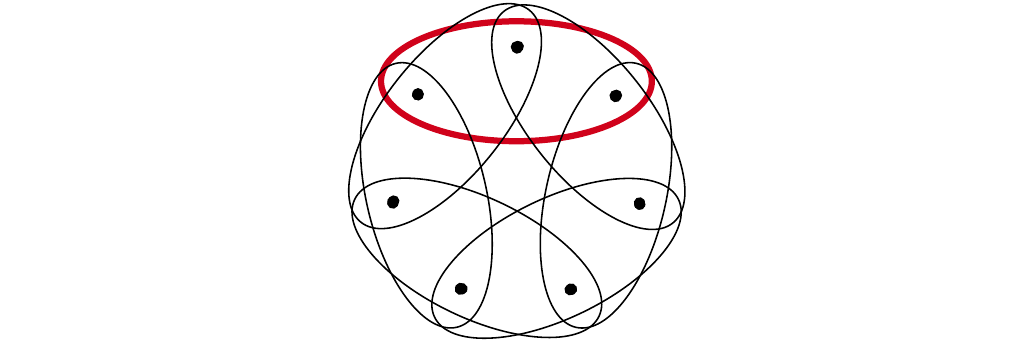}
    \caption{These curves form an isometrically embedded heptagon in $\Cp(S_7)$. The curves are all obtained from the same curve, say the red one at the top, by rotating. Notice that all curves induce different puncture separations.}
    \label{fig:heptagon}
\end{figure}

Thus, a good definition for surrounding pairs inside $\Cssdt$ is the following:
\begin{definition}
    Two vertices $\ov\alpha,\ov\beta\in \Cssdt$ form a surrounding pair if they are at distance $2$ in some heptagon inside $\Cssdt$.
\end{definition}
It is clear that if $\ov\alpha,\ov\beta$ are a surrounding pair then we can find lifts $\alpha,\beta$ that are a surrounding pair in $\Css$, just by lifting the corresponding heptagon. Thus $\alpha,\beta$ surround some minimal curve $\omega$, and we say that $\ov\alpha,\ov\beta$ surround its projection $\ov\omega$.

\begin{lemma}\label{surrpair}
The vertex $\ov\omega$ surrounded by $\ov\alpha,\ov\beta$ is well-defined, meaning it does not depend on the chosen heptagon, nor on its lift.
\end{lemma}
\begin{proof}
Let $\ov T,\ov T'\subset \Cssdt$ be two heptagons in which $\ov\alpha,\ov\beta$ are at distance $2$, and let $T,T'$ be two lifts of $\ov T,\ov T'$. Up to elements of $DT_K$ we can assume that the lifts of $\ov\alpha$ coincide, and let $\beta,\beta'$ be the lifts of $\ov\beta$. Let $\omega, \omega'$ be the minimal curves surrounded by the two pairs.  Let $g\in DT_K$ be such that $g(\beta)=\beta'$. Then the picture in $\C$ is as in Figure \ref{fig:liftingpairs}.
\begin{figure}[htp]
    \centering
    \includegraphics[width=0.75\textwidth]{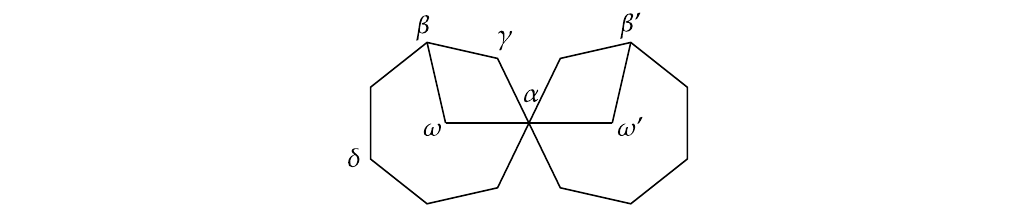}
    \caption{The two heptagons from the proof of Lemma \ref{surrpair}.}
    \label{fig:liftingpairs}
\end{figure}\\
Let $(s,\gamma_s)$ be as in Proposition \ref{cor3.6}. We now show that we can apply $\gamma_s$ to part of our diagram without breaking the heptagons, and ensuring that $\omega$ and $\omega'$ remain surrounded by the corresponding pairs. If this is true then we can proceed by induction on the complexity of $g$ to glue $\beta$ to $\beta'$, and in the end $\omega$ and $\omega'$ will both be the unique curve surrounded by $\alpha$ and $\beta$. We can argue as in the proof of Lemma \ref{deltaindep} to show that $\gamma_s$ must fix one of the following:
\begin{itemize}
    \item $\beta$;
    \item $\omega$ and some curves $\gamma,\delta\in T$ on the two paths in $T$ from $\beta$ to $\alpha$;
    \item $\alpha$;
    \item $\omega'$ and some curves $\gamma',\delta'\in T'$ on the two paths in $T'$ from $\beta'$ to $\alpha'$.
\end{itemize}
In all cases we can apply $\gamma_s$ beyond the cut set and proceed by induction. We just need to be careful that in the second case $\omega$ remains the curve surrounded by $\beta$ and $\gamma_s(\alpha)$. This is true, since the curve surrounded by a pair is characterised as the only minimal curve inside both of the three-punctured disks defined by the surrounding curves. Thus it suffices to notice that $\omega=\gamma_s(\omega)$ is still inside the disk surrounded by $\gamma_s(\alpha)$. The exact same argument shows that, in the fourth case, $\omega'$ remains the curve surrounded by $\alpha'$ and $\gamma_s(\beta')$.
\end{proof}

This shows that every surrounding pair in $\Cssdt$ corresponds to some minimal $\ov\omega$. Conversely, given some $\ov\omega$ it is easy to find some $\ov\alpha,\ov\beta$ that surround it: just pick a lift $\omega$, find a surrounding pair and let $T$ be the corresponding heptagon. Since by Lemma \ref{Bow7} there is a unique heptagon in $\Cp(S_7)$ up to the action of the mapping class group, $T$ looks like in Figure \ref{fig:heptagon}, and therefore all curves of $T$ induce different puncture separations. Thus, by Lemma \ref{puntsepdiverse}, the projection is injective when restricted to $T$, and therefore $\pi(T)$ is a $7$-cycle.

Now we want to define surrounding triples, in order to characterise the class of pairs that surround the same $\ov\omega$. Again, our definition is inspired by Lemma \ref{Bow7}.
\begin{definition}
    Three vertices $\ov\alpha,\ov\beta,\ov\gamma\in\Cssdt$ form a surrounding triple if they are pairwise surrounding pairs and there is no $\ov\delta\in \Cssdt$ which is adjacent to each of them.
\end{definition}

\begin{lemma}\label{projtriples}
    A surrounding triple $\alpha,\beta,\gamma\in\Css$ projects to a surrounding triple $\ov\alpha,\ov\beta,\ov\gamma\in\Cssdt$.
\end{lemma}

\begin{proof}
We already know that surrounding pairs project to surrounding pairs. Now let $\ov\delta\in \Cssdt$ and let $\delta$ be one of its lifts. We want to show that $\ov\delta$ is not adjacent to each of $\ov\alpha,\ov\beta,\ov\gamma$. Since $\alpha,\beta,\gamma$ is a surrounding triple we have that $\delta$ must intersect at least one of the three curves, say $\alpha$. If $\delta$ and $\alpha$ induce different puncture separations then $d_{\Cssdt}(\ov\alpha, \ov \delta)\ge 2$ by Lemma \ref{projpuntsep}. Otherwise $\delta$ and $\alpha$ surround the same three punctures, and in particular $\delta$ and $\beta$ must intersect with different puncture separations. Hence $d_{\Cssdt}(\ov\beta, \ov \delta)\ge 2$ for the same reason.
\end{proof}

\begin{corollary}
    Any two surrounding pairs for the same $\ov\omega$ are connected by a finite number of surrounding triples.
\end{corollary}

\begin{proof}
Just lift the two surrounding pairs in such a way that they surround the same lift $\omega$. Then the conclusion follows from Lemmas \ref{Bow7} and \ref{projtriples}.
\end{proof}
Moving on, we need to show that the vertex $\ov\omega$ surrounded by a surrounding triple is well-defined, in order to identify $\ov\omega$ with an equivalence class of surrounding pairs up to surrounding triples.
\begin{lemma}\label{liftingtriples}
    If $\ov\alpha,\ov\beta,\ov\gamma\in\Cssdt$ form a surrounding triple then they pairwise surround the same $\ov\omega$. Therefore, if two pairs are connected by a sequence of surrounding triples, then they surround the same $\ov\omega$.
\end{lemma}

\begin{proof}
Let $\ov T,\ov T',\ov T''$ be three heptagons for the three pairs. Lift them to three heptagons $T,T',T''$ in such a way that $T$ and $T'$ share some lift $\beta$ and $T'$ and $T''$ share some lift $\gamma$. Let $\alpha\in T$ and $\alpha''\in T''$ be the lift of $\ov\alpha$, and let $\omega,\omega'\,\omega''$ be the three curves surrounded by the three pairs, as in Figure \ref{fig:liftingtriples}.
\begin{figure}[htp]
    \centering
    \includegraphics[width=0.75\textwidth]{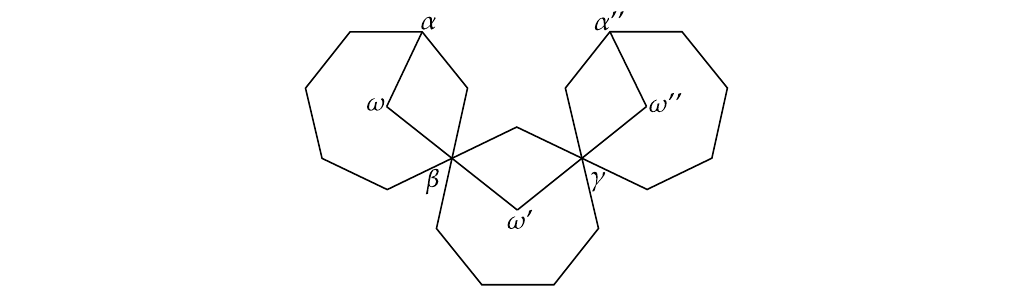}
    \caption{The three heptagons from the proof.}
    \label{fig:liftingtriples}
\end{figure}\\ 
Arguing as in Lemma \ref{surrpair} one can glue $\alpha$ to $\alpha''$ while preserving the fact that $\omega$ is the curve surrounded by $\alpha$ and $\beta$, and similarly for $\omega'$ and $\omega''$. Now $\alpha,\beta,\gamma$ are pairwise surrounding pairs. Moreover they cannot lie in the link of some other curve $\delta\in \Css$, because otherwise their projections would be at distance $1$ from $\ov\delta$. This implies that $\alpha,\beta,\gamma$ form a surrounding triple, and therefore $\omega=\omega'=\omega''$.
\end{proof}

\begin{corollary}
    There is a bijective correspondence between minimal vertices $\ov\omega\in \Cdt$ and equivalence classes of surrounding pairs, up to surrounding triples.
\end{corollary}

The only thing left to do is the following straightforward observation, that shows that automorphisms of $\Cssdt$ preserve disjointness between equivalence classes of surrounding pairs:
\begin{lemma}\label{disjointness}
    Two minimal vertices $\ov\omega, \ov\omega'$ are disjoint if and only if there are two disjoint vertices $\ov\alpha$ surrounding $\ov\omega$ and $\ov\alpha'$ surrounding $\ov\omega'$. A minimal vertex $\ov\omega$ is disjoint from a vertex $\ov\beta$ if and only if either $\ov\beta$ surrounds $\ov\omega$ or there is some $\ov\alpha$ surrounding $\ov\omega$ and disjoint from $\ov\beta$.
\end{lemma}
We now know how to intrinsically recover $\Cdt$ from $\Cssdt$. Thus we get the following, which is the $b=7$ case of Theorem \ref{csstoc}:
\begin{corollary}
    If $b=7$, for all large multiples $K$ every automorphism of $\Cssdt$ extends to an automorphism of $\Cdt$.
\end{corollary}

\subsection{At least 9 punctures}
We move on to the general case in which our sphere has at least $9$ punctures, temporarily ignoring the case $S_8$ which, as will become clear, is different in nature. First we need to identify the vertices corresponding to classes of $1$-separating curves inside $\Cssdt$. 

Recall that we denote by $G^*$ the dual of a graph $G$.
\begin{definition}
    Let $\Gamma$ be (a subgraph of) either $\C$ or $\Cdt$. A vertex $x$ \emph{splits} $\Gamma$ if $\left(\link_\Gamma(x)\right)^*$ is connected.
\end{definition}
Notice that we can recognise $\Cp$ inside $\Css$. More precisely, if $\delta\in\Css$ then $\delta\in\Cp$ if and only if $\delta$ does not split $\Css$. 
Now we want to establish the same result for $\Cssdt$.
\begin{lemma}\label{cpdtcomesplit}
A vertex $\overline\delta\in\Cssdt$ belongs to $\Cpdt$ if and only if it does not split $\Cssdt$.
\end{lemma}

Lemma \ref{cpdtcomesplit} is implied by the following:

\begin{lemma}\label{soprassesotto}
    Let $\delta\in \Css$ be a curve and let $\alpha,\beta\in \link(\delta)$. Let $\ov\alpha,\ov\beta,\ov \delta$ be their projections. Then $\delta$ separates $\alpha$ and $\beta$ if and only if $\ov\delta$ separates $\ov\alpha$ and $\ov\beta$.
\end{lemma}

\begin{proof}
Suppose that $\ov\alpha,\ov\beta$ are not separated by $\ov\delta$. Pick a  chain $\ov\gamma_0=\ov\alpha,\ov\gamma_1,\ldots,\ov\gamma_k=\ov\beta\in \link(\ov\delta)$ and choose some lifts $\gamma_0=\alpha,\gamma_1,\ldots,\gamma_{k}=\beta\in \link(\delta)$. Notice that, since the projection map is $1$-Lipschitz, these curves form a chain themselves, which means that $\delta$ does not separate $\alpha$ and $\beta$.\\
Conversely, suppose that $\alpha,\beta$ lie on the same subsurface $\Sigma$ cut out by $\delta$. This subsurface must contain one of the disks that $\alpha$ cuts out, call it $D_\alpha$, but cannot coincide with it. Therefore there must be some puncture that belongs to $\Sigma$ but not to $D_\alpha$. The same argument works for $\beta$ and some disk $D_\beta$. Hence it is always possible to find a strongly separating curve $\gamma\in \link(\delta)$ which intersects both $\alpha$ and $\beta$ and induces a different puncture separation. More precisely, one can always find some disk $D_\gamma\subset \Sigma$ with at least three punctures and which contains (at least) one of the punctures not in $D_\alpha$ and (at least) one of the punctures not in $D_\beta$, as in Figure \ref{fig:cuttingcurve}. Then $\ov\gamma$ must be at distance at least $2$ from both $\ov\alpha$ and $\ov\beta$ by Lemma \ref{puntsepdiverse}, and therefore $\ov\delta$ does not separate $\ov\alpha$ and $\ov\beta$. 
\end{proof}

\begin{figure}[htp]
    \centering
    \includegraphics[width=\textwidth]{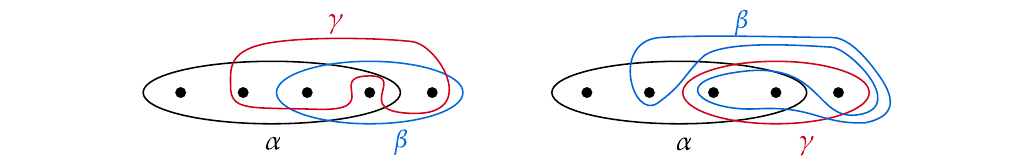}
    \caption{The shape of a possible curve $\gamma$ depends on whether one of the disks contains the punctures of the other.}
    \label{fig:cuttingcurve}
\end{figure}

Now we can exploit the previous results to intrinsically determine if two curves belong to a \emph{peripheral} $S_7$, that is, an $S_7$ cut out by a single curve. Here is where we need the number of punctures to be at least $9$, since in $S_8$ there is no strongly separating curve that cuts out some $S_7$.
\begin{lemma}[Bizarre simplices]\label{deltabizarreadventure}
    Let $b\ge9$. If $\Delta=(\delta_3,\ldots\delta_{b-6})\subseteq\Css$ is an ordered simplex of dimension $b-9$ such that:
    \begin{enumerate}
        \item $\delta_3\in \Cp$;
        \item every $\delta_i$ separates $\delta_{i-1}$ from $\delta_{i+1}$;
        \item there is no strongly separating curve that separates two consecutive curves $\delta_i$ and $\delta_{i+1}$;
    \end{enumerate}
    then $\link(\Delta)$ fills a peripheral $S_7$ cut out by $\delta_{b-6}$.
\end{lemma}

\begin{proof}
The second property tells us that $\Delta$ cuts the surface into two disks and some punctured annuli between consecutive curves, as in Figure \ref{fig:bizarre}. Moreover each of these annuli should contain just one puncture, because otherwise we could find some curve that contradicts the third property. Then, since the $S_{0,4}$ cut out by $\delta_3$ cannot contain any strongly separating curve, $\link(\Delta)$ must fill the other peripheral disk, call it $\Sigma$. Now it is enough to count the punctures in $\Sigma$, that must be six ($\delta_3$ cuts out three punctures, and each of the $b-9$ annuli contains a single puncture).
\end{proof}

\begin{figure}[htp]
    \centering
    \includegraphics[width=0.75\textwidth]{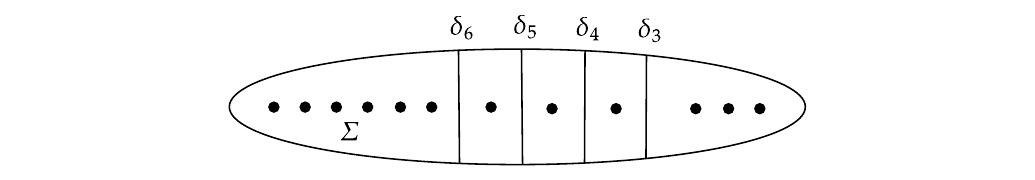}
    \caption{Starting from $\delta_3$, $\Delta$ is constructed by consecutively cutting out once-punctured annuli.}
    \label{fig:bizarre}
\end{figure}

Notice that these bizarre properties are stated only in terms of some vertices separating some others, and are therefore preserved when passing to the quotient and taking lifts by Lemma \ref{soprassesotto}. Hence we get the following:

\begin{corollary}
    Let $b\ge9$. If $\ov\Delta=(\ov\delta_3,\ldots\ov\delta_{b-6})\subseteq \Cssdt$ is an ordered simplex of dimension $b-9$ such that:
    \begin{enumerate}
        \item $\ov\delta_3\in \Cpdt$;
        \item every $\ov\delta_i$ separates $\ov\delta_{i-1}$ from $\ov\delta_{i+1}$;
        \item there is no $\ov\gamma\in\Cssdt$ that separates two consecutive $\ov\delta_i$ and $\ov\delta_{i+1}$;
    \end{enumerate}
    then every vertex $\ov\alpha\in \link(\ov\Delta)$ lifts inside a peripheral $S_7$ cut out by some lift $\delta_{b-6}$ of $\ov\delta_{b-6}$.
\end{corollary}

Now, given some peripheral $S_7$ and a bizarre simplex $\Delta$ for it, we are able to recognise $\mathcal{C}^{ss}(S_7)$, which is the subgraph of $\Css$ spanned by those curves $\gamma\in \link(\Delta)$ which do not cut out a pair of pants with the boundary of $S_7$, or equivalently for which there exists some other curve $\gamma'\in \link(\Delta)$ that separates $\gamma$ from $\delta_{b-6}$. Again, this property just involves some vertices separating some others, thus it holds in the quotient if and only if it holds in $\Css$. Therefore, if $\pi:\Css\to\Cssdt$ is the quotient projection, we can recognise $\pi(\mathcal{C}^{ss}(S_7))$ inside $\Cssdt$ in the same way. Notice that, as discussed in Corollary \ref{CUdt},  $\pi(\mathcal{C}^{ss}(S_7))$ is isomorphic to $\Css(S_7)/DT_K(S_7)$, therefore we are exactly in the framework of the previous subsection. Then we can define surrounding pairs and triples as in Figures \ref{fig:surrpair} and \ref{fig:surrtriple}. The following is a summary of \cite[Lemmas 6.1 and 6.3]{ssgraph}:
\begin{lemma}
    Let $b\ge9$. Two curves $\alpha,\beta\in \Css$ form a surrounding pair if and only if:
    \begin{itemize}
        \item They belong to $\Cp$;
        \item They belong to some peripheral $S_7$;
        \item They are a surrounding pair intrinsically inside $S_7$.
    \end{itemize}
    Three curves $\alpha,\beta,\gamma\in \Css$ form a surrounding triple if they are pairwise surrounding pairs and they all belong to some peripheral $S_6$ (which is defined exactly as a peripheral $S_7$, but this time the bizarre simplex has one more vertex).
\end{lemma}
Now we can almost repeat what we did for the case $b=7$, being careful to remember in which link we are. All proofs will have the same flavour as the corresponding ones, though they will involve more cut set arguments. 
\begin{itemize}
    \item First notice that, if $\ov T$ is a heptagon which lies in $\link_{\Cssdt} (\ov \Delta)$ for some simplex $\ov \Delta$, then $\ov T\star\ov\Delta$ is a generalised heptagon, thus we can find lifts $\Delta$ and $T\subset\link_{\Css}(\Delta)$ by Theorem \ref{mgonlift}.
    \item To prove the equivalent of Lemma \ref{surrpair} we must show that the vertex $\ov\omega$ surrounded by a pair is also independent of the chosen simplex $\ov\Delta$. Thus we take simplices $\ov\Delta,\ov\Delta'$ and heptagons $\ov T,\ov T'$ inside the respective links, and lift them to $T\in \link(\Delta)$ and $T'\in \link(\Delta')$. Then the proof follows the same steps, being careful to add $\Delta$ to the cut set containing $\omega$ (and similarly for $\Delta'$).
    \item It is clear that surrounding pairs (triples) project to surrounding pairs (triples) since the bizarre properties are preserved in the quotient. Then Lemma \ref{projtriples} follows.
    \item Lemma \ref{disjointness} deals with disjointness, and the same statement works in our case. 
\end{itemize}
We are left to prove an analogous to Lemma \ref{liftingtriples}, which needs a bit more care. 
\begin{lemma}\label{triple9}
    Let $b\ge9$. If $\ov\alpha,\ov\beta,\ov\gamma\in\Cssdt$ form a surrounding triple then they pairwise surround the same $\ov\omega$.
\end{lemma}
\begin{proof}
    As in the proof of Lemma \ref{liftingtriples} we actually want to show that a surrounding triple in $\Cssdt$ lifts to a surrounding triple in $\Css$. Since $\ov\alpha,\ov\beta$ are a surrounding pair they belong to some heptagon $\ov T$ inside the link of some simplex $\ov\Delta$, which corresponds to some peripheral $S_7$ (meaning that it lifts to a pants decomposition for the complement of a peripheral $S_7$). We can similarly find $\ov T',\ov \Delta'$ for $\ov\beta,\ov\gamma$ and $\ov T'',\ov \Delta''$ for $\ov\alpha,\ov\gamma$. Moreover, let $\ov \sigma$ be a simplex which corresponds to some peripheral $S_6$ and whose link contains $\ov\alpha,\ov\beta,\ov\gamma$.
    Now, choose lifts $\Delta,\Delta',\Delta'',\sigma$ for the various simplices. Lift $\ov T$ to a heptagon $T$ inside $\link(\Delta)$, and let $\alpha\in T$ be the lift of $\ov\alpha$. If $\alpha\not\in \link(\sigma)$ let $g\in DT_K$ be such that $g(\alpha)\in \link(\sigma)$, and replace $T$ and $\Delta$ with $g(T)$ and $g(\Delta)$. Apply the same procedure to $T',\Delta'$ and $T'',\Delta''$. Finally, let $\omega$ be the curve surrounded by $\alpha,\beta$, and define similarly $\omega'$ and $\omega''$. The situation is depicted in Figure \ref{fig:triplestagliati}.
    \begin{figure}[htp]
        \centering
        \includegraphics[width=0.75\textwidth]{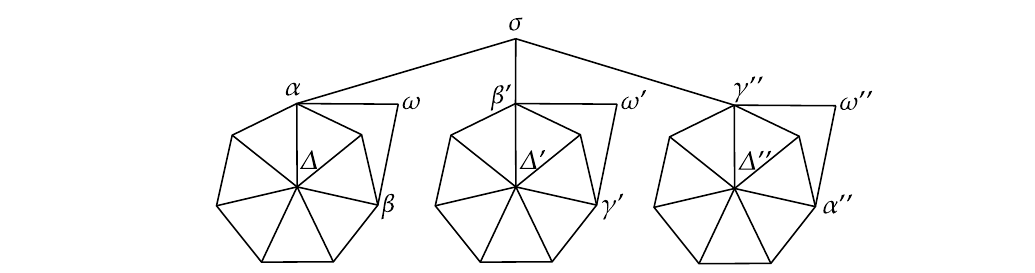}
        \caption{The various lifts from the proof of Lemma \ref{triple9}.}
        
        \label{fig:triplestagliati}
    \end{figure}

    Now let $\beta\in T$ and $\beta'\in T'$ be the corresponding lifts of $\ov\beta$, which we want to glue to each other. Let $g\in DT_K$ be an element mapping $\beta$ to $\beta'$. If $g$ is not the identity let $(s,\gamma_s)$ be as in Proposition \ref{cor3.6}. Arguing as in the proof of Lemma \ref{deltaindep} we see that $\gamma_s$ must fix pointwise one of the following:
    \begin{itemize}
        \item $\beta$;
        \item $\omega$, two curves $\delta,\varepsilon$ in $T$ and the whole $\Delta$;
        \item $\alpha$;
        \item $\sigma$;
    \end{itemize}
    Then we can apply $\gamma_s$ to everything beyond the cut set, while preserving that all $\omega$s are still the minimal curves surrounded by the corresponding pairs as in Lemma \ref{surrpair}, and proceed by induction.

    Thus we can glue $\beta$ to $\beta'$ and with the exact same argument we can glue $\gamma'$ to $\gamma''$. Therefore, our data up to this point span a graph which schematically looks like in Figure \ref{fig:triplesglued}.
    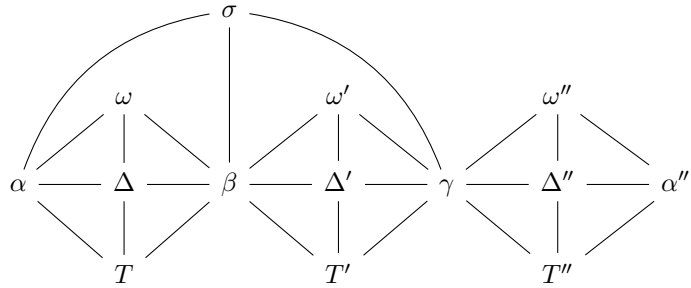
\begin{figure}[htp]
        \centering
       \begin{tikzcd}
            &&\sigma\ar[ddrr, no head, bend left=30] \ar[ddll, no head, bend right=30]\ar[dd,no head]&&&&\\
            &\omega\ar[dr,no head]\ar[d,no head]&&\omega'\ar[dr,no head]\ar[d,no head]&&\omega''\ar[dr,no head]\ar[d,no head]&\\
            \alpha\ar[r,no head]\ar[ur, no head]\ar[dr, no head]&\Delta\ar[r,no head]&\beta\ar[r,no head]\ar[ur, no head]\ar[dr, no head]&\Delta'\ar[r,no head]&\gamma\ar[r,no head]\ar[ur, no head]\ar[dr, no head]&\Delta''\ar[r,no head]&\alpha''\\
            &T\ar[ur,no head]\ar[u,no head]&&T'\ar[ur,no head]\ar[u,no head]&&T''\ar[ur,no head]\ar[u,no head]&
       \end{tikzcd}
        \caption{This graph represents all possible "paths" from $\alpha$ to $\alpha''$, and elements on the same column correspond to the same cut set. Notice that every heptagon actually represents two paths.}
        \label{fig:triplesglued}
    \end{figure}\\
    We are left to glue $\alpha$ to $\alpha''$, since then $\alpha,\beta,\gamma$ will be a surrounding triple and therefore surround the same minimal curve, which in turn will mean that $\omega=\omega'=\omega''$. Let $g\in DT_K$ be an element mapping $\alpha$ to $\alpha''$. If $g$ is not the identity let $(s,\gamma_s)$ be as in Proposition \ref{cor3.6}. Then $\gamma_s$ must fix pointwise one of the following cut sets:
    \begin{itemize}
        \item $\alpha$;
        \item $\sigma$, $\omega$, two curves $\delta,\varepsilon$ in $T$ and the whole $\Delta$;
        \item $\sigma$ and $\beta$;
        \item $\sigma$, $\omega'$, two curves $\delta',\varepsilon'$ in $T'$ and the whole $\Delta'$;
      
        \item $\gamma$;
        \item $\omega''$, two curves $\delta'',\varepsilon''$ in $T''$ and the whole $\Delta''$. 
    \end{itemize}
    Thus, as argued in Lemma \ref{surrpair}, we can apply $\gamma_s$ beyond the cut set, while preserving that every minimal curve is the one surrounded by the corresponding pair, and conclude by induction.
\end{proof}
To sum up we get the following, which is the case $b\ge9$ of Theorem \ref{csstoc}:
\begin{corollary}
    For every $b\ge9$ and all large multiples $K$, every automorphism of $\Cssdt$ extends to an automorphism of $\Cdt$.
\end{corollary}

\subsection{8 punctures}
We are left to deal with the case $b=8$. The idea will be to replace heptagons with a suitable subgraph which we will use to recognise surrounding pairs. Define $\mathfrak{O}$ as the graph obtained by adding the four longest diagonal to an octagon, as in Figure \ref{fig:oct}. A possible realisation of this graph inside $\Cp(S_8)$ is given by the eight curves in the same Figure, and \cite[Lemma 7.2]{ssgraph} shows that there is only one copy of $\mathfrak{O}$ inside $\Cp(S_8)$ up to the action of the mapping class group (that is, every $\mathfrak{O}$ corresponds to some curves arranged as in the Figure).
\begin{figure}[htp]
    \centering
    \includegraphics[width=0.75\textwidth]{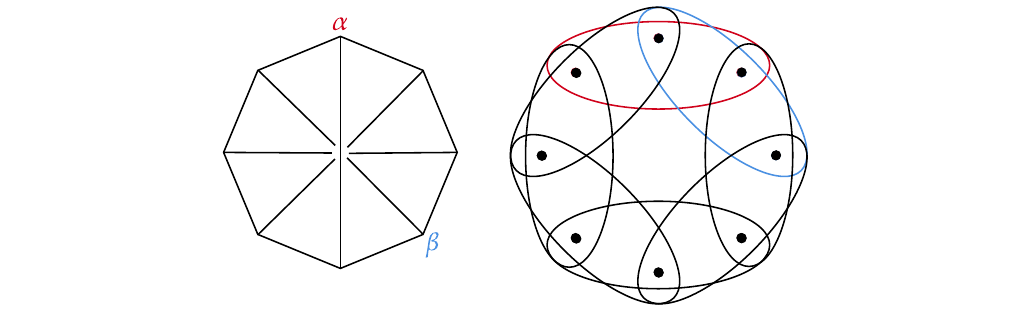}
    \caption{The graph $\mathfrak{O}$ and its realisation with $1$-separating curves, each of which surrounds three consecutive punctures. Notice that two vertices correspond to a surrounding pair if and only if they are connected by exactly two paths of length $2$ inside $\mathfrak{O}$.}
    \label{fig:oct}
\end{figure}\\
In \cite{ssgraph} it was proved that $\alpha,\beta\in \Css$ form a surrounding pair if and only if:
\begin{itemize}
    \item They belong to $\Cp$;
    \item They belong to some isometrically embedded copy of $\mathfrak{O}$ inside $\Cp$;
    \item They are connected by exactly two geodesic paths of length $2$ inside $\mathfrak{O}$.
\end{itemize}
Moreover, three curves $\alpha,\beta,\gamma\in \Css$ form a surrounding triple if and only if they are pairwise surrounding pairs and there is no $\delta\in \Css\setminus\Cp$ which is disjoint from all of them.\\
Now, in order to define surrounding pairs and triples inside $\Cssdt$ we need the following lemmas:
\begin{lemma}\label{projoct}
    For all large multiples $K$ every isometrically embedded copy of $\mathfrak{O}$ inside $\C$ projects isometrically into $\Cpdt$. In particular there exists an isometrically embedded copy of $\mathfrak{O}$ inside $\Cpdt$.
\end{lemma}

\begin{proof}
    Every two copies of $\mathfrak{O}$ differ by a mapping class, thus it suffices to argue as in Lemma \ref{projisometry} with $X_b$ replaced by $\mathfrak{O}$.
\end{proof}

\begin{lemma}
Every isometrically embedded copy $\ov{\mathfrak{O}}\subset\Cpdt$ admits an isometrically embedded lift $\mathfrak{O}\subset \Cp$. Therefore surrounding pairs lift to surrounding pairs.
\end{lemma}

\begin{proof}
It suffices to find a lift $\mathfrak{O}$, which will automatically be isometrically embedded since the projection map is $1$-Lipschitz (as argued in Lemma \ref{liftspec}). Now, it is useful to see $\mathfrak{O}$ as the $1$-skeleton of a Möbius band made of four squares, as in Figure \ref{fig:mobius}.
\begin{figure}[htp]
    \centering
    \includegraphics[width=0.75\textwidth]{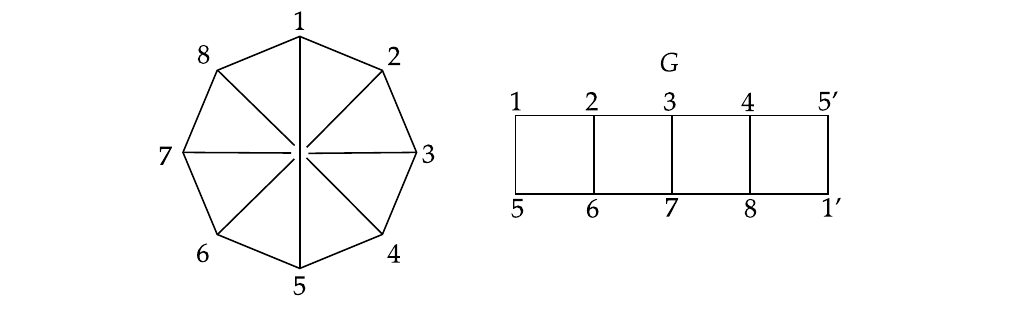}
    \caption{$\mathfrak{O}$ is obtained from the strip $G$ by gluing $1$ to $1'$ and $5$ to $5'$.}
    \label{fig:mobius}
\end{figure}\\
Clearly the graph $G$ in the Figure admits a lift, since we can lift each square and glue them together along common sides (recall that every $1$-simplex admits a unique $DT_K$-orbit of lifts). Now we are left to glue $1$ to $1'$ and $5$ to $5'$. Let $g\in DT_K$ be an element that maps the edge $1,5$ to the edge $1',5'$. If $g$ is not the identity let $(s,\gamma_s)$ be as in Lemma \ref{cor3.6}, applied to $x=1$. If $d_{\C}(s,1)\le 1$ we can apply $\gamma_s$ to the whole data and proceed by induction. Otherwise every path from $1$ to $1'$ must intersect the star of $s$. This means that there must be a square $Q$ where it is not possible to move from the left side to the right side without crossing the star of $s$. Thus there must be at least two vertices $p,q\in Q$ which lie in the star of $s$ and that cut $G$ in two connected components (more precisely, $p,q$ must be the vertices of either a diagonal or a vertical side). Thus we can apply $\gamma_s$ beyond $p$ and $q$. Notice that either $5'$ is one of $p$ and $q$ or $5'$ and $1'$ are on the same connected component cut out by $p$ and $q$; either way $\gamma_s$ is applied to the whole edge $1',5'$, and we can proceed by induction.
\end{proof}

Now we can proceed exactly as for the case $b=7$, with a few adjustments to cut sets arguments.
Given a surrounding pair $\ov\alpha,\ov\beta$, define the vertex $\ov\omega$ surrounded by the pair by lifting the pair and projecting the curve surrounded by the lift, as before.
\begin{lemma}\label{liftingpairs8}
    The vertex $\ov\omega$ surrounded by a pair $\ov\alpha,\ov\beta$ is well-defined.
\end{lemma}

\begin{proof}
Let $\ov{\mathfrak{O}}$ and $\ov{\mathfrak{O}}'$ be two copies of $\mathfrak{O}$ inside $\Cssdt$ that contain $\ov\alpha,\ov\beta$, and lift them to $\mathfrak{O}$ and $\mathfrak{O}'$. We can assume that the lifts of $\ov\alpha$ coincide, and let $\beta,\beta'$ be the lift of $\ov\beta$. Let $\gamma_1,\gamma_2,\delta_1,\delta_2$ be the curves in Figure \ref{fig:gluedoct}.
\begin{figure}[htp]
    \centering
    \includegraphics[width=0.75\textwidth]{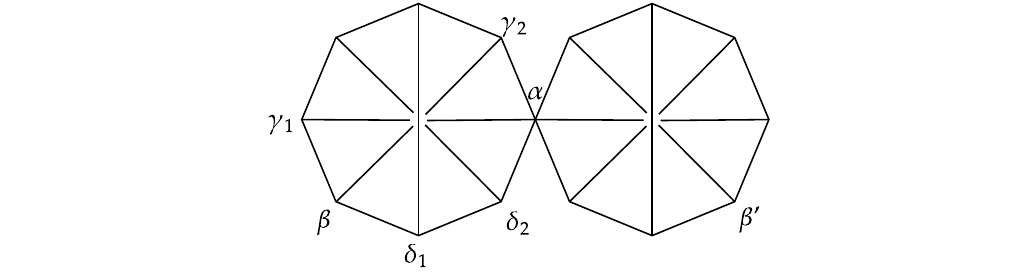}
    \caption{The two octagons from the proof of Lemma \ref{liftingpairs8}.}
    \label{fig:gluedoct}
\end{figure}\\
Let $\omega,\omega'$ be the curves surrounded by the two pairs, and let $g\in DT_K$ be an element that maps $\beta$ to $\beta'$. If $g$ is not the identity let $(s,\gamma_s)$ be as in Proposition \ref{cor3.6}. If $d_\C(\beta,s)\le 1$ we can apply $\gamma_s$ to the whole data and proceed by induction on the complexity. If $d_\C(\alpha,s)\le 1$ we can apply $\gamma_s$ just to the second octagon and again proceed by induction. If none of the previous hold we can assume without loss of generality that $d_s(\alpha,\beta)$ is large, and therefore every path from $\beta$ to $\alpha$ must pass through the star of $s$. In order to understand the following cut set argument we represent $\omega$ and part of the octagon as in Figure \ref{fig:data}.

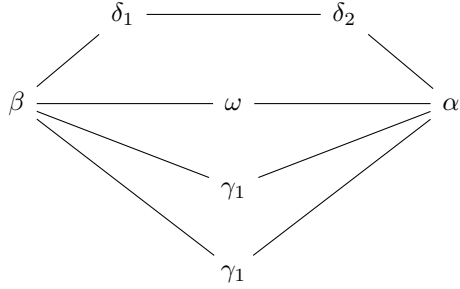
\begin{figure}[htp]
    \centering
    \begin{tikzcd}
    &\delta_1\ar[rr,no head]& &\delta_2& \\
    \beta \ar[ur,no head]\ar[drr,no head]\ar[ddrr,no head]\ar[rr,no head]&& \omega\ar[rr,no head] && \alpha\ar[ul,no head]\ar[dll,no head]\ar[ddll,no head] \\
    && \gamma_1 && \\
    && \gamma_1 && \\
    \end{tikzcd}
    \caption{A schematic representation of $\{\omega\}\cup\link_{\mathfrak{O}}(\alpha)\cup\link_{\mathfrak{O}}(\beta)$ from Lemma \ref{liftingpairs8}.}
    \label{fig:data}
\end{figure}

We already see that $\gamma_s$ must fix $\omega,\gamma_1,\gamma_2$ and one between $\delta_1$ and $\delta_2$. In the first case $\gamma_s$ fixes $\link_{\mathfrak{O}}(\beta)=\{\gamma_1,\gamma_2,\delta_1\}$, which is a cut set for $\mathfrak{O}$ because it is the link of a vertex. Thus we can apply $\gamma_s$ to all curves but $\beta$. In the second case $\gamma_s$ fixes $\link_{\mathfrak{O}}(\alpha)=\{\gamma_1,\gamma_2,\delta_2\}$, and we can apply $\gamma_s$ to $\mathfrak{O}'\cup \link(\alpha)$. In both cases $\omega$ remains the curve surrounded by $\beta$ and $\gamma_s(\alpha)$, as argued in Lemma \ref{surrpair}; therefore we can proceed by induction.
\end{proof}
The rest of the argument is as for the case $b=7$.
\begin{itemize}
    \item Arguing exactly as in Lemma \ref{projtriples} we see that surrounding triples project to surrounding triples.
    \item Adapting the proof of Lemma \ref{liftingtriples} with the cut set arguments from Lemma \ref{liftingpairs8} we get that the vertex $\ov\omega$ surrounded by a surrounding triple is well-defined.
    \item Again, the conclusion of Lemma \ref{disjointness} holds, so we can recognise disjointness.
\end{itemize}
Thus we get the final piece of the proof of Theorem \ref{csstoc}:
\begin{corollary}
    If $b=8$, for all large multiples $K$ every automorphism of $\Cssdt$ extends to an automorphism of $\Cdt$.
\end{corollary}

\section{Extracting combinatorial data from a quasi-isometry}\label{sec:extract}
The main goal of this section is, roughly speaking, to show that quasi-isometries of $MCG(S_b)/DT_K$ induce automorphisms of a certain graph, which we will later on relate to $\C/DT_K$. This strategy is inspired by \cite[Section 5]{quasiflats}, where an approach to the quasi-isometric rigidity of mapping class groups is presented as a template to prove quasi-isometric rigidity of other \emph{hierarchically hyperbolic spaces} (HHS), $MCG(S_b)/DT_K$ being a HHS (for suitable $K$) \cite{BHMS}. More precisely, \cite[Theorem 5.7]{quasiflats} states that quasi-isometries of a HHS satisfying three additional assumptions induce automorphisms of a suitable graph. The assumptions are not satisfied by $MCG(S_b)/DT_K$ but nevertheless we will show that a very similar statement applies to that case, see Theorem \ref{hingesauto}, which is the main result of this section. This theorem in fact applies to any HHS satisfying certain assumptions (different from those of \cite{quasiflats}), other examples of such HHSs being pants graphs.

We will begin by recalling all facts about HHSs that we will need.

\subsection{HHS background}

A hierarchically hyperbolic space is a metric space $X$ that comes with certain additional data, most importantly a family of uniformly hyperbolic spaces $\{\C(Y)\}_{Y\in \mathfrak S}$, called \emph{coordinate spaces}, and uniformly coarsely Lipschitz maps $\pi_Y:X\to \C(Y)$. For mapping class groups, these are curve graphs of subsurfaces and maps coming from subsurface projections. Moreover, the index set $\mathfrak S$ of the set of hyperbolic spaces has two relations, \emph{orthogonality} and \emph{nesting}, denoted $\orth$ and $\nest$ respectively. For mapping class groups, these are containment and disjointness of subsurfaces (up to isotopy). When two elements $U,V$ of $\mathfrak S$ are not $\nest$- nor $\orth$-compatible, one says that they are \emph{transverse}. It is customary to denote a HHS simply by $(X,\mathfrak S)$.

A crucial fact about HHSs is that they satisfy a \emph{distance formula}. This means that for all sufficiently large threshold $s$ there exists $D$ such that for all $x,y\in X$ we have
$$d_X(x,y)\asymp_{D,D} \sum \{\{d_{Y}(x,y)\}\}_s,$$
where
\begin{itemize}
\item $\asymp_{D,D}$ denotes equality up to multiplicative and additive error at most $D$,
\item $\{\{A\}\}_s=A$ if $A\geq s$ and $\{\{A\}\}_s=0$ otherwise,
\item $d_{Y}(x,y)$ is shorthand for $d_{\C(Y)}(\pi_Y(x),\pi_Y(y))$.
\end{itemize}

The idea of orthogonality is that it corresponds to products, in the following sense. Given any $U\in \mathfrak S$, there is a corresponding space $F_U$ associated to it, which is quasi-isometrically embedded in $X$ ($F_U$ is a HHS itself with index set $\mathfrak S_U=\{Y\in \mathfrak S: Y\nest U\}$, and in mapping class groups these essentially correspond to mapping class groups of subsurfaces). Given a maximal set $\{U_i\}$ of pairwise orthogonal elements of $\mathfrak S$, there is a corresponding \emph{standard product region} $P_{\{U_i\}}$ which is quasi-isometric to the product of the $F_{U_i}$ (think of a Dehn twist flat as, coarsely, a product of annular curve graphs). The product region $P_{\{U_i\}}$ has the property that for all $U$ with $U_i\propnest U$ or $U\transverse U_i$ for some $i$, we have that $\pi_Y(P_{\{U_i\}})$ is a uniformly bounded set, that we denote $\rho^{\{U_i\}}_V$ (usually, this notation is used for a different bounded set at finite Hausdorff distance from the one defined here; this will not matter for us). Product regions can also be associated to non-maximal sets of pairwise orthogonal elements of $\mathfrak S$, but in this case there is an additional factor than the $F_{U_i}$; this will not be important for us.

The following is \cite[Corollary 1.28]{quasiflats} (which is stated for product regions corresponding to one element of the index set, but the proof does not use this):

\begin{lemma}\label{dxP}
For all sufficiently large $s$ there exists $D$ such that, if $\{U_i\}$ is a maximal collection of pairwise orthogonal indices and $x\in X$, then
$$d_X(x,P_{\{U_i\}})\asymp_{D,D}\sum_{W} \left\{\left\{d_W\left(x,\rho^{\{U_i\}}_W\right)\right\}\right\}_s$$
where the sum is made on all $W$ such that $U_i\propnest W$ or $U_i\transverse W$ for some $U_i$.
\end{lemma}

\begin{corollary}[Intersection of different product regions]\label{prodint}
If $\{U_i\}$ and $\{V_j\}$ are maximal collections of pairwise orthogonal indices, such that $\{U_i\}\cap\{V_j\}=\emptyset$ and every $V_j$ is $\nest$-minimal, then $P_{\{U_i\}}\Tilde\cap P_{\{V_j\}}$ is coarsely a point.
\end{corollary}

\begin{proof}
It is enough to apply Lemma \ref{dxP} with $x\in P_{\{V_j\}}$. In fact, every $W\not\in\{V_j\}$ contributes to the sum with a term that does not depend on $x$, but only on $\{U_i,V_j\}$ and $W$. Moreover, there is a finite number of terms that contribute to the sum, so if we choose a big enough threshold $s$ we get that 
$$d_X(x,P_{\{U_i\}})\asymp_{D,D} \sum_{j}\left\{\left\{d_{V_j}\left(x,\rho^{\{U_i\}}_{V_j}\right)\right\}\right\}_s $$
Therefore all points $x\in  P_{\{V_j\}}$ that lie within a given constant of $P_{\{U_i\}}$ have nearby projections to all $\C V_j$, and hence they are close to each other, as required. 
\end{proof}

Finally, we will be interested in quasiflats in HHSs and, related to quasiflats, orthants. In $X$ and the $F_U$s there are special (uniform) quasi-geodesics, called \emph{hierarchy paths}, which are those that project monotonically (with uniform constants) to all $\C(Y)$. Similarly, there are hierarchy rays and hierarchy lines (which are quasigeodesic rays and lines, respectively).  Given a product region $P_{\{U_i\}}$ and hierarchy lines (resp., rays) in some of the $F_{U_i}$, the product region contains a product of those (where for the indices $i$ for which a line/ray has not been assigned one chooses a point in $F_{U_i}$ instead). This is what we will refer to as \emph{standard $k$-flats (resp., orthants)}, where $k$ is the number of lines/rays. The \emph{support} of a standard $k$-flat (resp., orthant) is the set of $U_i$ for which a line/ray has been assigned. Related to this, a \emph{complete support set} is a subset $\{U_i\}_{i=1}^\nu\subseteq \mathfrak S$ of pairwise orthogonal indices with all $\C(U_i)$ unbounded, and with maximal possible cardinality $\nu$ among sets with these properties. This definition is relevant here as $\nu$ is then the maximal dimension of standard orthants (if the $\C(U_i)$ are not only unbounded, but also have non-empty Gromov boundary, as will be the case for us). Moreover, in this case points in the Gromov boundary $\partial C(U_i)$ each determine a hierarchy ray, and similarly pairs of points determine hierarchy lines. In particular, for every complete support set $\{U_i\}_i=1^\nu$ and for every choice of distinct points $p_i\in\partial\C U_i$ there exists an associated standard flat $\mathfrak F_{\{(U_i,p_i^\pm)\}}$, which is the product of the hierarchy lines in $F_{U_i}$ whose projections to $\C U_i$ have endpoints $p_i^\pm$.

\subsection{The hinge graph}

In this section we discuss the graph that quasi-isometries will induce automorphisms of. We will call it the hinge graph, and it is best thought of, for the purposes of this section, as encoding standard orthants, as well as their coarse intersections. Here, recall that given subsets $A,B$ of a metric space $X$, the coarse intersection $A\Tilde{\cap} B$, if well-defined, is a subspace of $X$ within bounded Hausdorff distance of all $N_R(A)\cap N_R(B)$ for $R$ sufficiently large.

\begin{remark}\label{stdint}
The coarse intersection of any two standard flats $\mathfrak F_1,\mathfrak F_2$ is well-defined by \cite[Lemma 4.12]{quasiflats}. In fact, the lemma gives a description of the coarse intersection as a subspace $F$ of the HHS whose projection to any hyperbolic space $U$ is the coarse intersection of $\pi_U(\mathfrak F_1)$ and $\pi_U(\mathfrak F_2)$. For example, the coarse intersection is a standard $1$-flat provided that all coarse intersections in the various hyperbolic spaces are bounded, except for one hyperbolic space where $\pi_U(\mathfrak F_1)$ and $\pi_U(\mathfrak F_2)$ are both quasilines with the same endpoints in the Gromov boundary.
\end{remark}

\begin{definition}[Hinge graph] As in \cite[Definition 5.2]{quasiflats}, let $\bf{Hinge}(\mathfrak{S})$ be the set of \emph{hinges}, that is, pairs $(U,p)$ where $U$ is contained in a complete support set and $p\in\partial C U$. We say that two hinges $(U,p)$ and $(V,q)$ are compatible if they are orthogonal and there exists some complete support set that contains both. Then we give $\bf{Hinge}(\mathfrak{S})$ a metric graph structure by declaring two hinges to be adjacent if and only if they are compatible.
\end{definition}

As in \cite[Definition 5.3]{quasiflats}, one can associate to a hinge $\sigma=(U,p)$ a standard $1$-orthant, denoted $h_\sigma$. This is a hierarchy ray whose projection to $\C(U)$ is a quasigeodesic ray asymptotic to $p$, and whose projections to all other coordinate spaces are uniformly bounded. Another property of $h_\sigma$, stated in \cite[Remark 5.4]{quasiflats}, is that if $\sigma\neq\sigma'$ are two different hinges then $d_{Haus}(h_\sigma,h_{\sigma'})=\infty$.

\begin{lemma}\label{preciseint}
Let $(X,\mathfrak{S})$ be a HHS. If for every $U\in\mathfrak{S}$ either the space $\C U$ is bounded or $|\partial \C U|\ge 4$ then:
\begin{enumerate}
\item For every hinge $\sigma=(U,p)$ there exist two standard flats $\mathfrak F_1,\,\mathfrak F_2$ whose coarse intersection is supported in $U$, is coarsely a standard $1$-flat, and contains $h_\sigma$.
\item For any two compatible hinges $\sigma=(U,p)$ and $\sigma'=(V,q)$ there exist two standard flats $\mathfrak F_1,\,\mathfrak F_2$ whose coarse intersection is a standard $2$-flat supported in $\{U,V\}$ and containing $h_\sigma$ and $h_{\sigma'}$.
\end{enumerate}
\end{lemma}

This lemma will replace the three assumptions from \cite[Section 5]{quasiflats}.

\begin{proof}
\begin{enumerate}
\item Choose another ideal point $q\in \partial \C U\setminus\{p\}$, and let $\gamma$ be a hierarchy path in $F_U$ whose projection to $\C U$ has limit points $p$ and $q$. Then pick a complete support set $\{U_i\}$ completing $U=U_1$, and for each $i\ge2$ choose four distinct ideal points $p_i^\pm, q_i^\pm\in \C U_i$. Let $\gamma_i, \delta_i$ be two bi-infinite hierarchy paths in each $F_{U_i}$ whose projection to $\C U_i$ have limit points, respectively, $p_i^\pm$ and $q_i^\pm$. Then $\mathfrak F_1= \gamma\times \prod_i\gamma_i$ and $\mathfrak F_1= \gamma\times \prod_i\delta_i$ have the required coarse intersection in view of the discussion in Remark \ref{stdint}.
\item The proof is very similar, in this case fixing two hierarchy lines rather than one.
\end{enumerate}
\end{proof}

\begin{remark}[Hidden hypothesis: asymphoricity]\label{rem:asymporicity}
    In the following theorem the HHSs will be required to be \emph{asymphoric}, as in \cite[Definition 1.14]{quasiflats}. For our purposes, it suffices to know that asymphoric HHS include all hierarchically hyperbolic groups, such as our quotients of interest $MCG/DT_K$ (see Remark \ref{rem:prop_hhs_mcgdtn}).
%
\end{remark}

\begin{theorem}\label{hingesauto}
Let $(X,\mathfrak{S})$, $(Y,\mathfrak{T})$ be asymphoric HHSs such that for every $U\in\mathfrak{S}$ either the space $\C U$ is bounded or $|\partial \C U|\ge 4$, and similarly for every $V\in\mathfrak{T}$. Then every quasi-isometry $f:\,X\to Y$ induces an isomorphism $f_{hin}$ between the hinge graphs, such that for all $\sigma\in\bf{Hinge}(\mathfrak{S})$ we have $d_{Haus}(h_{f_{hin}(\sigma)}, f(h_\sigma))<\infty$.
\end{theorem}

\begin{proof}
This is a very similar statement to \cite[Theorem 5.7]{quasiflats}, which applies to HHSs satisfying three additional assumptions (not satisfied in our case). An inspection of the proof shows that each use of the three assumptions can be replaced by our Lemma \ref{preciseint}.  
\end{proof}

The following is the analogue of \cite[Lemma 5.9]{quasiflats}, which is a corollary of \cite[Theorem 5.7]{quasiflats} and whose proof does not use the additional assumptions on the HHSs further. Given a quasi-isometry between HHSs, the lemma gives a condition for the image of a standard flat to lie within bounded Hausdorff distance of a standard flat in terms of the induced map $f_{hin}$ on hinge graphs. Crucially for applications, the bound is only in terms of the HHSs and the quasi-isometry constants.

\begin{lemma}[Flats go to flats]\label{flatstoflats}
Let $(X,\mathfrak{S})$, $(Y,\mathfrak{T})$ be asymphoric HHSs, such that for every $U\in\mathfrak{S}$ either the space $\C U$ is bounded or $|\partial \C U|\ge 4$, and similarly for every $V\in\mathfrak{T}$.  Let $f:\,X\to Y$ be a quasi-isometry. There exists a constant $C_0$, depending only on the quasi-isometry constants of $f$, with the following property. Let $\{U_i\}_{i=1}^\nu\subseteq\mathfrak{S}$ be a complete support set, and let $p_i^\pm$ be distinct points in $\partial \C U_i$. Suppose that there exist a complete support set $\{V_i\}_{i=1}^\nu\subseteq\mathfrak{T}$ and distinct points $q_i^\pm\in\partial \C V_i$ such that for every $i=1,\ldots, \nu$ we have $f_{hin}(U_i,p_i^\pm)=(V_i,q_i^\pm)$. Then, $d_{Haus}(f(\mathfrak F_{\{(U_i,p_i^\pm)\}}),  \mathfrak F_{\{(V_i,q_i^\pm)\}})\le C_0$.
\end{lemma}

\section{Quasi-isometric rigidity}\label{sec:qi} This Section is devoted to the proof of Theorem \ref{qimcgdtn}. For the rest of the Section, let $S=S_b$ be a sphere with $b\ge 7$. Let us start from the HHS structure of $MCG^\pm/DT_K$.
\begin{definition}
For $\Delta$, $\Delta'$ simplices of some graph $G$, we write $\Delta\sim\Delta'$
to mean $\link(\Delta)=\link(\Delta')$. We denote by $[\Delta]$ the $\sim$–equivalence class of $\Delta$, and by $\mathfrak{S}$ the set of equivalence classes of non-maximal simplices. Finally, we define the saturation of $\Delta$ as the set of vertices $v\in G$ for which there exists a simplex $\Delta'$ such that $v \in \Delta'$ and $\Delta'\sim\Delta$, i.e.
$$\text{Sat}(\Delta)=\left(\bigcup_{\Delta'\in[\Delta]}\Delta'\right)^{(0)}.$$
\end{definition}

\begin{remark}\label{rem:prop_hhs_mcgdtn}
By \cite[Theorem 7.1]{BHMS}, $MCG/DT_K$ is a hierarchically hyperbolic \emph{group}, that is, it has a "nice" proper and cocompact action on some hierarchically hyperbolic space from which it inherits a HHS structure (see \cite[Definition 1.21]{HHSII}). Moreover, combining \cite[Theorem 7.1 and Proposition 8.13]{BHMS}, we see that the HHS structure has the following properties:
\begin{enumerate}[label=(\roman*)]
    \item The index set is the set $\mathfrak{S}$ of $\sim$–equivalence classes of simplices inside $\Cdt$;
    \item There is a bijection $j:\,\mathfrak{S}\to \mathfrak{S}^{\ge 1}/DT_K$, where  $\mathfrak{S}^{\ge 1}$ is the set of essential, non-annular, possibly disconnected subsurfaces $U\subseteq S$.
    \item Two elements $[\overline\Delta],[\overline\Sigma]\in\mathfrak{S}$ are orthogonal (resp. nested) if $j\left([\overline\Delta]\right)$ and $j\left([\overline\Delta]\right)$ admit representatives, which we also call \emph{lifts}, that are disjoint (resp. nested);
    \item If $\ov\Delta$ is not a facet, then the coordinate space $\C(\ov\Delta):=\mathcal{C}([\ov\Delta])$ associated to $[\ov\Delta]$ is quasi-isometric to $\link(\ov \Delta)$ (see \cite[Claim 6.12]{BHMS}). Otherwise, the pointwise stabiliser of the saturation $P(\ov\Delta):=\text{Pstab}(\text{Sat}(\ov\Delta))$ is a hyperbolic group acting properly and cocompactly on $\C(\ov\Delta)$.
\end{enumerate}
\end{remark}

We will use \emph{convex-cocompact subgroups} to show that the $\C(\ov\Delta)$ have at least 4 points at infinity (or they are bounded). This notion was introduced in \cite{FarbMosher}, and it admits several characterisations. The properties that we will need are that there exist free convex-compact subgroups inside every finite index subgroup (see e.g. \cite[Theorem 1.4]{FarbMosher}), and the following fact from \cite{convccpt}. If a subgroup $Q<MCG(S)$ is convex-cocompact then there exists a constant $D$ such that, for every element $h\in Q$ and for every two curves $x,s\subset S$ we have $d_s(x,h(x))\le D$ whenever the quantity is defined.

In our context, convex-cocompact subgroups survive in deep enough quotients:
\begin{lemma}\label{cvccpt}
    Let $S$ be a connected surface of finite type and whose complexity is at least $2$. Given a convex-cocompact subgroup $Q<MCG(S)$, for all large multiples $K$ the projection $\pi|_Q$ is injective and the orbit maps of $Q$ to $\mathcal{C}(S)/DT_K$ are quasi-isometric embeddings.
\end{lemma}
\begin{proof}
    This is just \cite[Theorem 7.1.iii]{BHMS}.
\end{proof}

Now, in order for our machinery from Section \ref{sec:extract} to work, we must ensure that we are in the assumptions of Theorem \ref{hingesauto}:
\begin{lemma}\label{CovDelta}
For all large multiples $K$ the following holds. For any $[\ov\Delta]\in \mathfrak{S}$, either the space $\C(\ov\Delta)$ is bounded or it has at least four points at infinity.
\end{lemma}

\begin{proof}
First, assume that $\ov\Delta$ has codimension $1$. Then $\C(\ov\Delta)$ is quasi-isometric to the hyperbolic group $P(\ov\Delta)$, and we claim that there exists a copy of the free group on two generators $F_2$ inside $P(\ov\Delta)$. If this is the case then $P(\ov\Delta)$ is an infinite, non virtually cyclic hyperbolic group, hence it has at least four boundary points.\\
Let $\Delta\subset\C$ be a lift of $\ov\Delta$. By \cite[Proposition 8.13.vii]{BHMS} we have that $\pi(\text{Sat}(\Delta))=\text{Sat}(\ov\Delta)$, therefore $P(\ov\Delta)$ contains the quotient projection of $P(\Delta):=\text{Pstab}(\text{Sat}(\Delta))$. Thus it is enough to show that there is a copy of $F_2$ inside $P(\Delta)$ whose projection to $MCG^\pm/DT_K$ is injective. Now, $\Delta$ is a facet in $\C$, thus it cuts out a four-holed sphere $U$. Notice that a curve $\gamma\in \C$ is in the link of $\Delta$ if and only if $\gamma$ lies in $U$ and is not one of its boundary curves. Hence, a simplex $\Delta'$ which has the same link of $\Delta$ must be a pants decomposition for $S\setminus U$, including its boundary curves. This in turn means that, if we see $U$ just as a four-punctured sphere $S_4$ (that is, we forget about the difference between punctures and boundary curves), then $ P(\Delta)$ contains the pure mapping class group $PMCG(S_4)$. \\
Inside $PMCG(S_4)$, which is finite-index inside $MCG(S_4)$, we may find a convex-cocompact copy of $F_2$, call it $H$. Suppose by contradiction that $\pi|_H$ is not injective, which means that there is some $h\in H\cap DT_K\setminus \{1\}$. Now $h$ can not fix every curve, since the only elements of $MCG(S_4)$ with this property are the hyperelliptic involutions, which permute the punctures. Thus let $x\subseteq U$ be a curve such that $h(x)\neq x$ and let $(s,\gamma_s)$ be as in Proposition \ref{cor3.6}. Notice that $h(x)$ still lies on $U$, therefore both $x$ and $h(x)$ complete $\Delta$ to maximal simplices.\\
Now, suppose by contradiction that $d_{\C}(x,s)>1$, and therefore $d_s(x,h(x))>\Theta$. If we argue as in Lemma \ref{edgelift} we get that $\gamma_s$ must fix $\Delta$ pointwise. Hence $s$ is disjoint from all curves in $\Delta$, but $s\not\in \Delta$ since otherwise $\gamma_s$, which is a power of $T_s$, would fix $x$. Therefore $s$ lies in $U$, and by convex-cocompactness of $H$ there exists a constant $D$ (which we can choose independently of $U$) such that $d_s(x,h(x))\le D$. This is a contradiction if $\Theta>D$, which happens for all large multiples $K$.
\\
Then we must have that $d_{\C}(x,s)\le 1$, i.e. $x$ must be fixed by $\gamma_s$, and we can apply $\gamma_s$ to the whole data and proceed by induction on the complexity of $h$. In the end we must have that $h=\prod_{i=1}^r \gamma_{s_i}$ and every $\gamma_{s_i}$ fixes $x$. But then $h(x)=x$, which contradicts our hypothesis. Thus we proved the Lemma for $\C(\ov\Delta)$ whenever $\ov\Delta$ has codimension $1$.\\
Now suppose that $\ov\Delta$ is not a facet. Then $\C(\ov\Delta)$ is quasi-isometric to $\link(\ov\Delta)$, which is the projection of $\link(\Delta)$ for some lift $\Delta$. Notice that $\link(\Delta)=\mathcal{C}(U)$ is the curve graph of the subsurface $U$ cut out by $\Delta$. If $U$ has at least two connected components which are not pairs of pants, then $\mathcal{C}(U)$ is bounded, and so is its projection. Otherwise $U$ is some connected subsurface of complexity at least $2$, and $\C(\ov\Delta)$ is the image of the curve graph $\mathcal{C}(U)$ under the quotient map, which is isomorphic to $\mathcal{C}(U)/DT_K(U)$ by Corollary \ref{CUdt}.\\
Now our goal has become to show that, for every surface $U$ of complexity at least $2$, the quotient of the curve graph $\mathcal{C}(U)/DT_K(U)$ has at least four boundary points. But now we can just apply Lemma \ref{cvccpt}: if we choose a convex-cocompact copy of $F_2$ inside $MCG(U)$ then its projection is still a copy of $F_2$ whose orbit maps to $\mathcal{C}(U)/DT_K(U)$ are quasi-isometric embeddings, and we are done.
\end{proof}

Thus we get:
\begin{corollary}[of Theorem \ref{hingesauto}]\label{cor:hinge_map_mcgdtn}
    Every quasi-isometry $f:\,MCG/DT_K\to MCG/DT_K$ induces an automorphism $f_{hin}$ of the hinge graph, such that for all $\sigma\in\bf{Hinge}(\mathfrak{S})$ we have $d_{Haus}(h_{f_{hin}(\sigma)}, f(h_\sigma))<\infty$.
\end{corollary}

\subsection{Complete support sets for the quotient}

Moving forward, our next goal is to understand the structure of complete support sets. For the following theorems we think of $\mathfrak{S}$ as the set of classes of subsurfaces (thus omitting the bijection $j$ whenever possible).
\begin{lemma}\label{cssXmcgdtn}
    The following holds for all large multiples $K$. Let $\left\{\ov{U}_i\right\}_{i=1}^r$ be a collection of pairwise orthogonal indices. Then there exist pairwise disjoint representatives $\left\{U_i\right\}_{i=1}^r$.
\end{lemma}
Notice that this lemma is not at all obvious: we just know that every two indices in $\left\{\ov{U}_i\right\}_{i=1}^r$ have disjoint representatives, but this does not mean a priori that this conditions can all be satisfied simultaneously. Notice moreover that the $\ov U_i$s in the statement can be any equivalence classes of subsurfaces, not necessarily elements of a complete support set. 

\begin{proof}[Proof of Lemma \ref{cssXmcgdtn}]
We proceed by induction on $r$, the base case $r\le2$ being true by the description of the HHS structure. Then let $r\ge 3$ and choose three indices $\ov U_1,\ov U_2, \ov U_3$. Take lifts $U_1, U_2,  U_3, U_1'$ such that $U_1\perp U_2$, $U_2\perp U_3$ and $U_3\perp U_1'$. Let $g\in DT_K$ be an element mapping $U_1$ to $U_1'$. For each of these surfaces we choose a family $C_i$ of filling curves, in such a way that $C_1'=g(C_1)$. Since $U_1$ and $U_2$ are disjoint, we have a join $C_1\star C_2$, and similarly for the other disjoint pairs. Then morally we have a "path" $C_1, C_2, C_3, C_1'$, and we want to show that we can glue $C_1$ to $C_1'$, as if we were looking for a "closed lift" of this path. The situation in the curve graph is as follows:
    $$\begin{tikzcd}
    C_1\ar[r, no head]\ar[rrr,"g",bend right=20]{g}&C_2\ar[r, no head]&C_3\ar[r, no head]&C_1'
    \end{tikzcd}$$
    We proceed by induction on the complexity of $g$. If $g$ is the identity we are done; otherwise fix a curve $x\in C_1$, let $x'=g(x)$, and let $(s,\gamma_s)$ be as in Proposition \ref{cor3.6}, applied to $x$ and $g$. If $d_{\C}(x,s)\le1$ we can apply $\gamma_s$ to everything and proceed by induction. Otherwise, one between $C_2$ and $C_3$ must be fixed by $\gamma_s$ pointwise, since if not we can find a path from $x$ to $x'$ that does not intersect the star of $s$, thus violating the Bounded geodesic image Theorem \ref{bgit} (here we use the fact that consecutive $C_i$s form a join). Either way we can apply $\gamma_s$ to part of our chain and proceed by induction.\\
At the end of this process we get that $U_1,U_2,U_3$ are pairwise disjoint. Now suppose that, for $3\le k<r$ we can find representatives $U_1,\ldots, U_k$ such that $U_1,U_2,U_i$ are pairwise disjoint for all $3\le i\le k$, and we want to show that the same holds for $k+1$. As before, let $U_{k+1}$ be a representative for $\ov U_{k+1}$ which is disjoint from $U_2$ and let $U_1'$ be a representative for $\ov U_1$ disjoint from $U_{k+1}$. Let $C_1,\ldots, C_{k+1},C_1'$ be sets of filling curves for the corresponding subsurfaces. In the curve graph there is a "triangle" of the form $C_1,C_2,C_i$ for every $i=3,\ldots,k$, as shown in this schematic picture:
$$\begin{tikzcd}
    &C_3\ar[dddl, no head]\ar[dddr, no head]&&&\\
    &\vdots&&&\\
    &C_k\ar[dl, no head]\ar[dr, no head]&&&\\
    C_1\ar[rr, no head]\ar[rrrr,"g",bend right=20]{g}&&C_2\ar[r, no head]&C_{k+1}\ar[r, no head]&C_1'\\
    \end{tikzcd}$$
The same argument as before shows that we can glue $C_1$ to $C_1'$, this time without breaking the "triangles": in every inductive step, $\gamma_s$ fixes either $C_1$ (and we apply $\gamma_s$ to the whole data) or one between $C_2$ and $C_{k+1}$ (and we can apply $\gamma_s$ beyond these curves, without moving the triangles).\\
If we do this for $k=3,\ldots, r-1$ we can find  representatives $\{U_i\}_{i=1}^r$ such that $U_1\perp U_2$ and both $U_1$ and $U_2$ are disjoint from every other $U_i$. But now we can consider $U_1$ and $U_2$ as a single (possibly disconnected) subsurface and conclude by induction on $r$.
\end{proof}

In the proof of Lemma \ref{cssXmcgdtn} we actually showed that every "simplex of indices" admits a lift. Now we want to prove the uniqueness of these lifts, up to elements of $DT_K$:
\begin{lemma}\label{uniqueliftcss}
    For all large multiples $K$ the following holds. Let $\left\{\ov{U_i}\right\}_{i=1}^r$ be a collection of pairwise orthogonal indices. Any two collections of pairwise orthogonal representatives $\left\{U_i\right\}_{i=1}^r,\,\left\{U_i'\right\}_{i=1}^r$ are obtained one from the other via some element $g\in DT_K$.
\end{lemma}

\begin{proof}
    We proceed by induction on $r$, the base case $r=1$ being clear. If the conclusion holds for $r-1$ then, up to some element $g\in DT_K$, we can assume that $U_i=U_i'$ for $i=1,\ldots, r-1$. Now let $h\in DT_K$ be an element that maps $U_r$ to $U_r'$, and let $C_1,\ldots,C_r,C_r'$ be sets of filling curves such that $h(C_r)=C_r'$. If $h$ is not the identity, fix a point $x\in C_r$ and let $(s,\gamma_s)$ be as in Proposition \ref{cor3.6}. If $d_{\C}(x,s)\le1$ we can apply $\gamma_s$ to everything and proceed by induction on the complexity of $h$; otherwise $\gamma_s$ must fix $C_1,\ldots,C_{r-1}$ pointwise, so we can apply $\gamma_s$ to the "simplex" $\{C_1,\ldots,C_{r-1},C_r'\}$ and proceed by induction.
\end{proof}

The main reason we established Lemmas \ref{cssXmcgdtn} and \ref{uniqueliftcss} is that every complete support set for $MCG/DT_K$ lifts to a collection of pairwise orthogonal $S_4$ and $S_5$ of the maximal cardinality, which we call a \emph{complete support set} for the surface $S$; moreover, such a collection is unique up to the action of $DT_K$.

Recall that a surface is \emph{odd} if its complexity is odd, otherwise it is \emph{even} (in our case, $S_b$ is odd iff $b$ is even). Bowditch \cite[Section 6]{BowPants} pointed out that if a subsurface $U$ belongs to a complete support set for $S$, then $U$ must have one of the following shapes, which ensure that every complete support set containing $U$ can cut out at most one pair of pants. 
\begin{itemize}
    \item If $S$ is odd, then $U$ is an $S_{4}$, and each connected component of the complement $S-U$ is odd and meets $U$ in exactly one curve. 
    \item If $S$ is even then $U$ is either an $S_{4}$ with all but one of the complementary components odd, or an $S_{5}$ with all complementary components odd.
\end{itemize}

\begin{remark}[Cutting out a pair of pants]\label{rem:problema_pair_of_pants}
   Notice that the only case where $U$ does not admit a complete support set whose union is the whole $S$ is when $S$ is even and the even connected component of $S-U$ is a pair of pants. Indeed:
   \begin{itemize}
       \item If $S$ is odd, every connected component of $S-U$ is an odd sphere, and can be cut into four-holed spheres. The same holds if $S$ is even and $U$ is an $S_5$.
       \item If $S$ is even and $U$ is an $S_4$, we can fill the odd complementary components with four-holed spheres as above. Then, if the even component is not a pair of pants then it must be a sphere with at least five holes, and one can fill it with an $S_5$ and some $S_4$.
   \end{itemize}
\end{remark}

\subsection{Unambiguous domains}
Let $\sigma=(\ov U,p)$ be a hinge for $MCG/DT_K$, and let $\text{Compl}(\sigma)$ be the set of all tuples of pairwise orthogonal hinges that $\sigma$ completes to a complete support set. In other words, $\text{Compl}(\sigma)$ is the set of all facets in $\bf{Hinge}(\mathfrak{S})$ that $\sigma$ completes to a maximal simplex.
\begin{definition}
    Two hinges $\sigma, \sigma'$ are said to be \emph{equally completable}, or to have the same completions, if $\text{Compl}(\sigma)=\text{Compl}(\sigma')$.
\end{definition}
This definition clearly induces an equivalence relation which is preserved by any automorphism of the hinge graph. One would like to think that if two hinges $(\ov U, p)$ and $(\ov U', p')$ are equally completable then $\ov U=\ov U'$. This is unfortunately not the case. Indeed, if $S$ is even and $\ov U$ is the class of a subsurface $U$ which is an $S_{4}$ whose even complementary component is a pair of pants $P$, then $\ov U$ has the same completions as the class of the $S_{5}$ given by the union of $U$ and $P$. See Figure \ref{fig:s4s5} to understand the situation. Luckily, we will see that this is the only problem that could arise.

\begin{figure}[htp]
    \centering
    \includegraphics[width=0.75\textwidth]{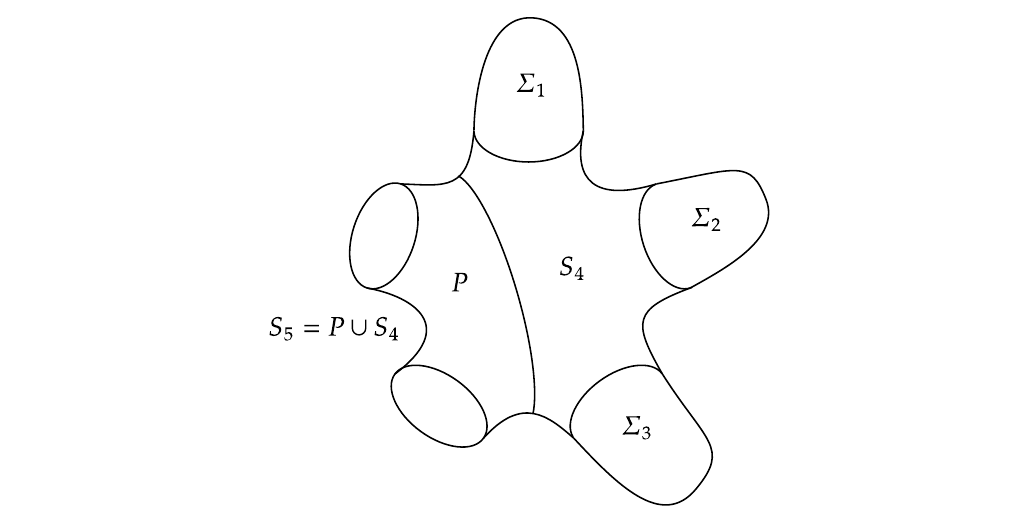}
    \caption{The classes of the $S_{4}$ and the $S_{5}$ in the Figure cannot be distinguished just by their completions, all of which lift inside the union of the odd complementary regions $\Sigma_1$, $\Sigma_2$, and $\Sigma_3$.}
    \label{fig:s4s5}
\end{figure}

\begin{definition}
    A hinge $\sigma$ is \emph{unambiguous} if any other equally completable hinge $\sigma'$ has the same support. A support $\ov U$ is unambiguous is every hinge supported on $U$ is unambiguous.
\end{definition}
We need another definition that characterises (some) unambiguous domains, and can be recognised from the hinge graph.
\begin{definition}
    A hinge $\sigma$ is \emph{minimal} if $\text{Compl}(\sigma)$ is maximal by inclusion, among completions.
\end{definition}

\begin{lemma}\label{minimalov}
For all large multiples $K$, a hinge $(\ov U,p)$ for $MCG/DT_K$ is minimal if and only if it is unambiguous and its support is the class of a four-holed sphere.
\end{lemma} 

\begin{proof}
    Let $U$ be a lift of $\ov U$. We examine all possible shapes of $U$.
    \begin{itemize}
        \item Suppose $U$ is an $S_5$, and choose some subsurface $V\sqsubsetneq U$ such that $V$ is an $S_4$ and $S-V$ does not contain a pair of pants. Such a $V$ exists, as one can choose a pair of pants $Q$ whose boundary intersects the boundary of $S-U$, and set $V=U-Q$. Every completion for $\ov U$ admits representatives that are disjoint from $U$ by Lemma \ref{cssXmcgdtn}, and therefore also from $V$. Thus every completion for $\ov U$ is also a completion for $\ov V$.  Moreover we claim that there exists a completion $\{\ov V_i\}$ for $\ov V$ but not for $\ov U$. To this purpose choose a completion $\{V_i\}$ for $V$, such that some boundary curve $\delta$ crosses one of the boundary components $\eta$ of $U$, and let $\{\ov V_i\}$ be their classes. Now recall that, as a consequence of Lemma \ref{projinlink}, if we fix some finite subset $F\subset \C$ then for all large multiples $K$ the projection map is an isometry on $g(F)$, for every mapping class $g\in MCG$. In our case, there is a finite number of $S_5$ inside $S$, up to the action of the mapping class group, and for each of these we can find $\delta$ and $\eta$ as above. Then we set $F$ as the union of these curves. 
        Now, any lift $\{V_i\}$ of $\{\ov V_i\}$ must intersect $U$. More precisely, $\eta$ must cross the boundary curve $\delta'$ corresponding to $\delta$, since
        $$d_{\C}(\delta',\eta)\ge d_{\Cdt}(\ov{\delta'},\ov\eta)=d_{\Cdt}(\ov\delta,\ov\eta)=d_{\C}(\delta,\eta)\ge2$$
        where we used that the projection is $1$-Lipschitz and that its restriction to $F$ is isometric. This shows that $\ov U$ is not minimal.
        \item Suppose $U$ is an $S_4$ which cuts out a pair of pants, as in Figure \ref{fig:s4s5}. Let $V$ be the $S_5$ given by the union of $U$ and the pair of pants cut out by $U$. Then, slightly abusing notation, we have that $\text{Compl}(\ov U)=\text{Compl}(\ov V)$. Indeed, every support $\ov W$ which completes $\ov V$ also completes $\ov W$; conversely, if $\ov W$ completes $\ov U$ then it admits a lift $W$ inside one of the complementary components of $U$, which is therefore disjoint from $V$. Thus $\ov U$ is ambiguous. Moreover, by the above argument $\ov V$ is not minimal as it is an $S_5$, and therefore neither is $\ov U$ which has the same completions as $\ov V$.
        \item The only case left is when $U$ is an $S_4$ which does not cut out a pair of pants. Choose a completion $\{U_i\}$ such that $U=S\setminus\bigcup U_i$, as in Remark \ref{rem:problema_pair_of_pants}, and let $\{\ov U_i\}$ be its projection. If $\ov V$ completes $\{\ov U_i\}$ then we may lift it to some $V\sqsubseteq U$ (here we used that, up to elements of $DT_K$, the lift of $\{\ov U_i\}$ is unique). But then $V=U$ since $U$ has minimal complexity. This proves that $\ov U$ is minimal and unambiguous.
    \end{itemize}
\end{proof}

\subsection{Minimal product regions are preserved}
\begin{remark}\label{rem:auto_hinge_minimal}
Any automorphism of the hinge graph must map minimal hinges with the same completions to minimal hinges with the same completions, therefore it acts on the set of minimal supports. Thus, if $f\colon MCG/DT_K\to MCG/DT_K$ is a quasi-isometry, and $\ov U$ is a minimal support, we can define $f_{supp}(\ov U)$ as the support of $f_{hin}(\ov U,p)$ for any hinge $(\ov U,p)$ supported on $\ov U$, where $f_{hin}$ is the map from Corollary \ref{cor:hinge_map_mcgdtn}.
\end{remark}

\begin{corollary}[Products go to products]\label{prodtoprod}
Let $f\colon MCG/DT_K\to MCG/DT_K$ be a quasi-isometry. There exists $C$, depending only on the quasi-isometry constants of $f$, with the following property. Let $\{\ov U_i\}\subseteq\mathfrak{S}$ be a complete support set made of minimal supports, and let $f_{supp}(\ov U_i)=\ov V_i$ for all $i$. Let $P_{\{\ov U_i\}}$ and $P_{\{\ov V_i\}}$ be the standard product regions defined by $\{\ov U_i\}$ and $\{\ov V_i\}$, respectively. Then $d_{Haus}\left(f\left(P_{\{\ov U_i\}}\right),  P_{\{\ov V_i\}}\right)\le C$.
\end{corollary}

\begin{proof}
Another way of stating Remark \ref{rem:auto_hinge_minimal} is that, if $(\ov U,p^\pm)$ are two hinges with the same minimal support and $f_{hin}(\ov U,p^+)=(\ov V,q^+)$, then there exists $q^-\in\partial \C V$ such that $f_{hin}(\ov U,p^-)=(\ov V,q^-)$. Thus the Flats to Flats Lemma \ref{flatstoflats} says that, if $\{(\ov U_i,p_i^\pm)\}$ is a complete support set made of minimal supports, with a choice of two points in every $\partial \C \ov U_i$, and if we set $(\ov V_i,q_i^\pm):=f_{hin}(\ov U_i,p_i^\pm)$, then $d_{Haus}(f(\mathfrak F_{\{(\ov U_i,p_i^\pm)\}}),  \mathfrak F_{\{(\ov V_i,q_i^\pm)\}})\le C_0$ for some constant $C_0$ depending only on the quasi-isometry constants of $f$. Hence
$$d_{Haus}\left(f\left(\bigcup \mathfrak F_{\{(\ov U_i,*)\}}\right),  \bigcup \mathfrak F_{\{(f_{supp}(\ov U_i),*)\}}\right)\le C_0 $$
where $\bigcup \mathfrak F_{\{(\ov U_i,*)\}}$ is the union of all standard flats supported in $\{\ov U_i\}$. 

Now, the conclusion follows if we prove the existence of some constant $C_1$ such that, for every complete support set $\{\ov U_i\}$ made of minimal supports, 
$$d_{Haus}\left(\bigcup \mathfrak F_{\{(\ov U_i,*)\}}, P_{\{\ov U_i\}}\right)\le C_1$$ 
In fact, if this is the case then we also get that 
$$d_{Haus}\left(f\left(\bigcup \mathfrak F_{\{(\ov U_i,*)\}}\right), f\left(P_{\{\ov U_i\}}\right)\right)\le C_2$$
where $C_2$ is some constant depending only on $C_1$ and the quasi-isometry constants of $f$. In turn, since $P_{\{\ov U_i\}}=\prod F_{\ov U_i}$, to prove the existence of such $C_1$ it suffices to select a close enough hierarchy line in each coordinate, as we shall do in Lemma \ref{lem:unif_dist_from_hpath_mcgdtn}.
\end{proof}

\begin{lemma}\label{lem:unif_dist_from_hpath_mcgdtn}
    There exists a constant $C_3$ such that whenever $\ov U$ is a minimal domain and $x\in F_{\ov U}$, there exists a hierarchy line $\gamma\subset F_{\ov U}$ such that $d_{F_{\ov U}}(x,\gamma)\le C_3$.
\end{lemma}

\begin{proof}
Since $P_{\{\ov U_i\}}=\prod F_{\ov U_i}$ it suffices to work in each coordinate, that is, we will show that there exists a constant $C_3$ such that whenever $\ov U\in\frakS$ is a minimal domain and $x\in F_{\ov U}$, there exists a hierarchy line $\gamma\subset F_{\ov U}$ such that $d_{F_{\ov U}}(x,\gamma)\le C_3$. Notice that, by the distance formula and the fact that $\ov U$ is a $\nest$-minimal, we have that $F_{\ov U}$ is uniformly quasi-isometric to $\C \ov U$, and hierarchy lines simply correspond to bi-infinite quasigeodesics (with certain constants). Hence we just need to show that every point $x\in \C \ov U$ is within uniformly bounded distance from a bi-infinite quasigeodesic. If $\ov U=j([\ov \Delta])$, as in Remark \ref{rem:prop_hhs_mcgdtn}.(iv), then $\C\ov U$ is uniformly quasi-isometric to the hyperbolic group $P(\ov\Delta)$, which has at least $2$ points at infinity and therefore admits bi-infinite quasigeodesics. Since a group acts transitively on itself the conclusion follows.
\end{proof}

\subsection{Terminal supports and $1$-separating curves}
We want to translate the automorphism $f_{supp}$ into an automorphism of the graph $\Cpdt$ from Section \ref{section:1tostrong}. In other words, we need a way to identify the minimal domains which are cut out by $1$-separating curves. To this purpose, we introduce the following definition.
\begin{definition}\label{1sep}
A support $\ov U$ is \emph{terminal} if it is the class of a subsurface $U$ of complexity $1$ which is cut out by a single curve. Notice that the boundary curve of such a $U$ is $1$-separating.
\end{definition}

\begin{lemma}\label{lem:term_supp_quoziente}
    For all large multiples $K$, a hinge $\sigma=(\ov U,p)$ has terminal support if and only if it is minimal and there exists a hinge $(\ov V,q)$, compatible with $\sigma$, such that any complete support set containing $(\ov V,q)$ must contain some $\sigma'$ supported in $\ov U$.
\end{lemma}

\begin{proof}
    Suppose $\ov U$ is terminal, and choose a lift $U$ of $\ov U$. Let $V$ be a support such that $U$ is one of the connected components of $S-V$; such a $V$ exists as $S-V$ has complexity at least $2$, since we assumed that $S=S_b$ with $b\ge 7$. Then every complete support set containing $\ov V$ lifts to a complete support set containing $V$, and therefore also $U$. Moreover, by the discussion in the proof of Lemma \ref{minimalov}, a terminal support is also minimal.
    
    Conversely, suppose that $\ov U$ is not terminal. If $\ov U$ is not minimal we have nothing to prove. Otherwise $\ov U$ is the class of a non-terminal $S_4$, and we want to show that every $\ov V\perp \ov U$ which is compatible with $\ov U$ admits a completion that does not contain $\ov U$. First notice that, up to the action of the mapping class group, there is a finite number of possible pairs $U\perp V$ of orthogonal subsurfaces. For each of these possibilities, let $\{V_i\}_{i=2}^\nu$ a completion for $V=V_1$, that we can choose in such a way that some boundary curve $\delta$ of $U$ crosses some boundary curve $\eta$ of some $V_i$, say, $V_2$. This is always possible because $U$, which is a non-terminal $S_4$, cannot coincide with the connected component $\Sigma$ of $S\setminus V$ it belongs to, and therefore there must be a relative boundary curve of $U$ inside $\Sigma$ that we can choose as $\delta$. Let $F$ be the finite union of all $\delta$s and $\eta$s that arise from these possibilities, and for all large multiples $K$ the projection is an isometry on $F$. Thus $\{\ov V_i\}$ is a completion for $\ov V$ that cannot contain $\ov U$, since any lift $U'$ of $\ov U$ contains a boundary curve $\delta'$ which must cross $\eta$ (we can argue precisely as in Lemma \ref{minimalov}). Hence $\ov U\pitchfork\ov V_2$ since $U'\pitchfork V_2$ for every lift $U'$. This implies that $\ov U$ cannot be an element of $\{\ov V_i\}$, and we are done.
\end{proof}

\begin{corollary}
    For all large multiples $K$, any self-quasi-isometry $f$ of $MCG/DT_K$ induces an automorphism $\phi$ of $\Cpdt$.
\end{corollary}
\begin{proof}
For every terminal support $\ov U$, let $\partial \ov U$ be the element $\ov\gamma\in \Cpdt$ such that there exist lifts $U$ and $\gamma=\partial U$. This is a one-to-one correspondence between terminal supports and vertices of $\Cpdt$ (here we are using that every $1$-separating curve cuts out exactly one $S_4$, as $S=S_b$ with $b\ge 7$). Now, since being terminal has a completely combinatorial characterisation, in view of Lemma \ref{lem:term_supp_quoziente}, $f_{supp}$ must map terminal supports to terminal supports, and we can set $\phi(\partial \ov U)=\partial f_{supp}(\ov U)$. Moreover, $f_{supp}$ preserves compatibility, which, for terminal subsurface, is equivalent to the fact that their boundaries admit disjoint representatives. Thus $\phi$ is an automorphism of $\Cpdt$.
\end{proof}

By Theorem \ref{AutoExt}, $\phi$ comes from some $\ov g\in MCG^\pm/DT_K$. Finally, we need to show that, if $f_{supp}$ and $\ov g$ agree on terminal subsurfaces then they agree on every minimal surface. More precisely we claim the following:
\begin{lemma}\label{minimalfromterminal}
    For all large multiples $K$, every minimal support $\ov U$ for $MCG/DT_K$ is uniquely determined by the terminal supports it is compatible with.
\end{lemma}

\begin{proof}
    Let $U$ be a lift of $\ov U$ and let $S\setminus U=\bigsqcup_{i=1}^4 \Sigma_{i}$. Moreover, let $\ov V$ be another support such that every terminal support $\ov T$ compatible with $\ov U$ is also compatible with $\ov V$. We claim that, for $i=1,\ldots, 4$, there exists a representative $V_i$ for $\ov V$ which is disjoint from $\Sigma_i$. If this is the case then there exists a representative $V$ which is disjoint from all $\Sigma_i$s (more precisely, we can use Lemma \ref{cssXmcgdtn} to lift the "simplex" $\{\ov V\}\cup\{\ov \Sigma_i\}_{i=1}^4$, and Lemma \ref{uniqueliftcss} shows that we can choose $\{\Sigma_i\}_{i=1}^4$ as lifts of $\{\ov \Sigma_i\}_{i=1}^4$). Therefore $V\nest U$, and equality holds since $U$ has already minimal complexity.\\
    First notice that there is a finite number of possibilities for $U$, up to the action of the mapping class group. For each of these possibilities look at its complementary components. Whenever one of these, call it $\Sigma$, is not terminal we choose two $1$-separating curves $\alpha,\alpha'$ that fill $\Sigma$ and a pants decomposition $\Delta$ for $S\setminus \Sigma$, including its boundary. Let $F$ be the finite set of curves given by the union of all these $\alpha,\alpha'$ and $\Delta$. Notice that Corollary \ref{projfilling} tells us that for all large multiples $K$ every lift of $\ov \alpha,\ov \alpha'$ inside $\link(\Delta)$ is still a pair of filling curves for $\Sigma$.\\
    Now we go back to our proof. If $\Sigma$ is already a terminal support then we can find a representative $V$ which is disjoint from $\Sigma$, and we are done. Otherwise let $\alpha, \alpha',\Delta$ be the image under some mapping class of the corresponding elements of $F$, which satisfy the property that every two lifts of $\ov \alpha,\ov \alpha'$ inside $\link(\Delta)$ fill $\Sigma$. Let $W,W'$ be the terminal subsurfaces cut out by $\alpha$ and $\alpha'$, and let $V,V'$ be some representatives of $\ov V$ such that $V\perp W$ and $V'\perp W'$. Let $g\in DT_K$ that maps $V$ to $V'$. Let $C$ be a collection of filling curves for $V$, and let $C'=g(V)$. Then in the curve graph the situation is as follows:
    $$\begin{tikzcd}
    C\ar[r, no head]\ar[rrrr,"g",bend right=20]{g}&\alpha\ar[r, no head]&\Delta\ar[r, no head]&\alpha'\ar[r, no head]&C'
    \end{tikzcd}$$
    Now if $g$ is not the identity pick some $x\in C$ and let $(s,\gamma_s)$ be as usual. If $d_{\C}(x,s)\le1$ we can apply $\gamma_s$ to everything and proceed by induction on the complexity of $g$. Otherwise we must be in one of these three cases:
    \begin{itemize}
        \item If $d_{\C}(\alpha,s)\le1$ we can apply $\gamma_s$ to everything after $\alpha$. Now $\gamma_s(\alpha)$ and $\gamma_s(\alpha')$ still fill $\gamma_s(\Sigma)$, and we can proceed.
        \item If $d_{\C}(\alpha',s)\le1$ we can apply $\gamma_s$ just to $C'$, without touching $\alpha$ and $\alpha'$.
        \item If $\Delta$ is fixed pointwise by $\gamma_s$ then we can apply $\gamma_s$ to everything after $\Delta$. Notice that $\gamma_s(\alpha')$ is still a lift of $\ov \alpha'$ in the link of $\Delta$, hence it still fills $\Sigma$ together with $\alpha$.
    \end{itemize}
    At the end of the induction we have that $C=C'$ is disjoint from both $\alpha$ and $\alpha'$, which may differ from the original curves but remain a pair of filling curves for a representative of $\ov \Sigma$. This gives us the required representative for $\ov V$.
\end{proof}

\subsection{Quasi-isometric rigidity}
We are finally ready to prove quasi-isometric rigidity of $MCG^\pm/DT_K$, which is Theorem \ref{qimcgdtn}. We subdivide the proof in two steps.
\begin{theorem}\label{selfqimcg}
    Let $S=S_{0,b}$ be a punctured sphere, with $b\ge 7$ punctures. For all large multiples $K$ the following holds. For every $T>0$ there exists $D>0$ such that every $(T,T)$-self-quasi-isometry $f$ of $MCG(S)/DT_K$ lies within distance $D$ of the left multiplication by some element $\ov g\in MCG^\pm/DT_K$, which depends only on the restriction of $f_{supp}$ to terminal supports.
\end{theorem}

\begin{proof}
Let $f$ be a self-quasi-isometry of $MCG/DT_K$, and let $f_{supp}$ be the induced map on minimal supports. With a slight abuse of notation, the restriction of $f_{supp}$ to terminal supports is an automorphism of $\Cpdt$ which comes from some element $\ov g\in MCG^\pm/DT_K$ by Theorem \ref{csstoc}. Then $f_{supp}$ and $g$ agree on terminal supports, and therefore also on every minimal support by Lemma \ref{minimalfromterminal}. Now, the Products to products Corollary \ref{prodtoprod} tells us that there exists some constant $C$, depending only on the quasi-isometry constants of $f$, such that, if $\{\ov U_i\}$ is a complete support set made of minimal supports, $f$ maps the corresponding  product region $P_{\{U_i\}}$ within Hausdorff distance at most $C$ from $\ov g(P_{\{U_i\}})$. 

We are left to prove that every point $x\in MCG/DT_K$ is the (uniform) coarse intersection of two standard product regions $P\Tilde{\cap} P'$, coming from minimal complete support sets. If this is the case then $f(x)$ will be the coarse intersection of $f(P)$ and $f(P')$, which lie at uniformly bounded distance from the coarse intersection of $\ov g(P)$ and $\ov g(P')$, which is coarsely $\ov g(x)$. Notice that it is enough to show that there is \emph{some} point $x_0\in MCG/DT_K$ which is the coarse intersection of two standard product regions with minimal supports, since $MCG/DT_K$ acts transitively on itself and maps product regions to product regions. In turn, by Corollary \ref{prodint} we are left to prove that there exist two complete support sets $\{\ov U_i\}$ and $\{\ov V_i\}$ with minimal, pairwise distinct supports (notice that we can apply Corollary \ref{prodint} since minimal supports are $S_4$, and therefore they are also minimal with respect to nesting). 

Firstly, fix a complete support set $\{\ov U_i\}$ as follows:
\begin{itemize}
    \item If $S$ is odd, choose any complete support set $\{\ov U_i\}$.
    \item If $S$ is even, let $\Tilde{U}$ be the class of an $S_5$ cut out by a single curve, and complete $\Tilde U$ to a complete support set.
\end{itemize}
In both cases, by inspection of the proof of Lemma \ref{minimalov}, one gets that $\{\ov U_i\}$ is made of minimal supports. Let $\{U_i\}$ be a lift of $\{\ov U_i\}$, and choose a pseudo-Anosov mapping class $h$ such that every boundary curve of $\{U_i\}$ crosses every boundary curve of $\{V_i\}=h\{U_i\}$. Now for all large multiples $K$ the projection is an isometry on the finite set $F$ of boundary curves of $\{U_i\}$ and $\{V_i\}$. Therefore we must have that $\ov U_i\neq \ov V_j$ for every choice of $i$ and $j$, because any two lifts $U_i'$ and $V_j'$ must have crossing boundaries. This proves the theorem.
\end{proof}

\begin{corollary}
For every $T>0$ there exists $D$ such that, if a $(T,T)$-self-quasi-isometry $f$ of $MCG/DT_K$ lies within finite distance of the identity, then it lies within distance $D$ of the identity.
\end{corollary}
\begin{proof}
    By the previous theorem we know that $f$ lies within distance $D$ of the left multiplication by some $\ov g$, which depends only on the induced map $f_{supp}$ on terminal supports. In turn $f_{supp}$ is induced by $f_{hin}$, thus if we show that this map is the identity then $\ov g$ can be chosen to be the identity, and the corollary follows. Now, Theorem \ref{hingesauto} tells us that $d_{Haus}(h_{f_{hin}(\sigma)}, f(h_\sigma))<\infty$. But then, since $d_{Haus}(f(h_{\sigma}),h_\sigma)<\infty$ we must also have that $d_{Haus}(h_{f_{hin}(\sigma)}, h_\sigma)<\infty$. This in turn means that $f_{hin}(\sigma)=\sigma$, since $h_\sigma$ has the property that if $\sigma\neq\sigma'$ then $d_{Haus}(h_\sigma,h_{\sigma'})=\infty$. 
\end{proof}

Thus Theorem \ref{qimcgdtn} is implied by the general statement below, which follows from standard arguments (see e.g. \cite[Section 10.4]{Schwa}):

\begin{lemma}
\label{lem:schwarz}
Let $H$ be a finitely generated group. Suppose that for every $T>0$ there exists $D>0$ such that:
\begin{enumerate}
    \item Every $(T,T)$-self-quasi-isometry of $H$ lies within distance $D$ of the left multiplication by some element of $H$; 
    \item If a $(T,T)$-self-quasi-isometry of $H$ lies within finite distance of the identity then it lies within distance $D$ of the identity.
\end{enumerate}
Then $H$ is a finite extension of the group $\text{QI}(H)$ of its self-quasi-isometries up to bounded distance. Moreover, if a finitely generated group $G$ is quasi-isometric to $H$ then $G$ and $H$ are weakly commensurable, meaning that there exist two finite normal subgroups $L\unlhd H$ and $M\unlhd G$ such that the quotients $H/L$ and $G/M$ have two finite index subgroups that are isomorphic.
\end{lemma}

\begin{proof}
Set $\mu:\,H\to \text{QI}(H)$ by mapping $h$ to the left multiplication by $h$. This map is surjective by Item (1); moreover it has finite kernel, since if $\mu(h)$ is within finite distance of the identity then it is within distance $D=D(1)$ of the identity, which means that $d_H(h,1)=d_H(\mu(h)(1),1)\le D$ and we conclude since balls in Cayley graphs are finite.\\
Regarding the second statement let $L$ be such that $H/L\cong \text{QI}(H)$, and let $\phi\,:G\to H$ be a quasi-isometry with quasi-inverse $\phi^{-1}$. We define a group homomorphism $\psi:\,G\to \text{QI}(H)$ by setting
$$\psi(g) (h)= \phi(g \phi^{-1}(h))$$
Notice that $\psi(g)$ is a self-quasi-isometry of $H$, whose constants $(T,T)$ depend only on the quasi-isometry constants of $\phi$ and $\phi^{-1}$. Let $D=D(T)$ as in the hypothesis. The same argument as before shows that $\psi$ has finite kernel: if $g\in \ker \psi$ then $\phi(g \phi^{-1}(1))$ is $D$-close to $1$, hence $g \phi^{-1}(1)$ is $D'$-close to $\phi^{-1}(1)$ for some other constant $D'(\phi,\phi^{-1},D)$. Thus the Lemma follows if we prove that $\psi$ is coarsely surjective.\\
For every $g\in G$ let $\theta(g)\in H$ be the element whose multiplication is $D$-close to $\psi(g)$. Then by construction the following diagram commutes:
$$\begin{tikzcd}
&H\ar{d}{\mu}\\
G\ar{ru}{\theta}\ar{r}{\psi}&\text{QI}(H)
\end{tikzcd}$$
Since $\mu$ is a quotient map it is $1$-Lipschitz and surjective. Hence it is enough to show that $\theta$ is coarsely surjective, which in turn will follow if we prove that $\theta$ lies within bounded distance from the quasi-isometry $\phi$. In fact $\theta(g)$ is $D$-close to $\phi(g \phi^{-1}(1))$, which in turn is uniformly close to $\phi(g)$ since $d_G(g,g\phi^{-1}(1))=\|\phi^{-1}(1)\|_G$ is constant.
\end{proof}

\section{Algebraic rigidity}
\label{sec:alg}
This section is devoted to the proof of Theorem \ref{thm:intro2}, which is covered by Theorems \ref{Out} and \ref{autmcg} below.

\begin{theorem}\label{Out}
    For every $b\ge 7$ and for all large multiples $K$ the following holds. Let $\phi:\,H\to H'$ be an isomorphism between finite index subgroups of $MCG^\pm(S_b)/DT_K$. Then $\phi$ is the restriction of an inner automorphism. In particular $\text{Out}(MCG^\pm(S_b)/DT_K)$ is trivial.
\end{theorem}

We recall some definitions from \cite{Commensurating} that we will need to state the exact theorem.

\begin{definition}
Two elements $h$ and $g$ of a group $G$ are \emph{commensurable}, and we write $h \stackrel{G}{\approx} g$, if there exist $m,n\in\mathbb{Z}\setminus\{0\}$, $k\in G$ such that $kg^m k^{-1}=h^n$ (that is, if they have non-trivial conjugate powers).
\end{definition}

\begin{definition}
If a group $G$ acts by isometries on a hyperbolic space $\mathcal{S}$, an element $g\in G$ is \emph{loxodromic} if for some $x\in\mathcal{S}$ the map $\mathbb{Z}\to \mathcal{S}$, $n\mapsto g^n(x)$ is a quasi-isometric embedding. In the same setting, an element $g\in G$ is \emph{weakly properly discontinuous}, or \emph{WPD}, if for every $\varepsilon>0$ and any $x\in \mathcal{S}$ there exists $N_0=N_0(\varepsilon,x)$ such that whenever $N\ge N_0$ we have
$$\left | \left\{h\in G | \max\left\{d_{S}\left(x, h(x)\right), d_{S}\left(g^N(x), hg^N(x)\right)\right\}\le \varepsilon\right\}\right |<\infty$$
We denote by $\mathcal{L}_{WPD}$ the set of loxodromic WPD elements.
\end{definition}
The following result is a special case of \cite[Theorem 7.1]{Commensurating}. Roughly speaking, the theorem says that if $H$ is a subgroup of $G$ and both act "interestingly enough" on some hyperbolic space, then any homomorphism $\phi:H\to G$ is either (the restriction of) an inner automorphism or it maps some loxodromic WPD to an element which is not commensurable to it.

\begin{theorem}\label{7.1}
    Let $G$ be a group acting coboundedly and by isometries on a hyperbolic space $\mathcal{S}$, with loxodromic WPD elements. Let $H\le G$ be a non-virtually-cyclic 
    subgroup such that $H\cap\mathcal{L}_{WPD}\neq \emptyset$, and let $E_G(H)$ be the unique maximal finite subgroup of $G$ normalised by $H$, whose existence is proven in \cite[Lemma 5.6]{Commensurating}. Let $\phi:\,H\to G$ be a homomorphism such that whenever $h\in H\cap \mathcal{L}_{WPD}$ then $\phi(h) \stackrel{G}{\approx} h$. If $E_G(H)=\{1\}$ then $\phi$ is the restriction of an inner automorphism.
\end{theorem}

\begin{proof}[Outline of the proof of Theorem \ref{Out}]
We just need to verify that, for all $b\ge7$ and for all large multiples $K$, the hypotheses of Theorem \ref{7.1} are satisfied for $G=MCG^\pm(S_b)/DT_K$ and any isomorphism $\phi:\,H \to H'$ between subgroups of finite index. In \cite[Theorems 2.1 and 5.2]{dfdt} it was proven that $G$ is not virtually cyclic, it acts coboundedly and by isometries on the hyperbolic space $S=\Cdt$, and the action admits loxodromic WPD elements. In particular, every subgroup of finite index is not virtually cyclic and it contains some power of every loxodromic WPD element, which remains loxodromic WPD. In Lemma \ref{commensiso} we show that $\phi$ has the required commensurating property, that is, for every $h\in H\cap \mathcal{L}_{WPD}$ we have that  $\phi(h) \stackrel{G}{\approx} h$. Moreover, in view of the general Lemma \ref{reduction}, in order to prove that $E_G(H)=\{1\}$ it will suffice to prove that $E_G(G)=\{1\}$, i.e. that $MCG^\pm/DT_K$ has no non-trivial finite normal subgroups. This is done in Lemma \ref{nofinite}.
\end{proof}

\begin{lemma}\label{commensiso}
    Let $G=MCG^\pm(S_b)/DT_K$ for $b\ge7$. For every isomorphism $\phi:\,H\to H'$ between finite index subgroups and every element $h\in H$ of infinite order, $h$ and $\phi(h)$ are commensurable.
\end{lemma}

\begin{proof}
    Fix a finite generating set for $G$, and let $d_G$ be the corresponding word metric. Since $H$ has finite index there exists a quasi-isometry $f:\,G\to H$, which we can choose to be the identity on $H$ (for example, every $g\in G$ can be sent to one of the closest elements of $H$). Let $\Phi=\phi\circ f:\,G\to H'$, which coincides with $\phi$ on $H$ and is a self-quasi-isometry of $G$ because $\phi$ is an isomorphism between finite index subgroups. Then by Theorem \ref{selfqimcg} there exist a constant $D$ and an element $g\in G$ such that $\Phi$ is $D$-close to the left multiplication by $g$. In particular $d_G(g,1)=d_G(g,\Phi(1))\le D$, and since the word metric is invariant under left multiplication we also have that $d_G(g^{-1},1)=d_G(1,g)\le D$. Then $\Phi$ is also $2D$-close to the conjugation by $g$, since for every $k\in G$ we have that
    $$d_G(gkg^{-1},\Phi(k))\le d_G(gkg^{-1},gk)+d_G(gk,\Phi(k))=d_G(g^{-1},1)+d_G(gk,\Phi(k))\le 2D$$
    where again we used the left-invariance of the word metric. Choosing $k=h^l$ for every $l\in\mathbb{Z}$ we get that the infinite subgroups $H_1=\langle ghg^{-1}\rangle $ and $H_2=\langle \phi(h)\rangle$ lie at Hausdorff distance at most $2D$ (here we used that $\Phi|_H\equiv \phi$). But now \cite[Proposition 9.4]{Hruska} states that, whenever $G$ has a left-invariant proper metric (in our case, the word metric) and $H_1, H_2$ are two subgroups, for every $D$ there exists a constant $D'$ such that, if we denote by $N_{R}(S)$ the $R$-neighbourhood of a set $S$,
    $$ N_{2D}(H_1)\cap N_{2D}(H_2)\subseteq N_{D'}(H_1\cap H_2)$$
    Thus $N_{D'}(H_1\cap H_2)$ is infinite because it contains $H_1$. This in turn implies that $H_1\cap H_2$ is infinite, since balls in the word metric are finite. Therefore there exist common powers $gh^m g^{-1}=\phi(h)^n$, as required.
\end{proof}
\begin{lemma}\label{reduction}
    Let $G$ be a group acting coboundedly, non-elementarily on a hyperbolic space $\mathcal{S}$, and suppose the action admits loxodromic WPD elements. Let $H\le G$ be a finite index subgroup. If $E_G(G)=\{1\}$ then $E_G(H)=\{1\}$.
\end{lemma}
\begin{proof}
    For every element $g\in G$ set 
    $$E_G(g)=\left\{k\in G\,|\,\exists m,n\in \mathbb{Z}\setminus\{0\}\mbox{ s.t. }k g^m k^{-1}=g^n\right\}.$$
    It follows from the proof of \cite[Lemma 6.18]{DGO} that there exists a loxodromic WPD element $g_0$ such that $E_G(g_0)=\langle g_0\rangle \ltimes E_G(G)$ (there the notation $K(G)$ is adopted for $E_G(G)$). Since $E_G(G)=\{1\}$ we have that $E_G(g_0)=\langle g_0\rangle$.\\
    Now, in \cite[Lemma 5.6]{Commensurating} it is proved that, whenever $H$ is a non-virtually-cyclic subgroup such that $H\cap \mathcal{L}_{WPD}\neq\emptyset $, then
    $$E_G(H)=\bigcap_{h\in H\cap \mathcal{L}_{WPD}} E_G(h).$$
    Choose $k\in\mathbb{N}_{>0}$ such that $g_0^k\in H$, and notice that $E_G(g_0^k)=E_G(g_0)$ by definition. Thus $E_G(H)$ is a finite subgroup of $E_G(g_0^k)$, which is infinite cyclic, and therefore $E_G(H)=\{1\}$.
\end{proof}
The following statement is well-known to experts, but we provide a proof since we could not find a suitable reference.
\begin{lemma}
    Let $S=S_{g,b}$ be a surface of finite type with genus $g$ and $b$ punctures, with $(g,b)\not \in\{ (0,2), (0,3), (0,4), (1,0), (1,1), (1,2), (2,0)\}$. Then $MCG^\pm(S)$ has no non-trivial finite normal subgroups. In other words $E_{MCG^\pm}(MCG^\pm)=\{1\}$.
\end{lemma}

\begin{proof}
    Let $N\unlhd MCG^\pm(S)$ be a finite normal subgroup and let $f\in N$. If $f$ fixes every isotopy class of simple closed curves then it must be the identity, for example by Ivanov's Theorem \cite[Theorem 1]{Ivanov:autC} and its extension to lower genera \cite[Theorem 1]{Korkmaz}. Then suppose by contradiction that $f(\gamma)\neq\gamma$ for some curve $\gamma$, and let $T_\gamma$ be the corresponding Dehn twist. Since $N$ is normal there exists $g\in N$ such that $f T_\gamma^2=T_\gamma^2 g$. Thus, using how Dehn twists behave under conjugation, we have $T_{f(\gamma)}^{\pm 2}=f T_\gamma^2 f^{-1}=T_\gamma^2 gf^{-1}$, where the sign depends on whether $f$ is orientation preserving or reversing. This can be rewritten as 
    \begin{equation}\label{tgamma} T_\gamma^{- 2}T_{f(\gamma)}^{\pm2}=gf^{-1}.\end{equation}
    Now, referring to the table at the end of \cite[Subsection 3.5.2]{FarbMargalit} we see that $T_{\gamma}^2$ and $T_{f(\gamma)}^2$ are the generators of a subgroup isomorphic to either $\mathbb{Z}^2$ or $F_2$, hence the left-hand side of (\ref{tgamma}) has infinite order. This is impossible, since the right-hand side is an element of the finite subgroup $N$.
\end{proof}

\begin{lemma}\label{nofinite}
    For every $b\ge 7$ and for all large multiples $K$, $ MCG^{\pm}(S_b)/DT_K$ has no finite normal subgroups.
\end{lemma}

\begin{proof} For short, we denote $MCG^{\pm}(S_b)$ simply by $MCG^{\pm}$.
    The mapping class group is acylindrically hyperbolic, hence by \cite[Lemma 6.18]{DGO} there exists a pseudo-Anosov element $g\in MCG^\pm$ such that $E_{MCG^\pm}(g)=\langle g\rangle\ltimes E_{MCG^\pm}(MCG^\pm)=\langle g\rangle$. Moreover, if we fix a curve $x\in \C$, arguing as in the proof of Corollary \ref{infinite} we can find $\Theta>0$ such that $\sup_{s\in\C,n\in\mathbb{Z}}d_s(x, g^n(x))<\Theta$. Then for all large multiples $K$ the projection map $\C\to\Cdt$ is an isometry on the axis $\{g^n(x)\}_{n\in\mathbb{Z}}$, thanks to Lemma \ref{projinlink}. In particular the element $\ov g\in MCG^{\pm}/DT_K$ induced by $g$ has infinite order, since $\ov g^n$ maps $\ov x$ to the projection of $g^n(x)$, which is not $\ov x$.
    Now, we claim that every finite normal subgroup $N\le MCG^\pm/DT_K$ must be trivial. First notice that $N$ moves $\ov x$ within distance $M$ for some $M\ge0$, since it is finite. Then the whole axis $\{\ov g^n(\ov x)\}_{n\in\mathbb{Z}}$ is moved within Hausdorff distance $M$, since
    $$\sup_{n\in\mathbb{Z},\,\ov\phi\in N}d_{\Cdt}\left(\ov g^n(\ov x), \ov\phi\circ\ov g^n(\ov x)\right)=\sup_{n\in\mathbb{Z},\,\ov\psi\in N}d_{\Cdt}\left(\ov g^n(\ov x), \ov g^n\circ \ov\psi(\ov x)\right)$$
    where we used that $N$ is normal. But then by left-invariance of the word metric we get that
    $$\sup_{n\in\mathbb{Z},\,\ov\psi\in N}d_{\Cdt}\left(\ov g^n(\ov x), \ov g^n\circ \ov\psi(\ov x)\right)=\max_{\ov\psi\in N}d_{\Cdt}\left(\ov x,\ov\psi(\ov x)\right)\le M$$
    Now, let $\ov\phi\in N$ and let $\phi\in MCG^\pm$ be one of its preimages. Fix $n\in\mathbb{Z}$, let $l=[x,g^n(x)]$ be a geodesic and let $s=\phi(l)$, which is a geodesic between $\phi(x)$ and $\phi\circ g^n(x)$. Notice that $\sup_{s\in\C}d_s(\phi(x), \phi(g^n(x)))=\sup_{s\in\C}d_{\phi^{-1}(s)}(x,g^n(x))<\Theta$, therefore both $l$ and $s$ project isometrically to geodesics $\ov l=[\ov x, \ov g^n(\ov x)]$ and $\ov s=\ov\phi(\ov l)$, again by Lemma \ref{projinlink}. We can complete these two segments to a quadrilateral $\ov Q$ of vertices $\ov x, \ov \phi(\ov x), \ov \phi \circ \ov g^n(\ov x), \ov g^n(\ov x) $, by adding two geodesic segments of length at most $M$. By Lemma \ref{mgonlift} there exists a lift $Q$ of $\ov Q$, and by Lemma \ref{quadranglelift} the lifts $l',s'$ of $\ov l, \ov s$ are $DT_K$-translates of $l,s$ respectively. Up to elements of $DT_K$ we can assume that $l=l'$. Moreover, let $k\in DT_K$ be such that $k(s)=s'$. Setting $\psi=k\circ \phi$, which still induces $\ov\phi\in MCG^\pm/DT_K$, we see that the vertices of $Q$ are $x, \psi(x), \psi(g^n(x)), g^n(x)$. But since $Q$ lifts $\ov Q$ we have that $d_\C(x,\psi(x))=d_{\C/DT_K}(\ov x,\ov \phi(\ov x))\le M$, and similarly $d_\C(g^n x,\psi(g^n x))\le M$. Now, since $M$ is independent of $n$ we can choose $n$ big enough (that is, $x$ and $g^n(x)$ far enough on the axis) that $\psi$ must belong to $E_{MCG^\pm}(g)=\langle g\rangle$ because of \cite[Lemma 6.7]{DGO} which says, roughly, that coarsely stabilising a large segment of an axis is equivalent to stabilising the whole axis (as noted in \cite{DGO}, the lemma has the same proof as \cite[Proposition 6]{BF:WPD}, which has more restrictive hypotheses). Thus $\ov \phi=\ov g^m$ for some $m\in\mathbb{Z}$, and we must have that $m=0$ since $\ov\phi$ has finite order. 
\end{proof}

As a consequence of Theorem \ref{Out} we can also describe the automorphism group of $MCG/DT_K$:
\begin{theorem}\label{autmcg}
    For every $b\ge 7$ and for all large multiples $K$ the following hold:
    \begin{itemize}
    \item $\text{Aut}(MCG(S_b)/DT_K)\cong MCG^\pm(S_b)/DT_K$;
    \item $\text{Out}(MCG(S_b)/DT_K)\cong \mathbb{Z}/2\mathbb{Z}$.
    \end{itemize}
\end{theorem}
We will need this auxiliary lemma:
\begin{lemma}\label{centerless}
For every $b\ge7$ and for all large multiples $K$, $MCG(S_b)/DT_K$ has trivial centre.
\end{lemma}

\begin{proof}
Let $G=MCG^\pm(S_b)/DT_K$ and $H=MCG(S_b)/DT_K$. The centre $Z(H)$ is contained in $E_{G}(h)$ for every $h\in H$, by definition of $E_G(h)$. Then
$$Z(H)\le \bigcap_{h\in H\cap \mathcal{L}_{WPD}} E_G(h)=E_G(H)$$
and we know that $E_G(H)=\{1\}$ by Lemma \ref{reduction}.
\end{proof}
\begin{proof}[Proof of Theorem \ref{autmcg}]
    Again, let $G=MCG^\pm/DT_K$ and $H=MCG/DT_K$. Theorem \ref{Out} gives a surjective map $\Phi:\,G\to \text{Aut}(H)$ mapping an element $g\in G$ to the restriction of the conjugation by $g$. Notice that $\Phi$ is injective when restricted to $H$, since this group has trivial centre by Lemma \ref{centerless}. Thus $\ker \Phi \cap H=\{1\}$, which in turn means that $\ker \Phi$ injects in the quotient $G/H\cong \mathbb{Z}/2\mathbb{Z}$. Hence $\ker \Phi$ is a finite normal subgroup of $G$, and it must be trivial by Lemma \ref{nofinite}. This proves that $G\cong \text{Aut}(H)$, and since $H$ has trivial centre we also get
    $$\text{Out}(H)=\text{Aut}(H)/\text{Inn}(H)\cong G/H\cong \mathbb{Z}/2\mathbb{Z},$$
    as required.
\end{proof}

\appendix

\section{Quasi-isometric rigidity of pants graphs of spheres}\label{section:pantsqi}
This appendix, which is as self-contained as possible, contains a proof of the following result:

\begin{theorem}[Quasi-isometric rigidity of pants graphs of spheres]\label{qirigidpants}
Let $S_b$ be a punctured sphere, with $b\ge7$. Any self-quasi-isometry $f$ of the pants graph $\mathbb{P}(S_b)$ is at uniformly bounded distance from an element of the extended mapping class group.
\end{theorem}
This partly recovers a result of Bowditch \cite[Theorem 1.4]{BowPants} with our machinery, though the proof will not be completely new since we will still rely on some results from \cite[Sections 6 and 7]{BowPants} and on the main theorem of \cite{ssgraph}. Indeed, the "source" of rigidity here will be the fact that automorphisms of a certain graph, the graph of 1-separating curves, can only be extended mapping classes \cite{BowPants,ssgraph}. Hence, our goal is to show that automorphisms of the hinge graph induce automorphisms of this other graph, Corollary \ref{hingetoc1}, and in order to show this the key thing to do is roughly the following. Since hinges are not just subsurfaces, but rather pairs $(U,p)$ where $U$ is a subsurface and $p$ is a point in the boundary of its curve graph, we have to be able to combinatorially determine which hinges have the same support subsurface. In fact, we will only do so for "enough" hinges. 

\subsection{Unambiguous subsurfaces}\label{unambiguous}
Let $\mathbb{P}(S_{b})$ be the pants graph of $S_b$, whose HHS structure is given as follows (see \cite[Theorem G]{HHSI}):
\begin{itemize}
    \item $\mathfrak{S}$ is the set of essential, non-annular subsurfaces; 
    \item $\perp$ is disjointness and $\sqsubseteq$ is inclusion;
    \item for every $U\in \mathfrak{S}$, $\C U$ is the curve graph;
    \item Projections are defined using subsurface projections.
\end{itemize}

\begin{remark}\label{rem:asymph_pants}
    Notice that if $U\in \mathfrak{S}$ is a connected subsurface then its curve graph has infinite diameter (see e.g. \cite[Proposition 3.6]{MasurMinski1}). Thus if $U\in \mathfrak{S}$ has bounded curve graph then it is the disconnected union of subsurfaces of positive complexity, and its curve graph is a non-trivial join. Hence every bounded curve graph has diameter at most $2$, and therefore the HHS structure of the pants graph satisfies the mysterious asymphoricity assumption from Remark \ref{rem:asymporicity}.
\end{remark}

Bowditch \cite[Section 6]{BowPants} pointed out that if $U$ belongs to a complete support set then $U$ must have one of the following shapes, which ensure that every complete support set containing $U$ can cut out at most one pair of pants from $S$. Recall that a surface is \emph{odd} if its complexity is odd, otherwise it is \emph{even} (in our case, $S_b$ is odd iff $b$ is even). If $S$ is odd then $U$ is an $S_{4}$ and each component of the complement is odd and meets $U$ in exactly one curve. If $S$ is even then $U$ is either:
\begin{enumerate}
    \item an $S_{4}$ with all but one of the complementary components odd;
    \item an $S_{5}$ with all complementary components odd.
\end{enumerate}
Notice that the curve graphs of both $S_4$ and $S_5$ have at least four points at infinity, so we are always in the assumptions of Theorem \ref{hingesauto}. Thus, a quasi-isometry $f$ of $\mathbb{P}(S_b)$ induces an automorphism $f_{hin}$ of the hinge graph, and we want to show that $f_{hin}$ maps hinges with the same underlying subsurface to hinges with the same underlying subsurface, at least in the vast majority of cases.\\
Given a hinge $\sigma=(U,p)$ let $\text{Compl}(\sigma)$ be the set of all tuples of pairwise orthogonal hinges that $\sigma$ completes to a complete support set. In other words, $\text{Compl}(\sigma)$ is the set of all facets in $\bf{Hinge}(\mathfrak{S})$ that $\sigma$ completes to a maximal simplex.
\begin{definition}
    Two hinges $\sigma, \sigma'$ are said to be \emph{equally completable}, or to have the same completions, if $\text{Compl}(\sigma)=\text{Compl}(\sigma')$.
\end{definition}
This definition clearly induces an equivalence relation which is preserved by any automorphism of the hinge graph. One would like to think that if two hinges are equally completable then they have the same underlying subsurface. However this is not always true: in the even case, an $S_{4}$ whose even complementary component is a pair of pants $P$ has the same completions as the $S_{5}$ given by the union of $S_{4}$ and $P$. See Figure \ref{fig:s4s5_pants} to understand the situation. Luckily we will see that this is the only problem that could arise.

\begin{figure}[htp]
    \centering
    \includegraphics[width=0.75\textwidth]{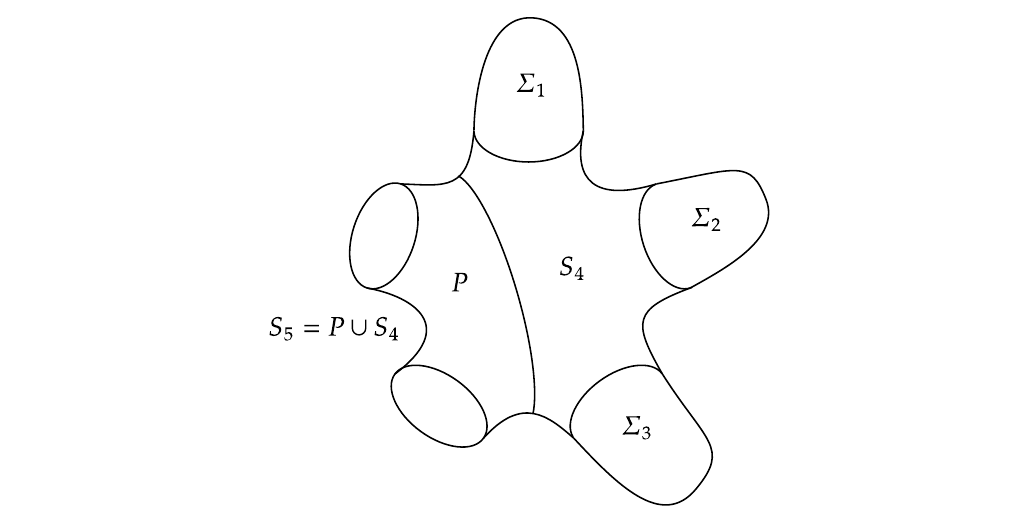}
    \caption{The $S_{4}$ and the $S_{5}$ in the Figure cannot be distinguished just by their completions, all of which must lie in the union of the $\Sigma_i$'s.}
    \label{fig:s4s5_pants}
\end{figure}

\begin{definition}
    A hinge $\sigma$ is \emph{unambiguous} if any other equally completable hinge $\sigma'$ has the same support. A support $U$ is unambiguous is every hinge supported in $U$ is unambiguous.
\end{definition}
We need another definition that characterises (some) unambiguous surfaces and can be recognised from the hinge graph.
\begin{definition}
    A hinge $\sigma$ is \emph{minimal} if $\text{Compl}(\sigma)$ is maximal by inclusion, among completions.
\end{definition}
\begin{lemma}\label{minambig}
    Let $b\ge7$. In the odd case, every hinge is minimal and unambiguous. In the even case, a hinge is minimal if and only if it is unambiguous and its support is a four-holed sphere.
\end{lemma}
\begin{proof}
    Let $\sigma=(U,p)$ be a hinge. We examine all possible shapes of $U$ to determine in which cases $\sigma$ is minimal and/or unambiguous. In the odd case $U$ is always an $S_{4}$ and each component of the complement is odd. Then we may find a completion $\{U_i\}$ whose union is $S\setminus U$: it suffices to choose a pants decomposition of every complementary component $\Sigma_i$, which must contain an even number of pairs of pants (since $\Sigma_i$ must be odd), and then match these pants in couples to get some four-holed spheres whose union covers $\Sigma_i$. Now $U$ must be unambiguous and minimal, since any other $V$ completing $\{U_i\}$ must be inside $S\setminus \bigcup U_i=U$, and therefore coincide with $U$ which has already minimal complexity. In the even case there are two possibilities:
    \begin{enumerate}
        \item Suppose $U$ is an $S_{5}$. Since $b\ge 7$ we know that $U$ does not coincide with $S$, so set $S\setminus U=\bigsqcup_{j=1}^k\Sigma_j$ with $k\le5$. Cover every $\Sigma_j$ with four-holed spheres. Moreover, choose a pair of pants $P\subset U$ whose boundary touches some four-holed sphere $W\subseteq\Sigma_1$, and let $U'=U\setminus P$. Then $U$ is not minimal: any completion of $U$ works also for $U'$, but if we replace $W$ with $W'=W\cup P$ we get a completion for $U'$ but not for $U$. See Figure \ref{fig:cutP} to understand the situation.
        \item Suppose $U$ is an $S_{4}$. Cover the odd complementary components with four-holed spheres. If the even component is not a pair of pants then we can cover it, and argue that $U$ must be minimal and unambiguous as in the odd case. Otherwise $U$ is ambiguous, since we are precisely in the case described in Figure \ref{fig:s4s5_pants}. This also means that $U$ has the same completions of a five-holed sphere, which is not minimal as we have already showed.
    \end{enumerate}
\end{proof}

\begin{figure}[htp]
    \centering
    \includegraphics[width=0.75\textwidth]{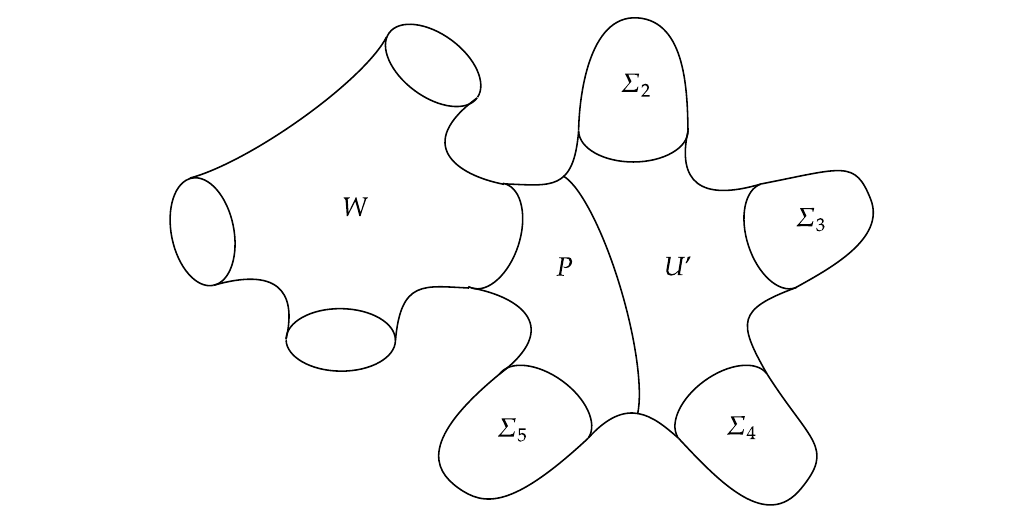}
    \caption{Replacing $W$ with $W'=W\cup P$ we get a completion for $U'$ but not for $U=U'\cup P$.}
    \label{fig:cutP}
\end{figure}

\begin{remark}\label{rem:auto_hinge_minimal_pants}
Any automorphism of the hinge graph must map minimal hinges with the same completions to minimal hinges with the same completions, therefore it acts on the set of minimal supports. Thus if $U$ is a minimal support we can define $f_{supp}(U)$ as the support of $f_{hin}(U,p)$ for any hinge $(U,p)$ supported in $U$.
\end{remark}

\begin{corollary}[Products go to products]\label{prodtoprod_pants}
Let $b\ge7$ and let $f$ be a self-quasi-isometry of the pants graph $\mathbb{P}(S_{b})$. There exists $C$, depending only on the quasi-isometry constants of $f$, with the following property. Let $\{U_i\}\subseteq\mathfrak{S}$ be a complete support set made of minimal supports, and let $f_{supp}(U_i)=V_i$ for all $i$. Let $P_{\{U_i\}}$ and $P_{\{V_i\}}$ be the standard product regions defined by $\{U_i\}$ and $\{V_i\}$, respectively. Then $d_{Haus}\left(f\left(P_{\{U_i\}}\right),  P_{\{V_i\}}\right)\le C$.
\end{corollary}

\begin{proof}
Another way of stating Remark \ref{rem:auto_hinge_minimal_pants} is that, if $(U,p^\pm)$ are two hinges with the same minimal support and $f_{hin}(U,p^+)=(V,q^+)$, then there exists $q^-\in\partial \C V$ such that $f_{hin}(U,p^-)=(V,q^-)$. Thus the Flats to Flats Lemma \ref{flatstoflats} says that, if $\{(U_i,p_i^\pm)\}$ is a complete support set made of minimal supports, with a choice of two points in every $\partial \C U_i$, and if we set $(V_i,q_i^\pm):=f_{hin}(U_i,p_i^\pm)$, then $d_{Haus}(f(\mathfrak F_{\{(U_i,p_i^\pm)\}}),  \mathfrak F_{\{(V_i,q_i^\pm)\}})\le C_0$ for some constant $C_0$ depending only on the quasi-isometry constants of $f$. Hence
$$d_{Haus}\left(f\left(\bigcup \mathfrak F_{\{(U_i,*)\}}\right),  \bigcup \mathfrak F_{\{(f_{supp}(U_i),*)\}}\right)\le C_0 $$
where $\bigcup \mathfrak F_{\{(U_i,*)\}}$ is the union of all standard flats supported in $\{U_i\}$. 

Now, the thesis follows if we prove the existence of some constant $C_1$ such that, for every complete support set $\{U_i\}$ made of minimal supports, 
$$d_{Haus}\left(\bigcup \mathfrak F_{\{(U_i,*)\}}, P_{\{U_i\}}\right)\le C_1$$ 
In fact, if this is the case then we also get that 
$$d_{Haus}\left(f\left(\bigcup \mathfrak F_{\{(U_i,*)\}}\right), f\left(P_{\{U_i\}}\right)\right)\le C_2$$
where $C_2$ is some constant depending only on $C_1$ and the quasi-isometry constants of $f$. In turn, since $P_{\{U_i\}}=\prod F_{U_i}$, to prove the existence of such $C_1$ it suffices to select a close enough hierarchy line in each coordinate, as we shall do in Lemma \ref{lem:unif_dist_from_hpath}.
\end{proof}

\begin{lemma}\label{lem:unif_dist_from_hpath}
There exists $C_3\ge0$ such that whenever $U\in\frakS$ is a minimal domain and $x\in F_U$, there exists a hierarchy line $\gamma\subset F_U$ such that $d_{F_U}(x,\gamma)\le C_3$.
\end{lemma}

\begin{proof}
    Since $P_{\{U_i\}}=\prod F_{U_i}$ it suffices to work in each coordinate, that is, we will show that there exists a constant $C_3$ such that whenever $U\in\frakS$ is a minimal domain and $x\in F_U$, there exists a hierarchy line $\gamma\subset F_U$ such that $d_{F_U}(x,\gamma)\le C_3$. Notice that, by the distance formula and the fact that $U$ is an $S_4$, and therefore a $\nest$-minimal support, we have that $F_U$ is uniformly quasi-isometric to $\C U$, and hierarchy lines simply correspond to bi-infinite quasigeodesics (with certain constants). Moreover $MCG(U)$ acts on $\C U$ coboundedly and by isometries, hence every point can be moved within uniformly bounded distance from a fixed quasigeodesic.
\end{proof}

\subsection{Automorphism of terminal subsurfaces}
\begin{definition}\label{1sep_pants}
A support $U$ is \emph{terminal} if it has complexity $1$ and it is cut out by a single curve. We say that the boundary curve of a terminal support is \emph{$1$-separating}, and we denote by $\Cp$ the full subgraph of $\C$ spanned by $1$-separating curves.
\end{definition}
\begin{lemma}
A hinge $\sigma=(U,p)$ has terminal support if and only if it is minimal and there exists a hinge $(V,q)$, compatible with $\sigma$, such that any complete support set containing $(V,q)$ must contain some $\sigma'$ which has the same completions as $\sigma$ (i.e., it has the same support). In particular, having terminal support is preserved by automorphisms of the hinge graph.
\end{lemma}
\begin{proof}
    This is just a restatement in our context of \cite[Lemma 6.3]{BowPants}. We just have to notice that a terminal hinge is minimal, which follows from the discussion of Lemma \ref{minambig}.
\end{proof}
\begin{corollary}\label{hingetoc1}
If $b\ge 7$, any automorphism of the hinge graph induces an automorphism of $\Cp$.
\end{corollary}
\begin{proof}
Any automorphism must map terminal supports to terminal supports and must preserve compatibility (which, for terminal subsurface, is equivalent to disjointness of their boundaries).
\end{proof}

\begin{proof}[Proof of Theorem \ref{qirigidpants}]
The previous discussion shows that $f$ induces an automorphism of $\Cp$, which is the restriction of some extended mapping class $g\in MCG^\pm$ when $b\ge 7$ (see \cite[Section 7]{BowPants} and \cite{ssgraph} for a proof, which is ultimately an application of Ivanov's Theorem). In other words, $f_{supp}$ and $g$ agree on terminal subsurfaces, and we want to show that they coincide on every minimal support $U$. We recall that, as showed in the proof of Lemma \ref{minambig}, every complementary components $\Sigma$ of $U$ has complexity at least $1$, i.e. $U$ does not cut out any pair of pants. Now, if $\Sigma$ has complexity $1$ then it is a terminal subsurface; otherwise $\Sigma$ has complexity at least $2$, and therefore there exist two terminal subsurfaces inside $\Sigma$ whose boundary curves fill $\Sigma$. Thus $U$ is the unique minimal support that is disjoint from all these terminal subsurfaces, and therefore $f_{supp}$ and $g$ must agree on $U$ since they both preserve disjointness.\\
By Corollary \ref{prodtoprod_pants} if $\{U_i\}$ is a complete support set made of minimal supports then $f$ maps the corresponding  product region $P_{\{U_i\}}$ within uniformly bounded Hausdorff distance from $g(P_{\{U_i\}})$, since they are both uniformly Hausdorff close to $P_{f_{supp}\{U_i\}}=P_{g\{U_i\}}$.\\
We are left to prove that every point $x\in\mathbb{P}(S_b)$ is the (uniform) coarse intersection of two standard product regions $P\Tilde{\cap} P'$, coming from minimal complete support sets. If this is the case then $f(x)$ will be the coarse intersection of $f(P)$ and $f(P')$, which lie at uniformly bounded distance from the coarse intersection of $g(P)$ and $g(P')$, which is coarsely $g(x)$. In order to prove this, it is enough to show that there is some point $x\in \mathbb{P}(S_b)$ which is the coarse intersection of two standard product regions with minimal supports, since the mapping class group acts cocompactly on $\mathbb{P}(S_b)$. In turn, by Corollary \ref{prodint} we are left to prove that there exist two complete support sets $\{U_i\}$ and $\{V_i\}$, with minimal supports, whose indices are pairwise distinct (notice that we can apply Corollary \ref{prodint} since minimal indices are $S_4$, and therefore they are also minimal with respect to nesting). If the surface is odd choose any complete support set $\{U_i\}$, which is already minimal. In the even case let $\Tilde{U}$ be an $S_5$ cut out by a single curve, and let $\{U_i\}$ be a support set that covers the complement of $\Tilde{U}$. Since $\{\Tilde{U}\}\cup\{U_i\}$ covers the surface, every $U_i$ must be minimal (this is the same argument as in the proof of Lemma \ref{minambig}). Now if we replace $\Tilde{U}$ with a terminal $S_4$ contained in $\Tilde{U}$, call it $U_1$, we get a complete support set $\{U_i\}$ made of minimal supports. In both cases, we can choose a pseudo-Anosov mapping class $\phi$ such that every boundary curve of $\{U_i\}$ crosses every boundary curve of $\{V_i\}=\phi\{U_i\}$. Therefore $U_i\neq V_j$ for every choice of $i$ and $j$.
\end{proof}

\bibliography{biblio}
\bibliographystyle{alpha}

\end{document}